\documentclass[english,12pt]{journal}
\usepackage{hyperref}

\usepackage[T1]{fontenc}
\usepackage[latin1]{inputenc}
\usepackage{geometry}
\usepackage{float}
\usepackage{mathrsfs}
\usepackage{amsmath}
\usepackage{amssymb}
\usepackage{graphicx}

\hypersetup{
    colorlinks=true,
    linkcolor=blue,
    filecolor=magenta,      
    urlcolor=cyan,
    citecolor = red,
}

\makeatletter

\usepackage{algorithm,algpseudocode}

\usepackage{amsthm}

\usepackage{mathrsfs}

\usepackage{amsfonts}

\usepackage{epsfig}

\usepackage{bm}

\usepackage{mathrsfs}

\usepackage{enumerate}

\@ifundefined{definecolor}{\@ifundefined{definecolor}
 {\@ifundefined{definecolor}
 {\usepackage{color}}{}
}{}
}{}
\newtheorem{ass}{Assumption}[section]

\usepackage{subfig}\usepackage[all]{xy}

\newcommand{\bbR}{\mathbb R}

\newtheorem{theorem}{Theorem}[section]
\newtheorem{lem}{Lemma}[section]
\newtheorem{rem}{Remark}[section]
\newtheorem{examp}{Example}[section]
\newtheorem{prop}{Proposition}[section]

\newcounter{hypA}

\newcounter{hypB}

\newcounter{hypD}

\newtheoremstyle{case}{}{}{}{}{}{:}{ }{}
\theoremstyle{case}

\usepackage{babel}

\makeatother

\newcommand{\bbE}{\mathbb{E}}

\newcommand{\bbZ}{\mathbb{Z}}
\newcommand{\bbN}{\mathbb{N}}

\newcommand{\cA}{\mathcal{A}}

\newcommand{\cQ}{\mathcal{Q}}

\newcommand{\cM}{\mathcal{M}}

\newcommand{\cG}{\mathcal{G}}
\newcommand{\sX}{\mathsf X}

\newcommand{\cI}{{\cal I}}

\newcommand{\ba}{\bm{\alpha}}
\newcommand{\bx}{\bm{x}}
\newcommand{\bZ}{\bm{Z}}

\newcommand{\cC}{\mathcal{C}}
\newcommand{\cO}{\mathcal{O}}
\newcommand{\cL}{\mathcal{L}}

\usepackage[normalem]{ulem}

\def\blu#1{{{#1}}}

\begin{document}
\pagenumbering{gobble}

%
%
%

\pagenumbering{arabic} 
\begin{center}

{\Large \textbf{Multi-index Sequential Monte Carlo ratio estimators for 
Bayesian Inverse problems}}

\vspace{0.5cm}

BY KODY J.~H. LAW$^{1}$, NEIL WALTON$^{1}$, SHANGDA YANG$^{1}$ \& AJAY JASRA$^{2}$

{\footnotesize $^{1}$School of Mathematics, University of Manchester, Manchester, M13 9PL, UK.}
{\footnotesize E-Mail:\,} \texttt{\emph{\footnotesize kody.law,
neil.walton,
shangda.yang@manchester.ac.uk}}\\
{\footnotesize $^{2}$Applied Mathematics and Computational Science Program; Computer, Electrical and Mathematical Sciences and Engineering Division, King Abdullah University of Science and Technology, Thuwal, 23955-6900, KSA.}
{\footnotesize E-Mail:\,} \texttt{\emph{\footnotesize ajay.jasra@kaust.edu.sa}}
\end{center}


\begin{abstract}
We consider the problem of estimating expectations with respect to a target
distribution with an unknown normalizing constant, and where even the unnormalized
target needs to be approximated at finite resolution.
This setting is ubiquitous across science and engineering applications, 
for example in the context of Bayesian inference where a physics-based model
governed by an intractable partial differential equation (PDE) appears in the likelihood.
A multi-index Sequential Monte Carlo (MISMC) method is used to construct ratio estimators 
which provably enjoy the complexity improvements of multi-index Monte Carlo (MIMC)
as well as the efficiency of Sequential Monte Carlo (SMC) for inference.
In particular, the proposed method provably achieves the canonical complexity
of MSE$^{-1}$, while single level methods require MSE$^{-\xi}$
for $\xi>1$. This is illustrated on examples of Bayesian inverse problems 
{with an elliptic PDE forward model in $1$ and $2$ spatial dimensions, where 
$\xi=5/4$ and $\xi=3/2$, respectively.
It is also illustrated on more challenging 
log Gaussian process models, 
where single level complexity is {approximately $\xi=9/4$} and 
multilevel Monte Carlo (or MIMC with an inappropriate index set) gives 
{$\xi = 5/4 + \omega$, for any $\omega > 0$}, 
whereas our method is again canonical.
\blu{We also provide novel theoretical verification of 
the product-form convergence results which MIMC requires
for Gaussian processes built in spaces of mixed regularity defined 
in the spectral domain, 
which facilitates acceleration with fast Fourier transform methods
via a cumulant embedding strategy, 
and may be of independent interest in the context of 
spatial statistics and machine learning.}} 
\end{abstract}

\section{Introduction}

There has been an explosion of work over the past decade involving 
the enormously successful multilevel Monte Carlo (MLMC) method \cite{giles2015multilevel} 
for estimating expectations with 
respect to distributions which need to be approximated.
The canonical example is the problem for forward uncertainty quantification 
(UQ), where a single realization of the random variable of interest requires 
the solution to a stochastic, ordinary, or partial differential equation (SDE, ODE or PDE) \cite{oksendal2003stochastic,boyce2017elementary,strauss2007partial}. 
The MLMC framework formulates this problem in terms of a sum of increments
corresponding to approximations at successive resolutions, or {\em levels}. 
Under a suitable coupling of the increments, 
which is typically fairly trivial in the forward context,
the variance of the increments decays as the resolution and cost increase,
and so progressively fewer samples are required to control the variance at higher levels.

In the context of Bayesian inference, one typically 
requires expectations with respect to target distributions
for which the normalizing constant is unknown. 
As an example, let $\pi$ denote a probability density on $\mathsf{X} \times \mathsf{Y}$. 
Assume we know how to evaluate $\pi(x,y)=\pi(y|x)\pi_0(x)$ and $\pi_0(x)$
but not $\pi(y) = \int_\mathsf{X} \pi(y|x) \pi_0(x) dx$.
Now consider the case where one observes $y \in \mathsf{Y}$
and would like to infer the {\em posterior distribution} $\pi(x|y)$, 
given by  
\begin{equation}\label{eq:bayesintro}
\pi(x | y) = \frac{\pi(y|x) \pi_0(x)}{\pi(y)} \, .
\end{equation}
This is referred to as the Bayesian framework, and $\pi(y|x)$ and $\pi_0(x)$
are referred to as the likelihood and the prior, respectively \cite{robert2013monte}.
Note that once the goal of \eqref{eq:bayesintro} is established, 
then a method should be capable of efficiently approximating integrals of the form:
$$
\frac1Z \int_\mathsf{X} \varphi(x) f(x) dx \, ,
$$
where $f(x) \propto \pi(y|x) \pi_0(x)$ and \blu{$Z=\int_\mathsf{X} f(x) dx$}
(i.e. $f(x)$ itself only needs to be proportional to the joint density).
Methods which have been designed for exactly this purpose include
Markov chain Monte Carlo (MCMC) \cite{geyer1992practical}, 
importance sampling \cite{robert2013monte},
and combinations thereof such as sequential Monte Carlo (SMC) samplers \cite{del2006sequential,chopin2020introduction}.
The latter methods are particularly powerful, handling elegantly some of the 
most challenging issues that arise in this context, 
such as small variance, strong dependence between variables, and multimodality.

Over the past decade the excitement about MLMC has intersected with the 
Bayesian computation community, in particular relating to the context of 
Bayesian inverse problems \cite{stuart2010inverse}, where an intractable PDE often 
appears inside the likelihood of the posterior distribution of interest.
 {For instance, we will later consider the case where the likelihood takes the from:
\begin{align*}
\pi(y| x) \propto e^{- \frac{1}{2} \| y - \mathcal G(x) \|^2} \, ,
\end{align*}
where $y$ is an observed set of real-valued outputs and $\cG(x)$ is a 
solution to the outputs from the intractable PDE for a given set of input parameters $x$.}
This context appears to be much more subtle, due to the complications
of combining these technologies.
Early work related to MCMC \cite{hoang2013complexity,dodwell2015hierarchical} 
and SMC samplers \cite{ourmlsmc,moral2017multilevel,beskos2018multilevel}.
More recently, the methodology has also been applied to the context of 
partially observed diffusions \cite{jasra2017multilevel,hoel2016multilevel}, 
for parameter inference \cite{ourac},
online state inference \cite{jasra2017multilevel,chernov2021multilevel,gregory2016multilevel,hoelmlenkf3,ballesio2020wasserstein,mlpflevy}, 
or both \cite{ourmismc2}.
A notable recent body of work relates to continuous-time 
observations in this context \cite{hengmlpf,beskos2021score,ruzayqat2021multilevel}. 
Another notable trend is the application of 
randomized MLMC methods 
\cite{vihola,rmlsmc,rmlpf,heng2021unbiased,rmlmcinf} in this context. 
Typically these methods require unbiased estimators of increments,
which is particularly challenging in the inference context.
{The first work to use randomized MLMC in the context of inference was \cite{vihola}, 
and unbiased increment estimators were available in the context of that work.}
Other more recent instalments have utilized double randomization
strategies in order remove the bias of increment estimators \cite{rmlsmc,rmlpf,heng2021unbiased}.  

The benefits of MLMC are somewhat hampered by the dimension 
of the underlying problem. This is an important issue, particularly in the context of a PDE or a SPDE.
For example, the error associated with a finite element method (FEM) approximation of a PDE
typically depends upon the mesh diameter, $h$, while the number of degrees of 
freedom typically scales like $h^{-d}$, where $d$ is the dimension of the associated PDE.
The multi-index Monte Carlo (MIMC) method was introduced to gracefully handle {the dimension 
dependence of this problem}
\cite{haji2016multi} following, in spirit, from the seminal work on sparse grids \cite{bungartz2004sparse}.
Instead of an estimator based on a sum of increments, the MIMC
method constructs an estimator based on a sum over an index set of $d-$fold 
composition of increments.
Under \blu{suitable regularity} conditions this approach is able to 
leverage convergence in each dimension 
independently and thereby mitigate the curse of dimensionality.

The MIMC method has very recently been applied to the inference context \cite{ourmismc1,ourmismc2,ourmimc},
however the estimates required for increments of increments has proven
challenging from a theoretical perspective, and this has severely limited progress thus far.
In particular, an MIMC method for inference with provable
convergence guarantees does not currently exist, 
except a ratio estimator using simple importance sampling, as considered for
MLMC and QMC in this work \cite{scheichl2017quasi}. 
Such estimators are not expected to be practical for complex target
distributions due to a large constant associated to importance sampling 
\cite{chatterjee2018sample, agapiou2017importance}.

The current work breaks down this theoretical barrier and unveils the MISMC sampler
ratio estimator for posterior inference. 
By employing a ratio estimator, we introduce a theoretically tractable method which provably achieves the benefits of
both SMC samplers for inference and MIMC for multi-dimensional discretization. 
In particular, rather than dealing with self-normalized increments of increments, 
as previous methods have done, the innovation is to construct instead a ratio of
MIMC estimators of an un-normalized integral and its normalizing constant, both
of which can be unbiasedly estimated with SMC sampler.
This seemingly minor difference of formulation 
substantially simplifies the analysis and enables us to establish 
a theory for the convergence of an MIMC method for inference problems --
a theory which until this point had been elusive. 

This article is structured as follows. In Section \ref{sec:motivate} we provide a class of motivating problems for the methodology that is developed.
In Section \ref{sec:comp} we provide a review of the relevant computational methodology that is used in our approach. In Section \ref{sec:MISMC}
we present our method and theoretical results. In Section \ref{ref:numerics} we present numerical results. Finally, in the appendix several technical results are
given, necessary for the theory that is presented in Section \ref{sec:MISMC}.

%

\section{Motivating Problems}\label{sec:motivate}

We consider the setting of Bayesian inference for an elliptic partial differential equation and for the Log Gaussian Cox model, where we must also perform numerical estimation. 

\subsection{Elliptic partial differential equation}
\label{sec:pde}
{
We consider 
the following elliptic PDE.
Consider a convex domain
$\Omega \subset \bbR^D$ with boundary
$\partial \Omega \in C^0$, a function (force vector field) $\mathsf{f}: \Omega \rightarrow \mathbb R$ and a function (permeabilitiy) $a(x):\Omega \rightarrow \mathbb R_+$ which is parameterized by $x\in \sX$. For each $x\in \sX$, we define the (pressure field) $u(x):\Omega \rightarrow \mathbb R$ as the solution to the following PDE on $\Omega$
\begin{align}
  -\nabla \cdot (a(x) \nabla u (x)) &= \mathsf{f}, \qquad\text{on }\Omega \ ,
  \label{eq:elliptic} \\
  u(x) &= 0, \qquad \text{on }  \partial \Omega \ .
  \label{eq:boundary}
\end{align}

In the above PDE, we assume the force vector field is known, e.g. $\mathsf{f}=1$. However, we assume the permeability
$a(x)$ depends upon a parameter $x$ which is a random variable, specifically $x \sim \pi_0$.
The dependence of $a$ on $x$ 
induces a dependence of the solution $u$ on $x$. 
Hence the solution itself, $u(x)(z)$, is a random variable for each $z \in \Omega$ .
}

For concreteness, assume that $D=2$ and $\Omega=[0,1]^2$.
Assume a uniform prior,
\begin{equation}
\label{eq:prior}
x \sim U(-1,1)^d  =: \pi_0 \ .
\end{equation}
For $x \sim \pi_0$, and $z \in \Omega$, the permeability will take the form
\begin{equation}\label{eq:diff}
a(x)(z) = a_0 + \sum_{i=1}^d x_i \psi_i(z) \, ,
\end{equation}
where $\psi_i$ are smooth functions with $\blu{\| \psi_i\|_{L^\infty({\Omega})}} \leq 1$ for $i=1,\dots,d$, 
and $a_0 > \sum_{i=1}^d x_i$.
In particular, for simplicity and concreteness, let $d=2$ and 
$$a(x)(z) = 3 + x_1 \cos(3\pi z_1)\sin(3 \pi z_2) + x_2 \cos(\pi z_1)\sin(\pi z_2).$$ 


\subsubsection{Finite element approximation and error estimates}
\label{sec:fem}

Consider the 1D piecewise linear nodal 
basis functions $\phi_j^K$ defined as follows,
for mesh $\{z^K_i = i / (K+1) \}_{i=0}^{K+1}$, and for $j=1,\dots,K$, 
\[
\phi^K_j(z) =
\begin{cases}
\frac{z-z^K_{j-1}}{z^K_j-z^K_{j-1}}, & z\in [z^K_{j-1},z^K_j] \\
1- \frac{z-z^K_{j}}{z^K_{j+1}-z^K_{j}}, & z\in [z^K_{j},z^K_{j+1}] \\
0, & {\rm else} \ .
\end{cases}
\]
Now, for $\alpha = (\alpha_1, \alpha_2) \in \mathbb N^2$, consider the tensor product grid over $\Omega=[0,1]^2$
formed by 
$$\{(z^{K_{1,\alpha}}_{i_1},z^{K_{2,\alpha}}_{i_2})\}_{i_1=0,i_2=0}^{K_{1,\alpha}+1,K_{2,\alpha}+1},$$ 
where $K_{1,\alpha}= 2^{\alpha_1}$ and $K_{2,\alpha}= 2^{\alpha_2}$
(and the mesh-width in each direction is bounded by $2^{-\alpha_k}$, $k=1,2$).
Let $i=i_1 + K_{1,\alpha} i_2$ for $i_1 =1,\dots, K_{1,\alpha}$ 
and $i_2=1,\dots, K_{2,\alpha}$
and $K_\alpha=K_{1,\alpha}K_{2,\alpha}$, and let 
$\phi^\alpha_i(z) = \phi^\alpha_{i_1,i_2}(z_1,z_2) = 
\phi^{K_{1,\alpha}}_{i_1}(z_1) \phi^{K_{2,\alpha}}_{i_2}(z_2)$ 
be piecewise bilinear functions.
The weak solution of the PDE
(\ref{eq:elliptic})--(\ref{eq:boundary}) 
will be approximated by
$u_\alpha(x) = \sum_{i=1}^{K_\alpha} u_\alpha^i(x) \phi^\alpha_i  \in V$.
Given $x$, the values of $u^i_{\alpha}(x)$ are defined by substituting the expansion into \eqref{eq:elliptic}
and taking inner product with $\phi^\alpha_j$ for $j=1,\dots, K_\alpha$.
In particular, observe that we have
$$
\left \langle -\nabla \cdot \left ( a(x) \nabla \sum_{i=1}^{K_\alpha} u_\alpha^i(x) \phi^\alpha_i \right ) , 
\phi_j^\alpha \right \rangle = \langle \mathsf{f} , \phi_j^\alpha  \rangle \ , \quad j=1,\dots, K_\alpha \ .
$$
Using integration by parts and observing that $\phi_i^\alpha|_{\partial \Omega} \equiv 0$, then
$$
\sum_{i=1}^{K_\alpha}  
\langle a(x)  u_\alpha^i(x) \nabla \phi^\alpha_i , \nabla \phi_j^\alpha  \rangle 
= \langle \mathsf{f} , \phi_j^\alpha  \rangle \ , \quad j=1,\dots, K_\alpha \ .
$$

We can represent the solution as a vector 
${\bf u}_\alpha(x) = [u_\alpha^i(x):i=1,\dots, K_\alpha]$, and
define ${\bf f}_{\alpha,j} = \langle \mathsf{f}, \phi^\alpha_j \rangle$ and
$$
{\bf A}_{\alpha,ij}(x) := \int_{z_{1,j_1-1}}^{z_{1,j_1+1}} \int_{z_{2,j_2-1}}^{z_{2,j_2+1}}
a(x)(z) \nabla\phi^\alpha_i(z) \cdot \nabla \phi^\alpha_j(z) dz ,
$$
where we introduce the notation $j:=j_1 + j_2 K_{1,\alpha}$ (for $j_1 =1,\dots, K_{1,\alpha}$ and $j_2=1,\dots, K_{2,\alpha}$).
Observe that if $i=i_1 + i_2 K_{1,\alpha}$, then the integral is zero for all $i$ such that
$i_k < j_k-1$ or $i_k > j_k + 1$, for $k \in \{1,2\}$.  
So the above matrix ${\bf A}_\alpha(x)$  is sparse, and it is straight-forward to verify that it is symmetric positive definite.

The approximate weak solution to equations \eqref{eq:elliptic}, \eqref{eq:boundary}
is given by the system $${\bf A}_\alpha(x) {\bf u}_\alpha(x) = {\bf f}_\alpha.$$
 Due to the sparsity of ${\bf A}_\alpha(x)$, for $D \leq 2$ 
 the solution can be obtained for a cost of roughly $\cO(K_\alpha)$
using an iterative solver based on Krylov subspaces, such as conjugate gradients 
\cite{nocedal2006numerical}. 
For $D \geq 3$ it may no longer be possible to achieve a linear cost
-- see e.g. \cite{haji2016multi}.
See the references \cite{ciarlet2002finite, brenner2007mathematical} 
for further description and much more.

The weak solution $u$ of (\ref{eq:elliptic})-(\ref{eq:boundary}) is said to be \blu{$W^{2,2}$} regular if there exists a $C>0$, such that 
$$\|\nabla^2u\|
\leq C\|\mathsf{f}\| 
$$ for every $\mathsf{f} \in L^2(\Omega),$
where $\|\cdot\|$ denotes the $L^2(\Omega)$ norm.
For the purposes of the present work, it suffices to observe the following proposition
\cite{braess2007finite,elman2014finite}.
{\begin{prop}\label{prop:standardpde}
For $a(x)$ given by \eqref{eq:diff} and uniformly over $x \in [-1,1]^d$, 
$\mathsf{f}\in L^2$ and $\Omega$ convex,
the weak solution of (\ref{eq:elliptic})-(\ref{eq:boundary}) is $W^{2,2}$
regular, and there exists a $C>0$ such that
$$
\| \nabla (u_\alpha(x) - u(x) ) \| \leq C 2^{-\min\{\alpha_1,\alpha_2\}} \, .
$$
Furthermore, 
$$
\| u_\alpha(x) - u(x) \| \leq C 2^{-2\min\{\alpha_1,\alpha_2\}} \, .
$$
\end{prop}}

\subsubsection{A Bayesian inverse problem}
\label{sec:bip}
{In the PDE \eqref{eq:elliptic}-\eqref{eq:boundary}, the parameter $x$ is unknown. 
Here we infer estimates about the true value $x$ from noisy observations of the solution to the PDE, $u(x)$. A further confounding factor is that the closed form solution to $u(x)$ is, in general, not known in closed-form and instead we must numerically approximate $u(x)$ with $u_{\alpha}(x)$ as described above.}

Now observations $y$ will be introduced and we will consider the inverse problem,
given by 
\begin{equation}
\label{eq:targetinf}
\pi(dx) := \pi(dx|y) \propto L(x) \pi_0(dx) \ , 
\end{equation}
where $L(x) \propto \pi(y|x)$ and the dependence upon $y$ is suppressed in the notation.
We will use the notations $d\pi(x)=\pi(dx)=\pi(x)dx$
to mean the same thing, i.e. probability under $\pi$ 
of an infinitesimal volume element $dx$ (Lebesgue measure by default) 
centered at $x$, 
and the argument may be omitted from $d\pi$ where the meaning is understood.

Define the following vector-valued function 
\begin{equation}\label{eq:obop}
  \mathcal{G}(u(x)) = [v_1(u(x)),\dots,v_{\blu{n}}(u(x))]^\top,
\end{equation}
where $v_i \in L^2$ and 
$v_i(u(x)) = \int v_i(z) u(x)(z) dz$ for $i = 1,\dots,\blu{n}$,
for some $\blu{n}\geq 1$. 
It is
assumed that the data take the form
\begin{equation}\label{eq:ob}
  y = \mathcal{G}(u(x) ) + \nu, \qquad \nu \sim N(0,\Xi), \qquad \nu \perp x \, ,\footnote{Here we use $\perp$ to denote pairwise independence of random variables.}
\end{equation}
and we define $$L(x) := \exp\Big( -\frac12 |y-\cG(u(x))|_\Xi^2 \Big) \, .
$$
Here $y$ is suppressed from the notation. Also we apply the convention that $|w|_\Xi := (w^\top \Xi^{-1} w)^{1/2} $.

In particular, $u(x)$ is the (weak) solution map
of (\ref{eq:elliptic})--(\ref{eq:boundary}), for given input $x$.
Denote its weak approximation at resolution multi-index $\alpha$ 
by $u_\alpha(x)$. 
The approximated likelihood is
given by $$L_\alpha(x) :=\exp( -\frac12 |y-\cG(u_{\alpha}(x))|_\Xi^2 ),$$
and the associated target is 
\begin{equation}
\label{eq:target}
\pi_\alpha(dx) \propto 
L_\alpha(x) \pi_0(dx) \, .
\end{equation}

{The following proposition summarizes the key result.
\begin{prop}\label{prop:bip}
In the present context, there is a $C>0$ such that $u,u_\alpha \leq C$, 
hence a $c>0$ such that $L,L_\alpha \geq c > 0$, and so \eqref{eq:targetinf} and
\eqref{eq:target} are well-defined. 
Furthermore, following Proposition \ref{prop:standardpde}
 and the continuity of $L$ as a function of $u$, the following 
rate estimate holds uniformly in $x$
$$
| L_\alpha(x) - L(x) | \leq C 2^{-2\min\{\alpha_1,\alpha_2\}} \, .
$$
\end{prop}}

For the concrete example of $D=2$,
let the observations be given by $v_i(u) := u(z_i)$, 
for $i=1,\dots,4$, 
where $z_i \in \{(0.25,0.25), (0.25,0.75), (0.75,0.75), (0.75,0.25)\}$, 
and let $\Xi = \xi^2 I$.
This example has been considered in the context of an 
MLSMC sampler method in \cite{ourmlsmc}. 
{It is noted that this example extends the theory described, 
since $v_i \notin L^2$.}

\subsection{Log Gaussian Process models}
\label{sec:lgc}

{Another model problem which will be considered is 
the log-Gaussian process (LGP), 
and the related log-Gaussian Cox (LGC) process,
which are commonly used in spatial statistics. 
In this example 
the dimension of the state space grows with level. 

Specifically we aim to model 
a dataset comprised of 
the location of $n=126$ Scots pine saplings in a natural forest in Finland \cite{cox},
denoted $z_1,\dots,z_{n} \in [0,1]^2$.
The LGC version of our model is based on the one presented in \cite{heng_bishop_2020}. 
The process of interest is defined as $\Lambda = \exp(x)$ 
where $x$ 
is a Gaussian process, a priori distributed in terms of a KL-expansion as follows,
for $z \in [0,2]^2$,
\begin{equation}\label{eq:GPKL}
x(z) = \theta_1 + 
\sum_{k \in \bbZ \times \bbZ_+ \cup  \bbZ_+ \times 0} 
\blu{\rho_k(\theta)} (\xi_k \phi_k(z) + \xi^*_k \phi_{-k}(z)) \, , \quad \xi_k \sim 
\mathcal{CN}(0,1) ~~ {\sf i.i.d.} \, ,
\end{equation}
where $\mathcal{CN}(0,1)$ denotes a standard complex Normal distribution, $\xi^*_k$ is the complex conjugate of $\xi_k$, and
$\phi_k(z) \propto \exp[\pi i z \cdot k]$ are Fourier series basis functions (with $i=\sqrt{-1}$),
and 
\begin{equation}\label{eq:prodspec}
{\blu{\rho_k^2(\theta)} = \theta_2/((\theta_3 + k_1^2)(\theta_3 + k_2^2))^{\frac{\beta+1}{2}} \, .}
\end{equation}
{The coefficient $\beta$ controls the smoothness of the Gaussian process.}
{The parameters $\theta$ will be assumed known in the present work,
but these can also be fit within a hierarchical modelling framework.}
The associated prior measure is denoted by $\mu_0$.
Following the formulation from \cite{heng_bishop_2020}, 
the likelihoods are defined by 
\begin{eqnarray}\label{eq:lgclike} 
{\sf(LGC)} \qquad \frac{d \pi}{d \pi_0}(x) 
&\propto& \exp \left [\sum_{j =1}^{n} x(z_j) - \int_{[0,1]^2} \exp(x(z)) dz \right] \, , \\
\label{eq:lgplike}
{\sf(LGP)} \qquad \frac{d \pi}{d \pi_0}(x) 
&\propto& \exp \left [\sum_{j =1}^{n} x(z_j) - n\log \int_{[0,1]^2} \exp(x(z)) dz \right]  \, .
\end{eqnarray} 
See e.g. \cite{tokdar2007posterior} for a description of the LGP version, which is given second above.
Note that only $z \in [0,1]^2$ is required. 
The periodic prior measure is defined on $[0,2]^2$ so that no boundary conditions 
are imposed on the sub-domain $[0,1]^2$ and 
the fast Fourier transform (FFT) can be used for approximation,
as described below.

The finite approximation is constructed as follows.
First the KL expansion \eqref{eq:GPKL} is truncated
\begin{equation}\label{eq:GPKLfin}
x_\alpha(z) = \theta_1 + 
\sum_{k \in \mathcal{A}_\alpha} 
{\rho_k^2(\theta)} (\xi_k \phi_k(z) + \xi^*_k \phi_{-k}(z)) \, , \quad \xi_k \sim 
\mathcal{CN}(0,1) ~~ {\sf i.i.d.} \, ,
\end{equation}
where $\mathcal{A}_\alpha := \{-2^{\alpha_1/2},\dots,2^{\alpha_1/2}\} 
\times \{1,\dots,2^{\alpha_2/2}\} \cup
\{1,\dots,2^{\alpha_2/2}\}\times 0$.
Note that $x_\alpha(z)$ can be approximated on a grid 
$\{0,2^{-\alpha_1}, \dots 1-2^{-\alpha_1}\}\times\{0,2^{-\alpha_2}, \dots 1-2^{-\alpha_2}\}$
using the FFT with a cost $\mathcal{O}((\alpha_1+\alpha_2) 2^{\alpha_1+\alpha_2})$.
Now $\hat{x}_\alpha(z)$ is defined as an interpolant (for example linear) 
over the grid output from FFT.
The finite approximation of the likelihood is then defined by 
\begin{eqnarray}\label{eq:lgclikefin} 
{\sf(LGC)} \qquad \frac{d \pi_{\alpha}}{d \pi_{0}}(x_\alpha) 
&\propto& \exp \left [\sum_{j =1}^{n} \hat{x}_\alpha(z_j) - 
 Q(\exp(x_\alpha)) \right ] \, , \\  \label{eq:lgplikefin} 
{\sf(LGP)} \qquad \frac{d \pi_{\alpha}}{d \pi_{0}}(x_\alpha) 
&\propto& \exp \left [\sum_{j =1}^{n} \hat{x}_\alpha(z_j) - 
n \log Q(\exp(x_\alpha)) \right ] \, ,
\end{eqnarray} 
where $Q$ denotes a quadrature rule, 
which may for example be given by  
$Q(\exp(x_\alpha))=2^{-(\alpha_1+\alpha_2)} 
\sum_{h \in \prod_{i=1}^2\{0,2^{-\alpha_i}, \dots, 
1-2^{-\alpha_i}\}} \exp(x_{\alpha}(h))$
or $Q(\exp(x_\alpha))=\int \hat{x}_\alpha(z) dz $.

If 
one uses the prior with isotropic spectrum
${\rho_k^2(\theta) = \theta_2/(\theta_3 + k_1^2+k_2^2)^{\frac{3}{2}}}$,
then our target measure
coincides with the standard prior 
of \cite{heng_bishop_2020} in the limit as $\min_i {\alpha_i} \rightarrow \infty$.
One can understand the connection in this context via 
the circulant embedding method based on FFT \cite{lord}.  
However, previous work has employed the (dense) 
kernel representation of the covariance function
instead of diagonalizing it with FFT. 
For our product-form spectrum, the regularity would be matched
for $\beta=1$, corresponding to a product of Ornstein-Uhlenbeck processes.
Instead we will choose $\beta=1.6$ 
for convenience, which means that our prior is slightly smoother.

}

{
\subsubsection{LGP and LGC theoretical results}
\label{ssec:lgctheory}

First we state a simple convergence result for Gaussian process 
of the form \eqref{eq:GPKL} with spectral decay 
corresponding to \eqref{eq:prodspec}.

\begin{figure}[H]
    \centering
    \includegraphics[width=.6\linewidth]{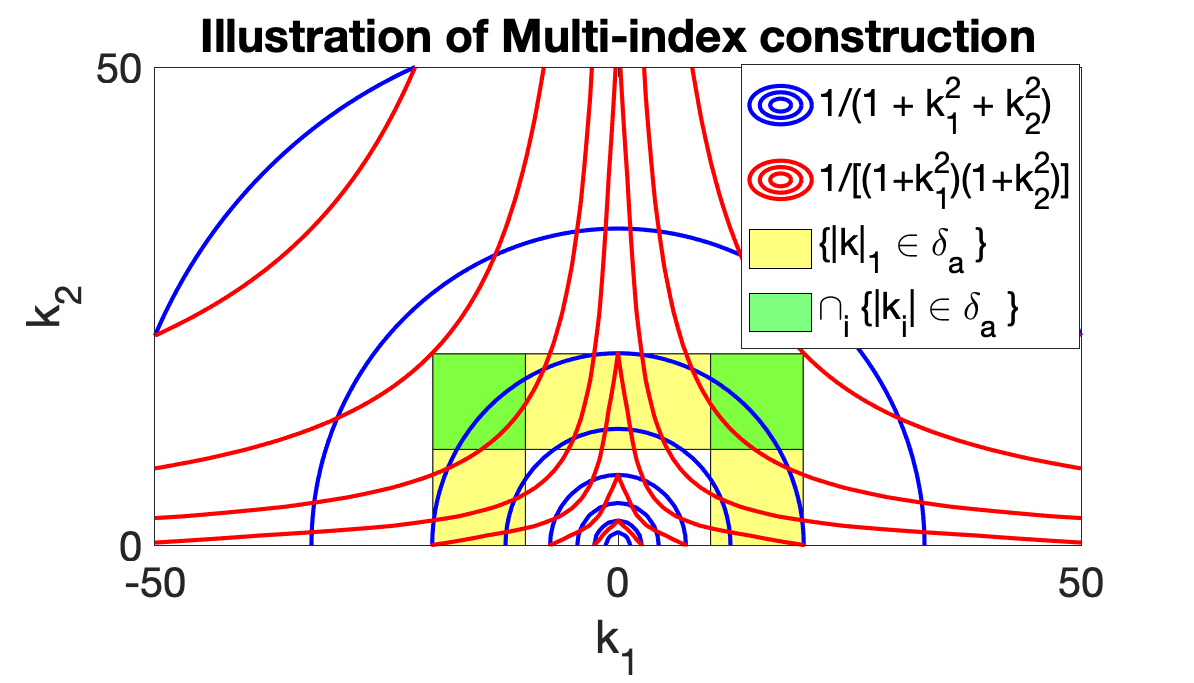}
    \caption{\blu{A cartoon of variance contours associated to a function in 
    $H_{1/2}^{\sf m}$ (red) and a function in $W^{1/2,2}$ (blue). 
    Letting 
    $\delta_a = [2^{a/2},2^{a}]$,
    the spectral region associated to an increment of approximations 
    on the index sets defined in \eqref{eq:GPKLfin}, 
    $\{ |k|_1 \in \delta_a \} =
    \{ k \in \cA_\alpha\} \cap 
    \{k\notin \cA_{\alpha/2}\}$ (with $\alpha=(a,a)$),
    is depicted in yellow. Its intersection 
    with the region associated to an {\em increment 
    of increments}, 
    $\cap_{i=1}^2\{ {|k_i|} \in \delta_a\}$, is depicted in green.}}
    \label{fig:mixed}
\end{figure}

It will be useful to define the following operator 
{$A$ on the space  
of functions in $L^2(\Omega)$:}
\begin{equation}\label{eq:mixedlap}
A=\sum_{k \in \bbZ^2} a_k \phi_k \otimes \phi_k \, 
, \qquad a_k = \blu{(1+k_1^2)(1+k_2^2)} \, ,
\end{equation}
the mixed Sobolev-like norms
\begin{equation}\label{eq:soblike}
\|x\|_q := \| A^{q/2} x \| \, ,
\end{equation}
where \blu{ $A^{q/2} = \sum_{k \in \bbZ^2} a_k^{q/2} \phi_k \otimes \phi_k$ and} we recall $\|\cdot\|$ is the $L^2(\Omega)$ norm,
and the spaces 
\begin{equation}\label{eq:soblike}
H_q^{\blu{\sf m}} := 
\{x \in L^2(\Omega); \|x\|_q < \infty \}\, .
\end{equation}
\blu{We note that these spaces ensure 
{\em mixed regularity} (hence superscript ${\sf m}$), 
rather than the typical regularity associated to standard Sobolev 
spaces. It is precisely this property which
multi-index methods are designed to exploit.
Figure \ref{fig:mixed} shows the contours of
functions in $H_{1/2}^{\sf m}$ (red) and  $W^{1/2,2}$ (blue), along with the regions associated to 
an increment (yellow) and increment of an increment at approximation level 
{$\alpha=(a,a)$}.
From inspection, it is clear
that increments of increments are higher order 
in comparison to increments for functions in the mixed space, 
but not for functions in the standard space.}

The following proposition is proven in Appendix \ref{app:lgc}.
 \begin{prop}\label{prop:prodspecconv}
 Let $x \sim \pi_0$, where $\pi_0$ is a Gaussian process 
 of the form \eqref{eq:GPKL} with spectral decay 
corresponding to \eqref{eq:prodspec}, 
and let $x_\alpha$ correspond to truncation on the index set 
$\cA_\alpha$ as in \eqref{eq:GPKLfin}. 
Then $x \in H_q^{\sf m}$ for all $q<\beta/2$, and for $r \in [0,q)$
there is a $C>0$ such that 
$$
\| x_\alpha - x \|_r^2 \leq C \|x\|_q^2 2^{-2(q-r)\min_i \alpha_i} \, .
$$
 \end{prop}

For $\beta=1$, \eqref{eq:prodspec} looks like a product of OU processes,
with the regularity of \blu{Wiener} measure \cite{pavliotis}. Hence,
following from Proposition \ref{prop:prodspecconv}, $x$ is almost surely continuous for
$\beta \geq 1$, and \eqref{eq:lgclike} and \eqref{eq:lgplike} are well-defined.
{It is worth noting that \eqref{eq:lgclike} and \eqref{eq:lgplike}
are not well-defined for the standard prior, e.g. from \cite{heng_bishop_2020},
since the Sobolev Embedding Theorem 
(see e.g. \cite{stuart2010inverse}) 
does not guarantee that the solution is almost surely continuous.
However, non-infinitesimal representations of \eqref{eq:lgclike}, i.e.
for finite partitions of the domain, can still be computed as long as $\Lambda$ is integrable.}

\blu{The following proposition ensures our LGC and LGP posterior
measures are well-defined on function space and has a density with respect to the prior.}
It is proven in Appendix \ref{app:lgc}.

{\begin{prop}\label{prop:lgcpost}
	Given $x: \Omega \rightarrow \mathbb R$ is a Gaussian process, with probability measure denoted $\pi_0$, defined on compact finite dimensional space $\Omega$, that is almost surely continuous and has a finite mean and covariance. If we define $\pi$ by
\begin{eqnarray*}{\sf(LGC)} \qquad \frac{d \pi}{d \pi_0}(x) 
&\propto& \exp \left [\sum_{j =1}^{n} x(z_j) - \int_{\Omega} \exp(x(z)) dz \right] \, , \\
{\sf(LGP)} \qquad \frac{d \pi}{d \pi_0}(x) 
&\propto& \exp \left [\sum_{j =1}^{n} x(z_j) - n\log \int_{\Omega} \exp(x(z)) dz \right]  \, ,
\end{eqnarray*} 
for $n \in \mathbb N$ then $\pi(dx)$ is a well-defined probability measure,
and can be represented in terms of its density with respect to $\pi_0$:
$$
\pi(dx) = \frac{d\pi}{d\pi_0} \pi_0(dx) \, .
$$ 
\end{prop}
}

The analogue to Proposition \ref{prop:bip} takes the following form.
The proposition is again reproduced, and proven, in Appendix \ref{app:lgc}.

\begin{prop}\label{prop:biplgc}
For both LGP and LGC, 
there is a $C>0$ such that for $x \sim \pi_0$
and $x_\alpha = P_{\cA_\alpha}x$\blu{, where  $P_{\cA_\alpha}$ denotes the orthogonal 
projection onto the index set $\cA_\alpha$ defined in \eqref{eq:GPKLfin},}
the following rate estimate holds for all $q<(\beta-1)/2$ 
$$
\bbE | L_\alpha(x_\alpha) - L(x) |^2 \leq C 2^{-2\min\{q,1\} \min\{\alpha_1,\alpha_2\}} \, .
$$
\end{prop}
}

\section{Computational Methodology}\label{sec:comp}

\blu{\subsection{Approximate Monte Carlo}
\label{sec:MC}
For concreteness, 
in this subsection we will consider the case of the PDE example from 
subsection \ref{sec:pde}.
The case of subsection \ref{sec:lgc} follows similarly.}
Let $\sX := [-1,1]^d$ be the domain of $x$.

\subsubsection{Monte Carlo}
\label{ssec:MC}

The forward uncertainty quantification (UQ) problem is the following.
Given a quantity of interest 
$\varphi_\alpha = \varphi \circ u_\alpha :  \sX \rightarrow \bbR$, 
compute its expectation
$$
\bbE \varphi_\alpha (x) = \int_{\sX} \varphi( u_\alpha(x) ) \pi_0(dx).
$$
The typical strategy is to independently sample $x^{i} \sim \pi_0$, for $i = 1, \dots, N$, and then approximate 
$$
\bbE \varphi_\alpha (x) \approx \frac1N \sum_{i=1}^N \varphi( u_\alpha(x^i) ).
$$
For example, we can let 
$$
\varphi_\alpha(x) = \| u_\alpha(x) \|^2 = \int_{\Omega} |u_\alpha(x)(z)|^2 dz 
\approx \sum_{i=1}^{K_{\alpha_1}} \sum_{j=1}^{K_{\alpha_2}} u^i_\alpha(x) u^j_\alpha(x) \int_{\Omega} \phi_i^\alpha(z) \phi_j^\alpha(z) dz \, ,
$$
\blu{where the latter can be written as} ${\bf u}_\alpha^T {\bf K}_\alpha {\bf u}_\alpha$, 
where ${\bf K}_{\alpha,ij} := \langle \phi_i^\alpha, \phi_j^\alpha \rangle$.

\subsubsection{Multi-index Monte Carlo}
\label{sec:MIMC}

With MIMC \cite{haji2016multi}, we apply the approximation 
\begin{equation}\label{eq:mimc}
\bbE \varphi( x ) \approx \sum_{\alpha \in \cI}\bbE [ \Delta \varphi_\alpha( x ) ] \ ,
\end{equation}
where the difference of differences operator is defined as 
$\Delta \varphi_\alpha := \Delta_D \circ \cdots \circ \Delta_1 \varphi_\alpha$
with $\Delta_i \varphi_\alpha := \varphi_\alpha - \varphi_{\alpha - e_i}$, 
$e_i$ is the $i^{\rm th}$ standard basis vector in $\bbR^D$, 
and $\varphi_{\alpha} \equiv 0$ if $\alpha_i<0$ for any $i$.
The task is then to approximate 
the expectation of the increment of increments 
for each 
$\alpha \in \cI \subset \bbZ^D$.
For example, for $D=2$, one must approximate
\[\begin{split}
 \bbE [ \Delta \varphi_\alpha( x ) ] = &
\int_{[-1,1]^2} \Big ( [\varphi( u_\alpha(x)) - \varphi(u_{\alpha_1,\alpha_2-1}(x))] \\
& -  [\varphi( u_{\alpha_1-1,\alpha_2}(x)) - \varphi(u_{\alpha_1-1,\alpha_2-1}(x))] \Big ) \pi_0(dx) \ .
\end{split}\]
To do this we sample $x^i_\alpha \sim \pi_0$, i.i.d. for $i = 1, \dots, N_\alpha$, 
and then approximate 
\[\begin{split}
\bbE [ \Delta  \varphi_\alpha( x ) ] \approx  
\bbE_\alpha^{N_\alpha} [ \Delta \varphi_\alpha( x ) ] & 
:= \frac1{N_\alpha}\sum_{i=1}^{N_\alpha}  \Delta  \varphi_\alpha( x^i_\alpha ) 
\end{split}\]
Observe that $\bbE [ \bbE_\alpha^{N_\alpha} [ \Delta \varphi_\alpha( x ) ] ] = \bbE [ \Delta  \varphi_\alpha( x ) ]$.
Furthermore, based on approximation properties of $u_\alpha$, one expects a $C>0$ such that
\begin{equation}\label{eq:goodcoupling}
\bbE \left [ (\bbE_\alpha^{N_\alpha} [ \Delta \varphi_\alpha( X ) ] - \bbE [ \Delta  \varphi_\alpha( X ) ] )^2 \right ] 
\leq \frac{C}{N_\alpha} \prod_{i=1}^D 2^{-\beta_i \alpha_i} \, .
\end{equation}
For the example in subsection \ref{sec:fem} we have $\beta_1=\beta_2=4$
\cite{haji2016multi}. 

In particular, as we will now describe, we know how to choose the index set $\cI$ 
and schedule of $\{N_\alpha\}_{\alpha \in \cI}$ 
such that the following estimator delivers the same MSE
for significantly smaller cost than the standard method of 
subsection \ref{sec:MC} 
$$
\bbE \varphi( x ) \approx \bbE_\cI^{\rm MI} \varphi( x ) := 
\sum_{\alpha \in \cI}  \bbE_\alpha^{N_\alpha} [ \Delta \varphi_\alpha( x ) ] \ ,
$$
where $\bbE_\alpha^{N_\alpha}$ indicates that $N_\alpha$ independent samples 
are used at each level $\alpha$.
A concise, but not comprehensive, summary of the approach is given in the  review \cite{giles2015multilevel}. For a detailed treatment see \cite{haji2016multi}.
MLMC corresponds to the case in which there is one index. 
The MLMC methodology is more generally applicable to problems in which the target 
distribution -- in this case the pushforward of $\pi_0$ through $u$, 
$(u)_{\#}\pi_0$,
i.e. the distribution of $u(x)$ for $x\sim \pi_0$ -- 
needs to be approximated first to finite resolution, 
$\alpha$, before Monte Carlo can be applied.


\begin{ass}\label{ass:MIMC}
There exist positive constants $s_i$, $\beta_i$, $\gamma_i$ and $C$ for $i = 1,2,...,D$, such that the following holds
\begin{itemize}
    \item[(a)] $|\bbE[\Delta\varphi_\alpha( x )]| \leq C 2^{-\langle \alpha, s \rangle}$;
    \item[(b)] $\bbE \left [ (\bbE_\alpha^{N_\alpha} [ \Delta \varphi_\alpha( X ) ] - \bbE [ \Delta  \varphi_\alpha( X ) ] )^2 \right ] \leq C N_{\alpha}^{-1}
    2^{-\langle \alpha, \beta \rangle}$;
    \item[(c)] ${\rm COST}(\Delta\varphi_\alpha( x )) \leq C 
    2^{\langle \alpha, \gamma \rangle}$.
\end{itemize}
\end{ass}
\blu{For a random variable $X$, 
the cost function ${\rm COST}(X)$ denotes the computational complexity of a single sample of $X$.}{The following two propositions are standard results for MIMC and are proven in \cite{haji2016multi}.} 


\begin{prop}\label{prop:tdMIMC}
Assume Assumption \ref{ass:MIMC}, 
with $\beta_i > \gamma_i$, for all $i = 1,...,D$.
Then for the total degree index set $\mathcal{I}_{L}:=\{\alpha \in \bbN^{D}:\sum_{i=1}^{D}\delta_{i}\alpha_{i} \leq L, \sum_{i=1}^{D}\delta_{i} = 1\}$, there are values of $L \in \bbN$, $\delta_{i} \in (0,1]$ and $\{N_\alpha\}_{\alpha \in \mathcal{I}_{L}}$ such that
\begin{equation}
    \bbE\left[\left(\sum_{\alpha\in\mathcal{I}_{L}}\bbE_\alpha^{N_\alpha} [ \Delta \varphi_\alpha( X ) ] - \bbE [\varphi( X )]\right)^2\right] < C\varepsilon^2,
\end{equation}
with a computational complexity of $\cO(\varepsilon^{-2})$ for any small $\varepsilon > 0$.
\end{prop}

\blu{\begin{rem}
\label{prop:ftMIMC}
Under the same assumptions as in Proposition \ref{prop:tdMIMC},
if the index set is replaced with
the tensor product index set 
$\mathcal{I}_{L_1:L_d}:=\{\alpha \in \bbN^{D}:\alpha_1 \in \{0,...,L_1\},...,\alpha_D \in \{0,...,L_D\}\}$,
then the same complexity result can be obtained
only with an {\bf additional constraint} that
$\sum_{j=1}^{D}\gamma_j/ s_j \leq 2$.
\end{rem}}

\blu{
\subsection{Monte Carlo for Inference}}
\label{ssec:mcinference}

\blu{For simplicity, in this subsection, 
we define the algorithm for the target 
$\pi$, although we note that in practice this cannot be implemented for finite cost for our targets, and it must be replaced with $\pi_\alpha$.
This sets the stage for our method, which combines the inference 
approach with the approximation approach 
described in subsection \ref{sec:MC}.}

\vspace{10pt}

\subsubsection{Markov chain Monte Carlo and Importance Sampling}
\label{sssec:mcmc}

In the context of Bayesian inference, the objective is ultimately to compute expectations
with respect to a probability distribution $\pi$ proportional to 
$f>0$, where one can evaluate $f$ but not its integral, denoted 
by \blu{$Z = \int f(dx)$, so $\pi(dx) = f(dx)/Z$. 
In particular, we define $f(dx):= L(x) \pi_0(dx)$} as the target, 
in the limit $\alpha \rightarrow \infty$ of \eqref{eq:target}.
That is, for arbitrary
$\varphi:\sX \rightarrow \bbR$, 
we want to compute integrals
\begin{equation}
\label{eq:expect}
\pi(\varphi) := \int_{\sX} \varphi(x) \pi(dx) = \frac1{Z} \int_{\sX} \varphi(x) \blu{f(dx)} = 
\frac{f(\varphi)}{Z} \, .
\end{equation}
If we could simulate from $\pi$, we would approximate this by
\begin{equation}
\label{eq:iid}
\bbE^N(\varphi):= \frac1{N}\sum_{i=1}^{N} \varphi(x^{(i)}), \quad x^{(i)}\sim \pi \, .
\end{equation}
However, in the present context this is not possible because the normalizing constant $Z$ is typically unknown and must be calculated numerically.
Markov chain Monte Carlo (MCMC)
and (self-normalized) importance sampling are the standard methods 
to solve \blu{such problems} \cite{robert2013monte}.  
Both methods provide estimators  $\widehat{\varphi}^{N}$
with a dimension-independent convergence rate
analogous to $\bbE^N(\varphi)$, 
for some $C_\varphi > 0$:
$$
\| \widehat{\varphi}^{N} 
- \pi(\varphi) \|_{p}^2 \leq \frac{C_\varphi}N
$$
For MCMC, $C_\varphi$ typically depends at worst polynomially 
on $d$, and can sometimes be made independent \cite{pcn, nealpcn}. 
However, due to its intrinsic locality, MCMC is doomed to fail for 
distributions which are concentrated around several modes with low probability in between. 
In the case of importance sampling, the latter case is handled gracefully, 
however one must be careful since often
$C_\varphi = \cO(e^d)$ \cite{bengtsson2008curse, chatterjee2018sample, agapiou2017importance}.
To be precise, estimating $\int \varphi d\pi$ using samples from $\overline{\pi}$ results in
$C_\varphi = \cO( \exp(D_{KL}(\pi_\varphi \| \overline{\pi} )))$, where 
$$\pi_\varphi = \frac1{\int\varphi d\pi } \varphi \pi \, ,$$
and $D_{KL}( \nu \| \mu)$ is the Kullback-Leibler divergence from $\mu$ to $\nu$ 
\cite{chatterjee2018sample}.


If one can simulate from some $\overline{\pi}$ such that 
\blu{$\overline{\pi}(dx)=\frac1{\overline{Z}}\overline{f}(dx)$ with $\overline{Z} = \int \overline{f}(dx)$,
$\overline{f}(dx) = \overline{L}(x) \pi_0(dx)$,
 and $L/\overline{L} \leq C$}, 
then importance sampling consists of replacing the 
above unbiased approximation by the following
biased but consistent one
\blu{\begin{equation}
\label{eq:is}
\frac{\sum_{i=1}^{N_s} \varphi(x^{(i)}) 
\frac{L(x^{(i)})}{\overline{L}(x^{(i)})}}
{\sum_{i=1}^{N_s} \frac{L(x^{(i)})}{\overline{L}(x^{(i)})}}, \quad x^{(i)}\sim \overline{\pi}.
\end{equation}}

MCMC methods instead proceed by constructing a Markov chain 
$\cM: \sX \times \sigma(\sX) \rightarrow [0,1]$, 
where $\sigma(\sX)$ is the sigma algebra of measurable 
sets, such that for all $A \in \sigma(\sX)$
$$
(\pi \cM)(A) := \int_{\sX} \pi(dx') \cM(x',A) = \pi(A) \, , 
$$
i.e. the Markov chain is $\pi-$invariant.
Provided the Markov chain is also ergodic then 
one may simulate a trajectory 
and approximate \eqref{eq:expect} 
by
\begin{equation}
\label{eq:iid}
\frac1{N_s}\sum_{i=1+N_b}^{N_s+N_b} \varphi(x^{(i)}), \quad x^{(i)}\sim \cM^{(i)}(x^{(0)},\cdot) \,.
\end{equation}
Here, as above, \blu{$N_s$} is the number of samples used, while $N_b$ is the number of initial samples that are unused because we must first allow our Markov chain to approach stationarity.

The most popular MCMC method is Metropolis Hastings (MH),
which proceeds by designing a proposal Markov kernel $\cQ$
such that the following composition Markov kernel is ergodic.
\blu{
First, sample 
$x' \sim \cQ(x^{(i)}, dx') = q(x^{(i)}, x') \pi_0(dx')$, 
then let $x^{(i+1)} = x'$ with probability
\begin{equation}\label{MH}
	\min\left \{1, 
\frac{L(x')q(x', x^{(i)}) }{L(x^{(i)})q(x^{(i)}, x' )} 
\right\} \, . 
\end{equation}}
Otherwise, let $x^{(i+1)} = x^{(i)}$. 
Notice that again, as in \eqref{eq:is}, 
only the un-normalized target density $L$ is required.
\blu{
Note that in order to customize the presentation to the context
at hand, we presented a particular category of MH methods,
designed for probability measures on general state
spaces, which have 
densities with respect to 
$\pi_0$.
Such methods are justified by the framework of
\cite{tierney}, 
and a particularly convenient instantiation arises 
for Gaussian process priors $\pi_0$,
where it is easy to define $\cQ$ such that
$q(x,x')=q(x',x)$ for all $x,x' \in \sX$. 
See \cite{nealpcn,pcn},
and the more recent 
slice sampler variant \cite{murray2010elliptical}.}
 

Sequential Monte Carlo samplers 
combine these 2 fundamental algorithms -- 
importance sampling, and propagation by MCMC
kernels -- 
along a sequence of intermediate targets, and are able to achieve 
some very impressive results.
The next subsection introduces this technology.

\subsubsection{Sequential Monte Carlo samplers}
\label{sec:smc}

Sequential Monte Carlo (SMC) samplers are able to merge the best of both worlds, by repeatedly leveraging 
 importance sampling on a sequence of target distributions which
are close. 
In particular, define
$h_1, \dots, h_{J-1}$ by \blu{$h_j = L_{j+1}/L_j$, 
where $L_1 = \overline{L}$, $L_J = L$, $\overline{L}$ appears in \eqref{eq:is} 
(and may likely be $\pi_0$), 
$f(dx)=L(x)\pi_0(dx)$ is the un-normalized target, 
and for $j=2,\dots, J-1$, $L_i$ \blu{interpolates} in between.}

Let $\pi_j = f_j/Z_j$, where \blu{$Z_j = \int f_j(dx)$ and 
$f_j(dx) = L_j(x) \pi_0(dx)$.
Note that $\overline{f}(dx) \prod_{i=1}^{J-1} h_j(x) = f(dx)$}. 
The idea of SMC is to simulate from $\overline{\pi}=\pi_1$ and use these
samples to construct a self-normalized importance sampling estimator of $f_2$ with weights $h_1$
as in \eqref{eq:is}, and iterate for $j=1,\dots, J-1$, 
resulting in a self-normalized importance sampling estimator of $\pi$.
There is however an obvious issue with this idea. 
{In particular, 
the locations of the sampled points \blu{remain} unchanged over each stage of the algorithm for this sequential importance sampling 
estimator. This leads to degeneracy that is no better than the original (one step) importance sampling estimator \eqref{eq:is}.}

The key idea introduced in \cite{jarzynski1997nonequilibrium, neal2001annealed, chopin2002sequential, del2006sequential}
is to use Markov transition kernels between successive target distributions $\pi_j$ and $\pi_{j+1}$
in order to decorrelate (or ``jitter'') the particles, 
while preserving the intermediate target.
The standard approach is to let $\cM_j$ for $j=2,\dots,J$
be such that $(\pi_j \cM_j)(dx) = \int \pi_j(dx') \cM_j(x',dx) = \pi_j(dx)$, 
e.g. an MCMC kernel of the type introduced in the previous 
subsection, \eqref{MH}. 

The resulting SMC algorithm is given in Algorithm \ref{algo:smc}.
Define 
\begin{equation}\label{eq:empirical}
\pi_j^N(\varphi) := \frac1{N} \sum_{i=1}^N \varphi(x_j^{(i)}) \, .
\end{equation}

\begin{algorithm}[h]
\caption{SMC sampler} 
Let $x_1^{(i)} \sim \pi_1$ for $i=1,\dots, N$, and $Z^N_1=1$.
For $j=2,\dots,J$, repeat the following steps: 
\begin{itemize}
\item[(0)] Store \blu{$Z_j^N = Z_{j-1}^N \frac{1}{N}\sum_{k=1}^N h_{j-1}(x_{j-1}^{(k)})$.}
\item[(i)] Define 
$w_j^i = h_{j-1}(x_{j-1}^{(i)}) / \sum_{k=1}^N h_{j-1}(x_{j-1}^{(k)})$, {for $i=1,\dots, N$}.
\item[(ii)] Resample. Select $I_j^i \sim \{ w_j^1, \dots, w_j^N\}$, and let $\hat x_j^{(i)} = x_{j-1}^{(I_j^i)}$, {for $i=1,\dots, N$}.
\item[(iii)] Mutate. Draw $x_j^{(i)} \sim \cM_j(\hat x_j^{(i)},\cdot)$, {for $i=1,\dots, N$}.
\end{itemize}\label{algo:smc}
\end{algorithm}

\blu{In the resampling step of Algorithm \ref{algo:smc}, the samples are resampled according to their weights, so that some ``unfit'' (low weight) 
particles will ``die'' whilst other ``fit'' (high weight) ones will ``multiply''. 
As such, it can be viewed as a sort of genetic selection mechanism \cite{del2004feynman}.
One can understand this operation as preserving the distribution
of particles as well as the degeneracy of the sample, 
while exchanging variance of weights for redundancy of particles.
Therefore, at a given instance, there is no net gain, 
however future generations will have replenished diversity.
As an example, one can use multinomial resampling. 
See \cite{chopin2020introduction} for details.}



\subsubsection{Estimating the normalizing constant with SMC}
\label{sec:noco}

Recall \blu{$Z_j = \int f_j(dx)$, and observe that 
$$
\pi_j(h_j) = \frac{1}{Z_j} \int \frac{L_{j+1}(x)}{L_{j}(x)} f_{j}(dx)  = \frac{Z_{j+1}}{Z_j} \, .
$$ }
It follows that
the ratio of normalizing constants of $\pi_J=\pi$ to $\pi_1=\overline{\pi}$, $Z/\overline{Z}$,
is given by 
$$
\frac{Z_J}{Z_1} = \prod_{j=1}^{J-1} \pi_{j}(h_{j})\, .
$$
If $Z_1=1$ is known, then this is simply equal to $Z$, 
the normalizing constant of $\pi$.

Observe that using Algorithm \ref{algo:smc} 
we can construct an estimator of each factor by 
$$
\pi^N_{j}(h_{j}) = \frac1N \sum_{i=1}^N h_j(x_j^{(i)}) \, . 
$$
Recall that for any $\varphi:\sX \rightarrow \bbR$ we have  
\blu{$f_J(\varphi) := \int \varphi(x) f_J(dx) = 
f_J(1) \pi_J(\varphi)$, by definition.}
Now define the following estimator
\begin{equation}\label{eq:unbiased}
 f_J^N(\varphi) := \prod_{j=1}^{J-1} \pi^N_{j}(h_{j}) \pi_J^N(\varphi) 
 = Z_J^N \pi_J^N(\varphi) \, . 
\end{equation}
where \blu{$\pi_j^N(\varphi)$} and $Z_J^N$ are as defined in Algorithm \ref{algo:smc}.

Note that by definition
$$
\pi_J^N(\varphi)  = \frac{f_J^N(\varphi)}{Z_J^N} = \frac{f_J^N(\varphi)}{f_J^N(1)} \, .
$$




\section{Multi-index sequential Monte Carlo}
\label{sec:MISMC}

{With the necessary notation and concepts defined in the previous section, we now establish our theoretical results for Multi-Index Sequential Monte-Carlo. 
Through this we can provide theoretical guarantees for the Bayesian inverse problems, 
such as those defined in 
subsection \ref{sec:bip} and we develop methods which  
apply the MIMC methodology of subsection \ref{sec:MIMC} to that problem.}

{The main result is an estimator which retains the well-known 
efficiency of SMC samplers while provably achieving the complexity 
benefits of MIMC. \blu{This problem has been considered
before in \cite{ourmimc, ourmismc1, ourmismc2}, but the present work is the first to establish convergence guarantees under 
reasonable verifiable assumptions.} To this end,
our objective is to apply SMC samplers
to estimate \eqref{eq:expect} while utilizing a multi-index decomposition of the form
\eqref{eq:mimc}. }

{After formulating our problem and introducing some additional notation, we present and prove our main convergence result Theorem \ref{thm:main}.}

\subsection{Formulation}

For convenience we denote the vector of multi-indices 
{$$\bm{\alpha}(\alpha) := (\ba_1(\alpha),\dots, \ba_{2^D}(\alpha)) \in \bbZ_+^{D\times2^D} \, ,$$}
where $\ba_1(\alpha)=\alpha$, $\ba_{2^D}(\alpha)=\alpha - \sum_{i=1}^D e_i$, 
and $\ba_i(\alpha)$ for $1<i<2^D$ correspond to the intermediate multi-indices
involved in computing $\Delta \varphi_\alpha$, as described above \eqref{eq:expect}. 
We note that when $\alpha$ is on the boundary of $\bbZ^D_+$ then
several of the terms involved in $\Delta \varphi_\alpha$ are 0, 
but we find this notation more expedient
than letting $\bm{\alpha}(\alpha) \in \bbZ_+^{D\times k_\alpha}$
where $k_\alpha=2^{\#\{i; \alpha_i\neq 0\}} \in \{0,2,\dots, 2^D\}$ 
adjusts the dimension $k_\alpha$ when $\alpha$ is on the boundary of the index set.

Define 
\blu{$f_\alpha(dx) := L_\alpha(x) \pi_0(dx)$, 
$Z_\alpha := \int_\sX f_\alpha(dx)$ and 
$\pi_{\alpha}(dx) = f_\alpha(dx) / Z_{\alpha}$}, following from \eqref{eq:target}.
There are 2 fundamental strategies one may adopt for estimating
$\pi(\varphi) = f(\varphi)/f(1)$ using a multi-telescoping identity as in \eqref{eq:mimc}. 
{The first considers the following representation
\begin{equation}\label{eq:sn_increment}
\pi(\varphi) = \sum_{\alpha\in \bbZ_+^D} \Delta( \pi_\alpha(\varphi_\alpha) ) = 
\sum_{\alpha\in \bbZ_+^D} \Delta\left ( \frac1{Z_\alpha} f_\alpha(\varphi_\alpha) \right) \, .
\end{equation}}
{Note we allow $\varphi_\alpha$ to depend on $\alpha$ -- 
for example it could involve the solution to the PDE.}

{Directly estimating $\Delta (\pi_\alpha(\varphi_\alpha))$} 
would be quite natural if we were able to sample from a coupling of
$(\pi_{\ba_1(\alpha)},\dots, \pi_{\ba_{2^D}(\alpha)})$
i.e. a distribution $\Pi_\alpha: \sigma(\sX^{2^D}) \rightarrow [0,1]$ such that 
$$
\blu{\int_{\bx_{-j} \in \sX^{2^D-1}} \Pi_\alpha(d\bx)  = 
\pi_{\ba_j(\alpha)}(d\bx_j)} \, ,\qquad\text{for }j=1,\dots, 2^D. \footnote{Here $\bx_{-j}$ omits the $j^{\rm th}$ coordinate from $\bx = (\bx_1,\dots, \bx_{2^D}) \in \sX^{2^D}$} $$ 
In practice, however, this is non-trivial to achieve. 
One successful strategy for MLMC methods
is to construct instead an approximate coupling $\Pi_\alpha$ 
such that $\pi_{\ba_i(\alpha)}/\Pi_\alpha$ is bounded for all $i=1,\dots, 2^D$,
then simulate from this and construct self-normalized importance sampling 
estimators of the type \eqref{eq:is} for each of the individual summands of 
$\Delta\left ( \frac1{Z_\alpha} f_\alpha(\varphi_\alpha) \right)$
appearing in \eqref{eq:sn_increment}.
This strategy was introduced for MLMCMC in \cite{ourac} and
has subsequently been applied to MIMC in the contexts of 
MCMC \cite{ourmimc} and SMC \cite{ourmismc1, ourmismc2}.
These MIMC works lack rigorous convergence results,
due to the challenge of achieving rigorous rates for the individual summands,
as well as the effect of cumbersome off-diagonal terms in the MSE estimates
arising from bias of the summands (which are higher-order for MLMC). 
Both of these issues are handled elegantly with the present method.

In the present work, we adopt the second fundamental strategy, which is 
to use the ratio decomposition
\begin{equation}\label{eq:increments_ratio}
{\pi(\varphi) = \frac{f(\varphi)}{f(1)} = 
\frac{\sum_{\alpha\in \bbZ_+^D} \Delta( f_\alpha(\varphi_\alpha) )} 
{\sum_{\alpha\in \bbZ_+^D} \Delta (f_\alpha(1))} \, .}
\end{equation}
In their limiting forms in \eqref{eq:sn_increment} and \eqref{eq:increments_ratio},
the expressions are equivalent, however from an approximation perspective they
are fundamentally different.
In the context of SMC, there are advantages to the latter. 
In particular, this alleviates both of the issues with arising 
{from bias of the summands} in the aforementioned approach,
which have prevented rigorous convergence results until now.
A similar strategy was used to construct randomized MLMC estimators for 
Bayesian parameter estimation with particle MCMC in \cite{vihola}.
This method comprises the main result of this work,
and its development is the topic of the following subsection.

\subsection{Main result}
\label{sec:main}

\blu{In order to make use of \eqref{eq:increments_ratio} we need to construct estimators of 
$\Delta( f_\alpha(\zeta_\alpha) )$, both for our quantity of interest 
$\zeta_\alpha = \varphi_\alpha$ and for $\zeta_\alpha=1$.
To that end we shall construct a coupling which approximates 
$\Pi_\alpha$, and has well-behaved importance weights 
with respect to $\Pi_\alpha$.}
Let 
\begin{equation}\label{eq:priorcoupling}
\Pi_0(d\bx) = \pi_0(d\bx_1) \prod_{i=2}^{2^D} \delta_{\bx_1}(d \bx_i) \, ,
\end{equation}
where $\delta_{\bx_1}$ denotes the Dirac delta function at $\bx_1$.
Note that this is an exact coupling of the prior in the sense that 
for any $j \in \{1,\dots, 2^D\}$
\blu{\begin{equation}\label{eq:priormarg}
\int_{\bx_{-j} \in \sX^{2^D-1}} 
\Pi_0(d \bx) = \pi_{0}(d\bx_j) \, .
\end{equation} }
Indeed it is the same coupling used in subsection \ref{sec:MIMC}.
It is hoped that this coupling of the prior will carry over to provide error estimates analogous to \eqref{eq:goodcoupling}
for the estimator \eqref{eq:increments_ratio}, when computed using SMC.
{We note that one can estimate \eqref{eq:increments_ratio} directly by 
importance sampling with respect to the prior, as 
described in subsection \ref{sec:MIMC}, however this is not expected to be as efficient as using SMC.} 
We hence adapt Algorithm \ref{algo:smc} to an extended target which is an approximate coupling 
of the actual target as in \cite{ourac, ourmimc, ourmismc1, ourmismc2, vihola}, 
and utilize a ratio of estimates analogous to \eqref{eq:unbiased}, similar to what was done in 
\cite{vihola}.
To this end, we define a likelihood on the coupled space as
\begin{equation}\label{eq:coupledlike}
{\bf L}_\alpha(\bx) = \max \{ 
L_{\ba_1(\alpha)}(\bx_1), \dots, L_{\ba_{k_\alpha}(\alpha)}(\bx_{k_\alpha}) \}  \, .
\end{equation}
{Note that $k_\alpha=2^{\#\{i; \alpha_i\neq 0\}} \in \{0,2,\dots, 2^D\} \leq 2^D$ 
adjusts the effective dimension of the target when $\alpha$ 
is on the boundary of the index set.
The approximate coupling is defined by
\blu{\begin{equation}\label{eq:approxcoupling}
F_\alpha(d\bx) = {\bf L}_\alpha(\bx) \Pi_0(d\bx) \, , \qquad 
\Pi_\alpha(d\bx) = \frac1{F_\alpha(1)} F_\alpha(d\bx) \, .
\end{equation}}

\begin{examp}[Approximate Coupling]
As an example of the approximate coupling constructed in equations
\eqref{eq:priorcoupling}, \eqref{eq:coupledlike}, and 
\eqref{eq:approxcoupling}, let $D=2$, $d=1$, and $\alpha=(1,1)$.
Then we have
$$
\Pi_{(1,1)}(x_1,x_2,x_3,x_4) \propto 
\max\{L_{11}(x_4), L_{10}(x_3), L_{01}(x_2), L_{00}(x_1)\} 
\pi_0(x_1) \delta_{x_1}(x_2) \delta_{x_1}(x_3) \delta_{x_1}(x_4) \, .
$$
Note that for our choice of prior coupling \eqref{eq:priorcoupling},
we effectively have a single
distribution 
$$
\Pi_{(1,1)}(x) \propto 
\max\{L_{11}(x), L_{10}(x), L_{01}(x), L_{00}(x)\} 
\pi_0(x) \, ,
$$
but 
any coupling of the prior which preserves the marginals as in 
\eqref{eq:priormarg} is admissible,
so we prefer to consider this as a target on the ``diagonal hyperplane'' 
$x_1=x_2=x_3=x_4$, as above.
\end{examp}}

Let $H_{\alpha,j} = {\bf L}_{\alpha, j+1}/{\bf L}_{\alpha,j}$ for some intermediate distributions
$F_{\alpha,1}, \dots, F_{\alpha,J}=F_{\alpha}$.
In our case, we use the natural intermediate targets
$F_{\alpha,j}(d\bx) = {\bf L}_\alpha(\bx)^{\tau_j} \Pi_0(d\bx)$,
{where $\tau_1=0$, $\tau_j<\tau_{j+1}$, and $\tau_J=1$}.
 For $j=1,\dots,J$, we define 
\[
\Pi_{\alpha,j}(d\bx) = \frac1{F_{\alpha,j}(1)} F_{\alpha,j}(d\bx) 
\]
and 
we let $\bm{\cM}_{\alpha,j}$ be a Markov transition kernel 
 such that $(\Pi_{\alpha,j} \bm{\cM}_{\alpha,j})(d\bx) = \Pi_{\alpha,j}(d\bx)$, 
analogous to $\cM$ in subsection \ref{sec:smc}. 
\blu{
Any MCMC kernel  
as described in subsection \ref{sssec:mcmc}
with target distribution $\Pi_{\alpha,j}$
is suitable for this purpose. 
An example 
is the Metropolis-Hastings kernel described above and in \eqref{MH}.}

\blu{
\begin{rem}[Tempering]
\label{rem:temp}
Tempering accurately is crucial, because 
if the effective sample size drops too low, then the population 
will lack sufficient diversity to survive. 
The purpose of the sequential resampling and mutation is precisely
to preserve diversity in the sample.
Sometimes a fixed tempering schedule is suitable for this purpose, 
for example $\tau_j = (j-1)/(J-1)$. 
An alternative is to use an adaptive tempering strategy. 
Given 
a (possibly un-normalized) 
weighted sample $\{w^{(k)},\bx^{(k)}\}_{k=1}^N$,
the {\em effective sample size} (${\rm ESS}$) is 
defined as follows
\begin{equation*}
   {\rm ESS} = \frac{\left(\sum_{k=1}^{N}w^{(k)} \right)^2}{\sum_{k=1}^{N}(w^{(k)})^{2}} 
\end{equation*}
This quantity serves as a proxy for the variance of the weighted sample.
To understand the name, note that if $w^{(k)}\propto 1$ for all $k$, 
then ${\rm ESS}=N$, while if $w^{(k^*)}\propto 1$ for some $k^*$ and 
$w^{(k)}=0$ for $k \neq k^*$ then ${\rm ESS}=1$.
If $\tau_j = \tau_{j-1} + h$, for $h > 0$, then the intermediate weights will be 
$w^{(k)} = {\bold{ L}}_\alpha(\bx^{(k)})^{h}$,
and the corresponding ${\rm ESS}(h)$ is a scalar function of $h$ 
which quantifies the sample attrition that results from the 
importance sampling step; 
precisely what we are aiming to control.
The adaptive tempering parameter $\tau_j$ 
is computed by firstly solving  ${\rm ESS}(h)={\rm ESS_{min}}$
 with a pre-specified value of ${\rm ESS_{min}}$,
and then letting $\tau_j \leftarrow \tau_{j-1} + h$. 
In this way, the effective sample size is ${\rm ESS_{min}}$ 
each time importance sampling is carried out. 
The tempering procedure is carried until $\tau_j = 1$.
\end{rem}}

\blu{
\begin{rem}[Role of Dimension $D$]
\label{rem:dim}
Note that in high-dimensions 
one would select an index set in which there are few (or no) 
terms on the interior. 
In the present work, we 
do not explicitly consider the dependence on $D$ 
(which is reasonable for small $D\leq 5$ say),
however the methodology is applicable for high-dimensional targets and 
that is the subject of future work. 
The cost of simulating the approximate coupling at 
level $\alpha$ will feature a constant $2^D$
multiplying Assumption \ref{ass:rate}(C), 
because that is how many likelihood evaluations are required
to compute \eqref{eq:coupledlike},
and hence corresponding multi-increment.
The constant can be large, but this will not alter the complexity estimates.
\end{rem}}

\begin{algorithm}[h!]
\caption{SMC sampler for coupled estimation of $\Delta( f_\alpha(\zeta_\alpha) )$}
Let $\bx_1^{(i)} \sim \pi_1$ for $i=1,\dots, N$, $\bZ^N_1=1$, and $\omega_{1,k}=1$ for $k=1,\dots, 2^D$.
For $j=2,\dots,J$, $k=1,\dots, 2^D$, repeat the following steps for $i=1,\dots, N$:
\begin{itemize}
\item[(0)] Store $\bZ_j^N = \bZ_{j-1}^N \frac1N \sum_{k=1}^N H_{\alpha,j-1}(\bx_{j-1}^{(k)})$.
\item[(i)] Define 
$w_j^i = H_{\alpha,j-1}(\bx_{j-1}^{(i)}) / \sum_{k=1}^N H_{\alpha,j-1}(\bx_{j-1}^{(k)})$.
\item[(ii)] Resample. Select $I_j^i \sim \{ w_j^1, \dots, w_j^N\}$, and let $\hat \bx_j^{(i)} = \bx_{j-1}^{(I_j^i)}$.
\item[(iii)] Mutate. Draw $\bx_j^{(i)} \sim \blu{\bm{\cM}_{\alpha,j}(\hat \bx_j^{(i)},\cdot)}$.
\end{itemize}\label{algo:smccoup}
\end{algorithm}

For $j=1,\dots, J$, and for random variables $\bx_j^{(i)}$, $i=1,...,N$ (which will be sampled from the Markov chain $\blu{\bm{\cM}_{\alpha,j}}$) we
 define 
\begin{equation}\label{eq:nocoup}
\Pi^N_{\alpha,j}(d\bx) := \frac1N \sum_{i=1}^N \delta_{\bx_j^{(i)}}(d\bx) \, , 
\end{equation}
and then define
\begin{equation}\label{eq:unnocoup}
\bZ_{\alpha}^N := \prod_{j=1}^{J-1} \Pi^N_{\alpha,j}( H_{\alpha,j} ) \, , \quad  
F^N_\alpha(d\bx) := \bZ_\alpha^N \Pi^N_{\alpha,J}(d\bx) \, .
\end{equation}

We require the following assumption
\begin{ass} \label{ass:G1}
Let $J\in\mathbb{N}$ be given\blu{, and let $\sX$ be a Banach space}. For each $j\in\{1,\dots,J\}$ there exists some $C>0$ such that for all 
$(\alpha,\bx)\in\blu{\mathbb{Z}_+^D\times \sX^{2^D}}$ 
$$
C^{-1} < Z, \blu{H_{\alpha,j}(\bx), {\bf L}_\alpha(\bx)} \leq C \, .
$$
\end{ass}

\blu{The following proposition can easily be deduced from [Theorem 7.4.2, \cite{del2004feynman}].
\begin{prop}\label{prop:unnocoupled}
Under Assumption \ref{ass:G1} 
we have 
$
\bbE [F_{\alpha}^N(\psi)] = F_\alpha(\psi)$.
\end{prop}}

Define 
\begin{equation}\label{eq:psidef}
\psi_{\zeta_\alpha}(\bx) := \sum_{k=1}^{2^D} \iota_k \omega_k(\bx) \zeta_{\ba_k(\alpha)}(\bx_k) \, , \qquad 
\omega_k(\bx) := \frac{L_{\ba_k(\alpha)}\left (\bx_{k} \right)}
{{\bf L}_{\alpha}(\bx)} \, ,
\end{equation}
where $\iota_k \in \{-1,1\}$ is the sign of the $k^{\rm th}$ term in 
$\Delta f_\alpha$ \blu{and $\zeta_\alpha: \sX \rightarrow \bbR$.} 
Following from Proposition \ref{prop:unnocoupled} we have that
\begin{equation}\label{eq:coupledbias}
\bbE[F_\alpha^N(\psi_{\zeta_\alpha})] = F_\alpha(\psi_{\zeta_\alpha}) = \Delta f_\alpha(\zeta_\alpha) \, .
\end{equation}

Now given $\cI \subseteq \bbZ_+^D$, $\{N_\alpha\}_{\alpha \in \cI}$, 
and $\varphi :\sX \rightarrow \bbR$,
for each $\alpha \in \cI$, 
run an independent SMC sampler as in Algorithm \ref{algo:smccoup} 
with $N_\alpha$ samples, \blu{define $Z_\alpha^N = Z_J^N$},  
and define the MIMC estimator as 
\begin{equation}\label{eq:mismc}
\widehat\varphi^{\rm MI}_\cI = \frac{\sum_{\alpha \in \cI} F^{N_\alpha}_\alpha(\psi_{\varphi_\alpha})}
{\max\{\sum_{\alpha \in \cI} F^{N_\alpha}_\alpha(\psi_{1}), Z_{\min}\}} \, ,
\end{equation}
where $Z_{\rm min}$ is a lower bound on $Z$
as given in Assumption \ref{ass:G1}, 
\blu{and $F^{N_\alpha}_\alpha$ is defined in \eqref{eq:unnocoup}}.


\subsubsection{Theoretical results}
\label{eq:theory}

Throughout this subsection $C>0$ is a constant whose value may change from line to line. 
The following Theorem is the main theoretical result which underpins the results to follow.

\blu{\begin{theorem}\label{thm:main}
Assume Assumption \ref{ass:G1}. 
Then for any $J\in\mathbb{N}$ there exists a $C>0$ such that for any $N\in\mathbb{N}$, $\psi : \sX^{2^D} \rightarrow \bbR$ bounded and
measurable and $\alpha\in\mathbb{Z}_+^D$
$$
\bbE\left[| F_\alpha^N(\psi) - F_\alpha(\psi) |^2\right] \leq \frac{C}{N} F_\alpha(\psi^2).
$$
Furthermore, 
$$
F_\alpha(\psi_{\zeta_\alpha}^2) \leq 
C \int_\sX (\Delta( L_\alpha(x)\zeta_\alpha(x) ))^2 \pi_0(dx) \, ,
$$
where $\psi_{\zeta_\alpha}(\bx)$ is as \eqref{eq:psidef}.
\end{theorem}}
\begin{proof} 
The first 
result follows from Lemmas \ref{lem:mrt}, \ref{lem:qp} and \ref{lem:induc},
given in Section \ref{app:a}. 
The second result is derived as follows
\begin{eqnarray}\nonumber
F_\alpha(\psi_{\zeta_\alpha}^2) &=& \int_{\sX^{2^D}} \left ( \sum_{k=1}^{2^D} \iota_k \frac{L_{\ba_k(\alpha)}(\bx_{k} ) }
{{\bf L}_{\alpha}(\bx)} \zeta_{\ba_k(\alpha)}(\bx_k)  \right)^2
{\bf L}_\alpha(\bx) \Pi_0(d\bx) \\  \nonumber
&=& \int_{\sX^{2^D}} \frac1{{\bf L}_{\alpha}(\bx)} \left ( \sum_{k=1}^{2^D} \iota_k 
L_{\ba_k(\alpha)}(\bx_{k} ) \zeta_{\ba_k(\alpha)}(\bx_k) \right)^2 \Pi_0(d\bx) \\ \nonumber
&\leq & C \int_{\sX} (\Delta ( L_\alpha(x) \zeta_\alpha(x) ) )^2 \pi_0(dx) \, .
\end{eqnarray}
The first 2 lines are direct substitution and the inequality follows by defining  
$C^{-1} = \inf_{\bx \in \sX^{2^D}} {\bf L}_\alpha(\bx)$ and using the definition of $\Pi_0$ in \eqref{eq:priorcoupling}.
\end{proof}


%

Following from above, the assumptions below will be made.

\begin{ass}\label{ass:rate}
For any $\zeta: \sX \rightarrow \bbR$ bounded and Lipschitz,
there exist $C, \beta_i, s_i, \gamma_i >0$ for $i=1,\dots, D$ such that for 
resolution 
vector
$(2^{-\alpha_1},\dots, 2^{-\alpha_D})$, i.e. resolution 
$2^{-\alpha_i}$ in the
$i^{\rm th}$ direction, the following holds
\begin{itemize}
\item[(B)] \blu{$|\Delta f_\alpha(\zeta)|$} 
{$=: B_\alpha$} $\leq C 2^{-\langle \alpha, s \rangle}$;
\item[(V)] $\int_\sX (\Delta( L_\alpha(x)\zeta_\alpha(x) ))^2 \pi_0(dx) 
=: V_\alpha \leq C 2^{-\langle \alpha, \beta \rangle}$;
\item[(C)] ${\rm COST}(F_\alpha(\psi_\varphi)) =: C_\alpha \propto 
2^{\langle \alpha, \gamma \rangle}$.
\end{itemize}
\end{ass}

The proofs of the main Theorems will rely on one more result,
Lemma \ref{prop:MSEbound}, given immediately afterwards.

The next theorem comprises the main result of the paper.
%

\begin{theorem}\label{thm:main2}
Assume Assumptions \ref{ass:G1} and \ref{ass:rate}, with $\beta_i>\gamma_i$ for $i=1,\dots,D$. 
Then for any $\varepsilon>0$ and suitable $\varphi: \sX \rightarrow \bbR$, 
it is possible to choose a total degree index set $\mathcal{I}_{L}:=\{\alpha \in \bbN^{D}:\sum_{i=1}^{D}\delta_{i}\alpha_{i} \leq L, \sum_{i=1}^{D}\delta_{i} = 1\}$, $\delta_{i} \in (0,1]$ and $\{N_\alpha\}_{\alpha \in \mathcal{I}_{L}}$, such that for some $C>0$ 
$$
\bbE [ ( \widehat \varphi^{\rm MI}_\cI - \pi(\varphi) )^2 ] \leq C \varepsilon^2 \, ,
$$
and ${\rm COST}(\widehat \varphi^{\rm MI}_\cI) \leq C \varepsilon^{-2}$, the canonical rate.
The estimator $\widehat \varphi^{\rm MI}_\cI$ is defined in equation \eqref{eq:mismc}.
\end{theorem}

\begin{proof}
Starting from Lemma \ref{prop:MSEbound} and given Theorem \ref{thm:main}, 
and the Assumptions \ref{ass:rate},  
the result follows in a similar fashion to standard MIMC theory 
\cite{haji2016multi, giles2015multilevel, ourmimc, ourmismc1, ourmismc2}. 
The proof is given in Appendix \ref{app:b} for completeness.
\end{proof}

\blu{\begin{rem}
\label{thm:main1}
Under the same assumptions as in Theorem \ref{thm:main2}, 
and similar to Proposition \ref{prop:tdMIMC},
if the index set is replaced with the tensor product index set 
$\mathcal{I}_{L_1:L_d}:=\{\alpha \in \bbN^{D}:\alpha_1 \in \{0,...,L_1\},...,\alpha_D \in \{0,...,L_D\}\}$,
then the same complexity result can be obtained
only with an {\bf additional constraint} that
$\sum_{j=1}^{D}\gamma_j/ s_j \leq 2$.
\end{rem}}

\begin{lem}\label{prop:MSEbound}
For the estimator \eqref{eq:mismc} $\widehat\varphi^{\rm MI}_\cI = \frac{\sum_{\alpha \in \cI} F^{N_\alpha}_\alpha(\psi_{\varphi_\alpha})}
{\max\{\sum_{\alpha \in \cI} F^{N_\alpha}_\alpha(\psi_{1}), Z_{\min}\}}$, the following inequality holds $$\bbE [ ( \widehat \varphi^{\rm MI}_\cI - \pi(\varphi) )^2 ] \leq C \max_{\zeta \in \{\varphi, 1\}} 
\left(\sum_{\alpha \in \cI}\bbE\left[\left(F^{N_\alpha}_\alpha(\psi_{\zeta_\alpha}) - F_\alpha(\psi_{\zeta_\alpha})\right)^2\right] + \left(\sum_{\alpha \notin \cI}F_\alpha(\psi_{\zeta_\alpha})\right)^2\right) \, .$$
\end{lem}

\begin{proof}
Recall that from \eqref{eq:mismc} we have
$\widehat\varphi^{\rm MI}_\cI = \frac{\sum_{\alpha \in \cI} F^{N_\alpha}_\alpha(\psi_{\varphi_\alpha})}
{\max\{\sum_{\alpha \in \cI} F^{N_\alpha}_\alpha(\psi_{1}), Z_{\min}\}}$. 
So 
\begin{eqnarray*}
\bbE [ ( \widehat \varphi^{\rm MI}_\cI - \pi(\varphi) )^2 ]  &=& 
\bbE\left[\left( \frac{\sum_{\alpha \in \cI} F^{N_\alpha}_\alpha(\psi_{\varphi_\alpha})}
{\max\{\sum_{\alpha \in \cI} F^{N_\alpha}_\alpha(\psi_{1}), Z_{\min}\}}
- \frac{f(\varphi)}{f(1)} \right)^2\right] \\ 
&=& \bbE\Bigg[\Bigg( \frac{\sum_{\alpha \in \cI} F^{N_\alpha}_\alpha(\psi_{\varphi_\alpha})}
{\max\{\sum_{\alpha \in \cI} F^{N_\alpha}_\alpha(\psi_{1}), Z_{\min}\} 
f(1)} 
 \bigg (
f(1)\\
&& -\ \ \max\{\sum_{\alpha \in \cI} F^{N_\alpha}_\alpha(\psi_{1}), Z_{\min}\}\bigg ) \\ 
&+& \frac{1}{f(1)}\bigg (\sum_{\alpha \in \cI} F^{N_\alpha}_\alpha(\psi_{\varphi_\alpha}) - f(\varphi)\bigg) \Bigg)^2\Bigg].
\end{eqnarray*}
Since $f(1) \geq Z_{\min}$ and $| \max\{A,Z\} - \max\{B, Z\} | \leq | A-B|$,
we have
\begin{equation}\label{eq:maxes}
\bbE\left[\left(\max\{\sum_{\alpha \in \cI} 
F^{N_\alpha}_\alpha(\psi_{1}), Z_{\min}\} - f(1)\right )^2\right] \leq 
\bbE\left[\left(\sum_{\alpha \in \cI} F^{N_\alpha}_\alpha(\psi_{1}) - f(1)\right )^2\right] \, .
\end{equation} 
Then, we have
\begin{eqnarray*}
\bbE [ ( \widehat \varphi^{\rm MI}_\cI - \pi(\varphi) )^2 ] &\leq& 
C \max_{\zeta \in \{\varphi, 1\}} 
\bbE\left[\left(\sum_{\alpha \in \cI} F^{N_\alpha}_\alpha(\psi_{\zeta_\alpha}) - f(\zeta) \right)^2\right] \\
&\leq & C \max_{\zeta \in \{\varphi, 1\}} 
\bbE\Bigg[\left(\sum_{\alpha \in \cI}F^{N_\alpha}_\alpha(\psi_{\zeta_\alpha}) - \sum_{\alpha \in \cI}F_\alpha(\psi_{\zeta_\alpha})\right)^2  \\
&& +\ \ \Bigg(\sum_{\alpha \in \cI}F_\alpha(\psi_{\zeta_\alpha}) - f(\zeta) \Bigg)^2\Bigg] \\
&=& C \max_{\zeta \in \{\varphi, 1\}} 
\left(\sum_{\alpha \in \cI}\bbE\left[\left(F^{N_\alpha}_\alpha(\psi_{\zeta_\alpha}) - F_\alpha(\psi_{\zeta_\alpha})\right)^2\right] + \left(\sum_{\alpha \notin \cI}F_\alpha(\psi_{\zeta_\alpha})\right)^2\right) \, .
\end{eqnarray*}
\end{proof}

\blu{\begin{rem}
\label{rem:lowerbound}
It is remarked that one {\bf always has} $Z >0$,
therefore given a target error level $\varepsilon$,
one can always {\bf replace} $Z_{\rm min} \leftarrow \varepsilon$,
and achieve the same result.
To see this, denoting 
$$
\hat{Z}^{\sf MI}=\sum_{\alpha \in \cI} F^{N_\alpha}_\alpha(\psi_{1}) \, ,
$$
observe that line \eqref{eq:maxes} can be replaced with
$$
|\max\{\hat{Z}^{\sf MI},\varepsilon\} - Z| \leq 
|\hat{Z}^{\sf MI} -  Z|
+ |\varepsilon| + |\max\{{Z}-\varepsilon,0\} - \max\{{Z},0\}|
\leq |\hat{Z}^{\sf MI} -  Z|
+ 2|\varepsilon| \, .
$$
\end{rem}}

The Theorem \ref{thm:main2} formulates the total degree index set with general $\delta$ satisfying some loose conditions. In the paper \cite{haji2016multi}, 
optimal $\delta$ 
is constructed 
according to a profit indicator. The focus of the present work is on the canonical case, 
where the complexity is dominated 
by low levels, so we simply choose $\delta \propto s$.
The proof of the Theorem \ref{thm:main2} is based on the general $\delta$,
and it is easy to see that this choice suffices.

To achieve the canonical rate of complexity, 
Theorem \ref{thm:main1} with the tensor product index set relies on the essential assumption 
that $\sum_{j=1}^{D}\frac{\gamma_j}{s_j}\leq2$, 
which ensures that the samples at the finest index do not dominate the cost. 
If the assumption is violated, then only the sub-canonical complexity 
$\sum_{j=1}^{D}\frac{\gamma_j}{s_j}$ can be achieved.
This rate may often be $D-$dependent, resulting in a so-called {\em curse-of-dimensionality}.
In comparison, Theorem \ref{thm:main2} with the total degree index set releases this constraint, 
and improves the computational complexity for many problems from sub-canonical to canonical,
as illustrated in the numerical examples.

{
\subsection{Verification of assumptions}

\blu{Here we briefly discuss the models considered before
in connection with the required Assumptions 
\ref{ass:rate}.}
Note that both posteriors have the form
$\exp(\Phi(x))$ for some $\Phi:\sf X \rightarrow \bbR$,
and are approximated by 
\blu{$\Phi_\alpha:\sf X \rightarrow \bbR$.}

\begin{prop}\label{prop:expdod}
\blu{Let $\sX$ be a Banach space with $D=2$ s.t. $\pi_0(\sX)=1$, 
with norm $\|\cdot\|_\sX$.}
For all $\epsilon>0$, there exists a $C(\epsilon)>0$ such that
the following holds for $\Phi, \Phi_\alpha$ given by \eqref{eq:lgplike}, \eqref{eq:lgplikefin},
or $\log(L),\log(L_\alpha)$ from \eqref{eq:target}, respectively:
$$
\Delta \exp(\Phi_\alpha(x_\alpha)) \leq 
C(\epsilon) \exp(\epsilon \|x\|^2_\sX) \left ( |\Delta \Phi_\alpha(x_\alpha)| + 
|\Delta_1 \Phi_{\alpha-e_2} (x_{\alpha-e_2})| 
|\Delta_2 \Phi_{\alpha-e_1} (x_{\alpha-e_1})|\right) \, .
$$
\end{prop}
\begin{proof}
Let us introduce the shorthand notation
$A_{11}=\Phi_\alpha(x_\alpha)$,
$A_{10}=\Phi_{\alpha-e_2}(x_{\alpha-e_2})$,
$A_{01}=\Phi_{\alpha-e_1}(x_{\alpha-e_1})$,
$A_{00}=\Phi_{\alpha-e_1-e_2}(x_{\alpha-e_1-e_2})$.
We have 
\begin{eqnarray*}
\Delta \exp(\Phi_\alpha(x_\alpha)) &=& 
\exp(A_{11}) - \exp(A_{10})
-\left ( \exp(A_{01}) - 
\exp(A_{00})\right) \, \\
&=&\exp(A_{10})
\left (\exp(A_{11}-A_{10}) - 1\right ) - 
\exp(A_{00})
\left ( \exp(A_{01}-A_{00})- 1 \right ) \, \\
&=& \exp(A_{10})
\left (\exp(A_{11}-A_{10})-\exp(A_{01}-A_{00}) \right ) \, \\
&&+
\left ( \exp(A_{01})-\exp(A_{00})\right )
\left ( \exp(A_{10}-A_{00})- 1 \right ) \\
&\leq & C(\epsilon) \exp(\epsilon \|x\|^2_\sX) ( |A_{11} - A_{10} - (A_{01}-A_{00})| \\
&& \qquad + 
|A_{01}-A_{00}| |A_{10}-A_{00}|) \, ,
\end{eqnarray*}
where we have added and subtracted 
$\exp(A_{10})
\left ( \exp(A_{01}-A_{00})- 1 \right )$
in going from the second to the third line.
The final line follows from the mean value theorem
and equations \eqref{eq:phibnd2} and \eqref{eq:phidiff}
with $\sX=H^{\sf m}_r$, $r>1/2$.
These trivially hold for \eqref{eq:target}.

\sloppy The issue which prevented us from achieving above for LGC
is that terms like $\exp(-A_{10}) \propto \exp(-\Phi_\alpha(x_\alpha))$
appear in the constant, which involve a factor like
$\exp(\int\exp(x(z))dz)$. Fernique Theorem does not guarantee 
that such double exponentials are finite. However, for LGP, we 
instead have 
$$
\exp(-A_{10}) \propto (\int_\Omega \exp(x(z))dz)^{n} 
\leq |\Omega|^n \exp(n \|x\|_{L^\infty({\Omega})}) \leq |\Omega|^n \exp(n \|x\|_r).
$$
\end{proof}

\vspace{10pt}
\noindent
{\bf PDE. }
The following proposition updates Proposition
\ref{prop:standardpde}, and 
is proven in the literature on mixed regularity 
of the solution of elliptic PDE, 
as mentioned already in subsection \ref{sec:MIMC}. 
See e.g. \cite{haji2016multi} and references therein.

 \begin{prop}\label{prop:mixedratepde}
Let $u_\alpha$ be the solution of \eqref{eq:elliptic}-\eqref{eq:boundary}
at resolution $\alpha$, as described in subsection \ref{sec:fem},
for $a(x)$ given by \eqref{eq:diff} and uniformly over $x \in [-1,1]^d$, 
and $\mathsf{f}$ suitably smooth.
Then there exists a $C>0$ such that
$$
\| \Delta u_\alpha(x) \|_{V} \leq C 2^{-\alpha_1-\alpha_2} \, .
$$
Furthermore, 
$$
\| \Delta u_\alpha(x) \| \leq C 2^{-2(\alpha_1+\alpha_2)} \, .
$$

 \end{prop}

Note that since $L_\alpha(x) \leq C < \infty$ by Assumption \ref{ass:G1}, the constant in Proposition
\ref{prop:expdod} can be made uniform over $x$, and hence the required rate
in Assumption \ref{ass:rate} is established immediately.

\vspace{15pt}

\noindent
{\bf LGP. } 
{Will restrict consideration to LGP here, since LGC features double exponentials
which are difficult to handle theoretically in this context.}
The following proposition updates Proposition
\ref{prop:prodspecconv}
as required for differences of differences.

 \begin{prop}\label{prop:prodspecconvmi}
 Let $x \sim \pi_0$, where $\pi_0$ is a Gaussian process 
 of the form \eqref{eq:GPKL} with spectral decay 
corresponding to \eqref{eq:prodspec}, 
and let $x_\alpha$ correspond to truncation on the index set 
$\cA_\alpha = \cap_{i=1}^2 \{|k_i|\leq2^{\alpha_i}\}$ as in \eqref{eq:GPKLfin}. 
Then there is a $C>0$ such that for all $q<(\beta-1)/2$
$$
\| \Delta x_\alpha \|^2 \leq C \|x\|_q^2 2^{-2q \sum_{i=1}^2 \alpha_i} \, .
$$
 \end{prop}
 
 \begin{proof}
The proof follows along the same lines as that of Proposition \ref{prop:prodspecconv}
(\ref{prop:prodspecconvapp}),
except instead of projection onto 
$\cup_{i=1}^2 \{|k_i| > 2^{\alpha_i}\}$, the projection here 
is onto the set of indices
$\cap_{i=1}^2 \{2^{\alpha_i-1} \leq |k_i| \leq 2^{\alpha_i}\}$, i.e.
$$
\|A^{-q/2} P_{\cap_{i=1}^2 \{2^{\alpha_i-1} \leq |k_i| \leq 2^{\alpha_i}\}}\|_{\cL(L^2,L^2)}
\leq C 2^{-q \sum_{i=1}^2 \alpha_i} \, .
$$
 \end{proof}
 
 The key phenomenon that takes place is that the difference of difference 
 $\Delta x_\alpha$ leaves a remainder which is an intersection 
 $\cap_{i=1}^2 \{2^{\alpha_i-1} \leq |k_i| \leq 2^{\alpha_i}\}$, rather than the 
 union $\cup_{i=1}^2 \{2^{\alpha_i-1} \leq |k_i| \leq 2^{\alpha_i}\}$, associated to the 
 truncation error in Proposition \ref{prop:prodspecconv}, which one would achieve
 with a single difference from $x_\alpha - x_{\alpha-1}$.
 This eliminates all indices in which $k_i^{-1} =\cO(1)$ for some $i$,
 and provides the required product-form rates.

 \begin{prop}\label{prop:milgc}
 The rate from Proposition \ref{prop:prodspecconvmi}
 is inherited by the likelihood, resulting in verification of Assumption
 \ref{ass:rate}(V) with $\beta_i=\beta$.
\end{prop}
\begin{proof}
The proof follows along the lines of Proposition \ref{prop:biplgc}.
In this case, following from Proposition \ref{prop:expdod},
for all $\epsilon>0$ and $q<\beta/2$, we have 
\begin{eqnarray*}
\bbE_{\pi_{0}} \left( \Delta L_\alpha(x_\alpha) \right)^2 
&\leq& \bbE_{\pi_{0}} C(\epsilon)\exp(\epsilon \|x\|^2_r) 
\left ( \Delta \Phi_\alpha(x_\alpha)+ 
\Delta_1 \Phi_{\alpha-e_2} (x_{\alpha-e_2}) 
\Delta_2 \Phi_{\alpha-e_1} (x_{\alpha-e_1})\right)^2 \\
&\leq & C 2^{-2q \sum_{i=1}^2\alpha_i} \, ,
\end{eqnarray*}
where the second line is computed with estimates similar to \eqref{eq:phidiff},
and Fernique Theorem to conclude, as in the proof of Proposition \ref{prop:biplgc}.

\end{proof}
}

{\section{{Proofs relating the Theorem \ref{thm:main}}}
\label{app:a}

In this section we prove Lemmas \ref{lem:mrt}, 
\ref{lem:qp} and \ref{lem:induc} from which Theorem \ref{thm:main} is an immediate consequence.  
We fix $\alpha \in \cI$ throughout this section and thus, to avoid notational overload, we henceforth suppress it from our notation.

For $j=2,\dots, J$, we define 
\begin{equation}\label{eq:phi}
\Phi_j(\Pi) := \frac{\Pi(H_{j-1} \bm\cM_j)}{\Pi(H_{j-1})} \, ,
\end{equation}
and observe that the iterates of the algorithm of subsection \ref{sec:smc} can be rewritten in the concise form
\begin{equation}\label{eq:phiform}
\bx^i_j \sim \Phi_j(\Pi^N_{j-1}) \,  , \quad {\rm for} ~~ i=1,\dots,N \, ,
\end{equation}
where we recall the definition \eqref{eq:empirical} for 
{the empirical measure $\Pi^N_{j-1}$}.
Let $\Phi_1(\Pi) := \Pi$.
A finer error analysis beyond Proposition \ref{prop:unnocoupled} requires 
keeping track of the effect of errors $(\Pi_j^N - \Phi_j(\Pi_{j-1}^N))$ and accounting for 
the cumulative error at time $J$. 
To this end, and for any $\psi: \sX^{2^D} \rightarrow \bbR$ bounded, 
we define the following partial propagation operator for $p = 1,\dots, J$
\begin{equation}\label{eq:q}
Q_{n,J}(\psi)(\bx_n) =
\int_{\sX^{2^D \times (J-n)}} \psi(\bx_J) \prod_{j=n+1}^{J} H_{j-1}(\bx_{j-1}) \bm\cM_j(\bx_{j-1}, \bx_j) d\bx_{n+1:J} \, ,
\end{equation}
where $Q_{J,J} = I_{2^D}$, i.e. $Q_{J,J}(\psi)(\bx_J) = \psi(\bx_J)$. 
We will assume for simplicity that $F_1 := \Pi_0$, the prior, so that $F_1(1)=1$.
Note that then $\Pi_0(Q_{1}(\psi)(\bx_1)) = F_J(\psi) = F(\psi)$.

We present now the following well-known representation of the error as a  martingale w.r.t. the natural 
filtration of the particle system
(see \cite{vihola,del2004feynman})
\begin{equation}\label{eq:mrt}
F_{J}^N(\psi) - F(\psi) = \sum_{n=1}^J 
\underbrace{F_n^N(1) \left [ \Pi_n^N - \Phi_n(\Pi_{n-1}^N)\right ] (Q_{n,J}(\psi))}_{S_{n,J}^N(\psi)} \, ,
\end{equation}
where we denote the summands as $S_{n,J}^N(\psi)$. 
This clearly shows the unbiasedness property presented in 
Proposition \ref{prop:unnocoupled}. 
In particular $\bbE F_{J}^N(\psi) = F(\psi)$, 
as can be seen by backwards induction conditioning first on 
$\{\bx_{J-1}^i\}_{i=1}^N$ and recalling the form of $F_n^N(1)$ 
given in \eqref{eq:unnocoup}.
This brings us to our first supporting lemma. Throughout these calculations $C$ is a finite constant whose value may change on each appearance.
The dependencies of this constant on the various algorithmic parameters is made clear from the statement.

\begin{lem}\label{lem:mrt}
Assume Assumption \ref{ass:G1}. 
Then for any $J\in\mathbb{N}$ there exists $C>0$ such that for any $N\in\mathbb{N}$  and any  $\psi:\sX^{2^D} \rightarrow \bbR$ bounded
and measurable
$$
\bbE[(F_{J}^N(\psi) - F(\psi))^2] \leq \frac{C}{N} \sum_{n=1}^J \bbE[ (Q_{n,J}(\psi) (\bx_n^1))^2 ].
$$
\end{lem}

\begin{proof}
Following from \eqref{eq:mrt} we have
\begin{eqnarray}\nonumber
\bbE(F_{J}^N(\psi) - F(\psi))^2 &\leq&  C \sum_{n=1}^J \bbE[(S_{n,J}^N(\psi))^2] \\ \nonumber
&\leq& \frac{C}{N} \bbE[(Q_{n,J}(\psi) (\bx_n^1))^2 ] \, .
\end{eqnarray}
The first inequality results from application of the Burkholder-Gundy-Davis inequality. 
The second inequality follows via an application of the conditional Marcinkiewicz-Zygmund inequality and the fact that $F_n^N(1)$
is upper-bounded by a constant via Assumption \ref{ass:G1}.
\end{proof}

\begin{lem}\label{lem:qp}
Assume Assumption \ref{ass:G1}.
Then for any $J\in\mathbb{N}$ and $n\in\{1,\dots,J\}$ there exists a $C>0$ such that for any $N\in\mathbb{N}$ and $\psi:\sX^{2^D} \rightarrow \bbR$ bounded
and measurable
$$
(Q_{n,J}(\psi)(\bx_n))^2 \leq C Q_{n,J}(\psi^2)(\bx_n) \, .
$$
\end{lem}

\begin{proof}
Observe that for any $\bx_n \in \sX^{2^D}$,
$Q_{n,J}(\psi)(\bx_n)/Q_{n,J}(1)(\bx_n)$ is a probability distribution.
Therefore, Jensen's inequality provides 
\begin{eqnarray}\nonumber
(Q_{n,J}(\psi)(\bx_n))^2  &=& (Q_{n,J}(1)(\bx_n))^2 \left( \frac{Q_{n,J}(\psi)(\bx_n)}{Q_{n,J}(1)(\bx_n)}\right)^2 \\
\nonumber
& \leq & Q_{n,J}(1)(\bx_n) Q_{n,J}(\psi^2)(\bx_n) \, .
\end{eqnarray}
The result follows with $C = \sup_{\bx_n\in \sX^{2^D}} Q_{n,J}(1)(\bx_n)$.
\end{proof}

\begin{lem}\label{lem:induc}
Assume Assumption \ref{ass:G1}.
Then for any $J\in\mathbb{N}$ and $n\in\{1,\dots,J\}$ there exists a $C>0$ such that for any $N\in\mathbb{N}$  and any  $\psi:\sX^{2^D} \rightarrow \bbR$ bounded
and measurable
$$
 \bbE[ Q_{n,J}(\psi^2)(\bx_n^1) ] \leq C F(\psi^2).
$$
\end{lem}

\begin{proof} We proceed by induction. 
The result for $n=1$ follows immediately from Lemma \ref{lem:qp} and
the fact that we defined $\Phi_1(\Pi_1) = \Pi_1 = F_1 = \Pi_0$:
\begin{eqnarray}\nonumber
\Pi_0(Q_{1,J}(\psi^2)(\bx_1)) &=& 
\int_{\sX^{2^D \times J}} \psi^2(\bx_J) \Pi_0(\bx_1) \prod_{j=1}^{J-1} H_j(\bx_j) \bm\cM_{j+1}(\bx_j,\bx_{j+1}) d\bx_{1:J} \\
\nonumber
&=& \int_{\sX^{2^D \times J}} \psi^2(\bx_J) F(\bx_J) d\bx_J \, .
\end{eqnarray}

Now, assume the result holds for $n-1$:
\begin{equation}\label{eq:induc}
 \bbE[ Q_{n-1,J}(\psi^2)(\bx_{n-1}^1) ] = 
\bbE(\Phi_{n-1}(\Pi_{n-2}^N)[ Q_{n-1,J}(\psi^2) ] ) \leq C F(\psi^2) \, ,
\end{equation}
and we will show that this implies it holds for $n$.

We have that
\begin{eqnarray}\nonumber
\Phi_{n}(\Pi_{n-1}^N)[ Q_{n,J}(\psi^2) ] &=& \underbrace{\Phi_n(\Phi_{n-1}(\Pi_{n-2}^N))[ Q_{n,J}(\psi^2) ]}_{T_1}  \\
\nonumber
&+& \underbrace{\{\Phi_{n}(\Pi_{n-1}^N)-\Phi_n(\Phi_{n-1}(\Pi_{n-2}^N))\}[ Q_{n,J}(\psi^2) ]}_{T_2} \, .
\end{eqnarray}
We consider bounding the expectations of $T_1$ and $T_2$ in turn. 

\noindent {$\bm{T_1}$. } We have 
\begin{eqnarray*}
&&\Phi_n(\Phi_{n-1}(\Pi_{n-2}^N))[ Q_{n,J}(\psi^2) ] = \frac1{\Phi_{n-1}(\Pi_{n-2}^N)(H_{n-1})} \\
&\times&
\int_{\sX^{2^D \times 2}} \Phi_{n-1}(\Pi_{n-2}^N)(d\bx_{n-1}) 
H_{n-1}(\bx_{n-1}) \bm\cM_n(\bx_{n-1},d\bx_n) Q_{n,J}(\psi^2)(\bx_n) \, .
\end{eqnarray*}
By  Assumption \ref{ass:G1} $\inf_{\bx} H_{n-1}(\bx)\geq C^{-1}$ and 
$$
Q_{n-1,J}(\psi^2)(\bx_{n-1}) = \int_{\sX^{2^D}} H_{n-1}(\bx_{n-1}) \bm\cM_n(\bx_{n-1},d\bx_n) Q_{n,J}(\psi^2)(\bx_n) \, .
$$
Therefore by the inductive hypothesis
\begin{eqnarray*}
\bbE \left( \Phi_n(\Phi_{n-1}(\Pi_{n-2}^N))[ Q_{n,J}(\psi^2) ] \right) 
&\leq& C \bbE\left ( \Phi_{n-1}(\Pi_{n-2}^N)[Q_{n-1,J}(\psi^2)] \right) \\
& \leq & C F(\psi^2) \, .
\end{eqnarray*}

\noindent {$\bm{T_2}$. } For the second term,  we have
$$
\left | \{\Phi_{n}(\Pi_{n-1}^N)-\Phi_n(\Phi_{n-1}(\Pi_{n-2}^N))\}[ Q_{n,J}(\psi^2) ] \right | = 
$$
$$
\left | \frac{\Pi_{n-1}^N(H_{n-1}M_nQ_{n,J}(\psi^2))}{\Pi_{n-1}^N(H_{n-1})} 
- 
\frac{\Phi_{n-1}(\Pi_{n-2}^N)(H_{n-1}M_nQ_{n,J}(\psi^2))}{\Phi_{n-1}(\Pi_{n-2}^N)(H_{n-1})} \right | \leq
$$
$$
\underbrace{\left | \frac1{{\Pi_{n-1}^N(H_{n-1})}}( \Pi_{n-1}^N - \Phi_{n-1}(\Pi_{n-2}^N))(Q_{n-1,J}(\psi^2)) \right |}_{T_{2,1}} +
$$
$$
\underbrace{\left | \frac{(\Phi_{n-1}(\Pi_{n-2}^N) -  \Pi_{n-1}^N)(H_{n-1})}
{\Pi_{n-1}^N(H_{n-1}) \Phi_{n-1}(\Pi_{n-2}^N)(H_{n-1})} \Phi_{n-1}(\Pi_{n-2}^N)(Q_{n-1,J}(\psi^2))  \right |}_{T_{2,2}} \, .
$$
These two terms are now considered.

{\noindent {$\bm{T_{2,1}}$. } 
The expected value of $T_{2,1}$ can be bounded as follows
\begin{eqnarray*}
  \mathbb E \left[ 
	\left|
		\left(
			\Pi_{n-1}^N - \Phi_{n-1} ( \Pi_{n-2}^N)
		\right) 
		(Q_{n-1,J}) (\psi^2)
	\right|
\right] 
&
\leq
&
\mathbb E\left[
	\left|
		\Pi_{n-1}^N(Q_{n-1,J})(\psi^2)
	\right|
\right]\\
&&+ \ \ 
\mathbb E \left[
	\left|
		\Phi_{n-1}(\Pi_{n-1}^N)(Q_{n-1,J}(\psi^2))
	\right|
\right] 
\\
&=
&
\mathbb E\left[
		\Pi_{n-1}^N(Q_{n-1,J})(\psi^2)
\right]\\
&&+ \ \ 
\mathbb E \left[
		\Phi_{n-1}(\Pi_{n-1}^N)(Q_{n-1,J}(\psi^2))
\right]
\\
&\leq 
&
2 
\mathbb E \left[
		\Phi_{n-1}(\Pi_{n-1}^N)(Q_{n-1,J}(\psi^2))
\right]
\\
&\leq 
&
2C F (\psi^2) \, .
\end{eqnarray*}
Where the the triangle inequality is used in the first line, 
positivity is used in the second, the Martingale property is used in the third,
and the induction hypothesis is used to conclude. 
Thus, after appropriately redefining the constant $C$, we have that
$$
\mathbb{E}[T_{2,1}] \leq CF(\psi^2).
$$
}

\noindent {$\bm{T_{2,2}}$. } Finally, for the second term, 
note that by Assumption \ref{ass:G1} there is a $C<\infty$ such that
$$
\left | \frac{(\Phi_{n-1}(\Pi_{n-2}^N) -  \Pi_{n-1}^N)(H_{n-1})}
{\Pi_{n-1}^N(H_{n-1}) \Phi_{n-1}(\Pi_{n-2}^N)(H_{n-1})} \right | \leq C \, .
$$
Thus, by again applying the induction hypothesis \eqref{eq:induc} one has that
$$
\mathbb{E}[T_{2,2}] \leq CF(\psi^2)
$$
and this suffices to complete the proof.
\end{proof}
}

\section{Numerical Results}
\label{ref:numerics}

The codes for the numerical tests can be founded in \url{https://github.com/Shangda-Yang/MISMCRE.git}.

First we considered the toy example of a 1D DE simplification of the PDE
introduced in subsection \ref{sec:pde}. Since the method reduces to a multilevel
method in this case, the results are provided in Appendix \ref{app:1dtoy}.

\subsection{2D Elliptic PDE with random diffusion coefficient}
In this subsection, we look at the 2D elliptic PDE with random diffusion coefficient
from subsection \ref{sec:pde}. 
The problem is defined in \eqref{eq:elliptic}-\eqref{eq:boundary}.
The domain of interest is $\Omega = [0,1]^2$, the forcing term is $\mathsf{f} = 100$,
$a(x)(z) = 3+x_1\cos(3z_1)\sin(3z_2)+x_2\cos(z_1)\sin(z_2)$, 
and the prior is 
$x\sim U[-1,1]^2$. The observation operator and observation take the form of \eqref{eq:obop} and \eqref{eq:ob} respectively. 

Let the observations be given at a set of four points - \{(0.25,0.25), (0.25,0.75), (0.75,0.25), (0.75,0.75)\}. Corresponding observations are generated by 
$y = u_{\alpha}(x^*) + \nu$, 
where $u_{\alpha}(x^*)$ is the approximate solution of the PDE at {$\alpha = [10,10]$ with $x^*=[-0.4836, -0.5806]$ drawn from $U[-1,1]^2$},
 and $\nu \sim N(0,0.5^2)$. Due to the zero Dirichlet boundary condition, the solution is zero when $\alpha_i = 0$ and $\alpha_i=1$ for $i = 1,2$. 
 So we set $\alpha_i \leftarrow \alpha_i + 2$ for $i = 1,2$ as the starting indices. 
 {The 2D PDE solver applied here is modified based on a MATLAB toolbox called IFISS \cite{ers07} such that the solver can accept a random coefficient and solve the problem of interest. \blu{The algorithm is applied with Metropolis-Hastings method and a fixed tempering schedule for all $\alpha$, where $J = 3$
 and $\tau_j=(j-1)/2$.}

For this example, we have $s_1 = s_2 = 2$ and $\beta_1 = \beta_2 = 4$ 
for the mixed rates 
corresponding to Assumptions \ref{ass:rate}, which implies that 
along the diagonal $\alpha_1 = \alpha_2$ the 
rates for $\Delta F_\alpha$ are
$s_1 + s_2 = 4$ and $\beta_1 + \beta_2 = 8$. 
This is shown in Figures \ref{fig:2DnlinMISMC1}
and \ref{fig:2DnlinMISMC12}.
The contour plot \ref{fig:2DnlinContour} performs a more general illustration. For the multilevel formulation, $s = 2$ and $\beta = 4$, which can be observed from Figure \ref{fig:2DnlinMLSMC}.

Considering the quantity of interest $x_1^2 + x_2^2$, MSE in the Figure \ref{fig:2DnlinRate} are calculated with 200 realisations. Total computational costs are computed with the same idea as the previous questions. The reference solution is computed by MLSMC instead of MISMC to avoid errors in algorithm. MISMC algorithm is carried on with the two different index sets mentioned above - tensor product and total degree index set. According to the rates of regression in the caption of Figure \ref{fig:2DnlinRate}, both MLSMC and MISMC with the self-normalised increment estimator and the ratio estimator have the rate of convergence close to -1 falling into the canonical case, which is as expected.

The advantage of MISMC can be shown in the following examples, where we can only achieve subcanonical rates with MLSMC and MISMC with the tensor product index set but the canonical rate with MISMC with the total degree index set. It is worth to note that this advantage is because the total degree index set which abandons the most expensive estimation.

\begin{figure}[H]
    \centering
    \includegraphics[width=.6\linewidth]{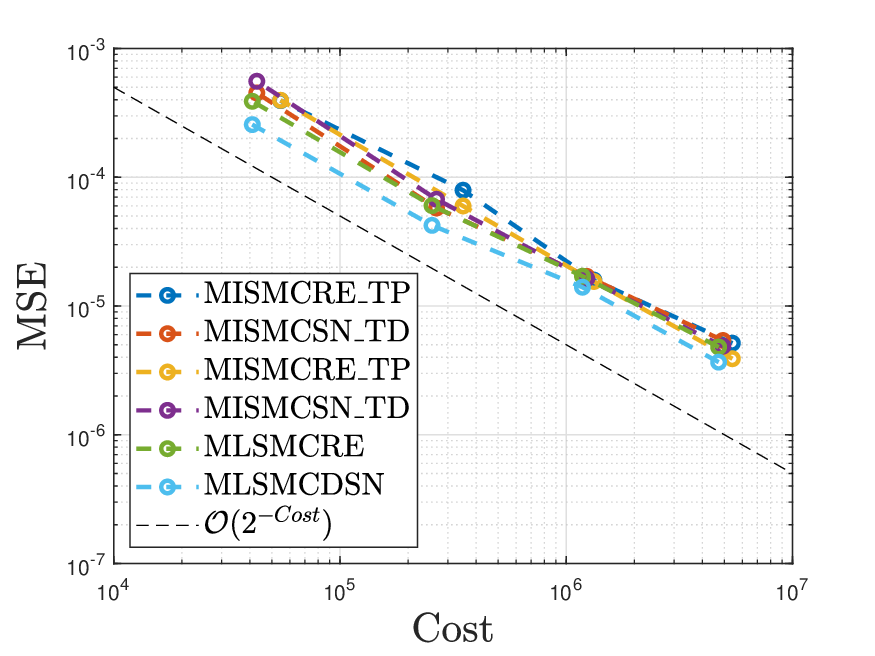}
    \caption{2D Elliptic PDE with random diffusion coefficient rate of convergence for MLSMC and MISMC with the self-normalised increments estimator and the ratio estimator, where MISMC is applied with the tensor product index set and the total degree index set. Each MSE is computed with 200 realisations. Rates of regression: (1) MISMCSN\_TP: $-1.007$ (2) MISMCSN\_TD: $-0.996$ (3) MISMCRE\_TP: $-0.964$ (4) MISMCRE\_TD: $-0.925$ (5) MLSMCSN: $-0.880$ (6) MLSMCRE: $-0.918$. }
    \label{fig:2DnlinRate}
\end{figure}

\subsection{Log-Gaussian Process Models}

After considering the PDE examples in previous subsections, 
we show the numerical results of the 
LGC model introduced in subsection \ref{sec:lgc}. {The parameters are chosen as $\theta=(\theta_1,\theta_2,\theta_3) = (0, 1, 110.339)$.} For this particular example, the increment rates associated to MLSMC are 
$s = 0.8$ and $\beta = 1.6$, while the mixed rates associated to 
MISMC are $s_{i} = 0.8$ and $\beta_{i} = 1.6$ for $i = 1,2$. 
The rates for $s$ and $\beta$ can be observed from the Figure \ref{fig:LGCMLSMC} 
and mixed rates 
for $s_i$ and $\beta_i$ for $i = 1,2$ can be observed from the Figures 
\ref{fig:LGCMISMC2}, \ref{fig:LGCMISMC12} and \ref{fig:LGCContour}. 
{This forward simulation method 
has a cost rate of $\gamma = 2 + \omega$, for any $\omega > 0$, 
while the traditional full factorization method used in \cite{heng_jacob_2019} 
(and references therein) has $\gamma=6$.
However, one has $\gamma_i=1+\omega < \beta_i < \gamma$. 
This means circulant embedding will deliver a single level complexity of approximately 
MSE$^{-9/4}$, 
while the traditional grid-based approach has complexity MSE$^{-19/4}$. 
An implementation of MLMC delivers MSE$^{-5/4 - \omega}$. 
Finally, MIMC with TD index set ($\delta_{i} = 0.5$ for $i = 1,2$)
delivers {\em canonical complexity of} MSE$^{-1}$.
Note that, because $\sum_{j=1}^{2}\frac{\gamma_{j}}{s_{j}} = 5/2 > 2$, 
the important assumption $\sum_{j=1}^{2}\frac{\gamma_{j}}{s_{j}} \leq 2$ 
for MISMC with TP index set 
is violated, the cost of the finest level samples
dominates the total cost, and MISMC TP 
therefore has the same sub-canonical complexity as MLMC.}

\blu{SMC sampler is applied with the 
pre-conditioned Crank-Nicolson (pCN) MCMC \cite{pcn, nealpcn} 
as the mutation kernel and adaptive tempering described in Remark \ref{rem:temp}.}
The quantity of interest is taken as $\varphi(x)=\int_{[0,1]^2}\exp(x(z))dz$ and 
$\alpha_{i} \leftarrow \alpha_i + 5$ for $i = 1,2$ are the starting indices.
Figure \ref{fig:LGCRate}, and the rate of regression in the caption, 
show the above claims that MISMC TD 
is canonical with rate of convergence close to -1 and MLSMC is subcanonical. 
MISMC TP 
is not included here since the 
computational complexity of this method is the same as that of MLSMC for this example. The only difference between the two methods is the constant. 
Compared with MLSMC, MISMC TP 
has extra indices and two extra terms at all internal indices.
This means that MISMC TP 
has a larger constant than MLSMC. 
MISMC TD 
turns the computational complexity from subcanonical, 
which all that is achievable with MISMC TP, 
MLSMC, and SMC, 
to canonical, indicating the benefits of MISMC TD. 

\begin{figure}[H]
    \centering
    \includegraphics[width=.6\linewidth]{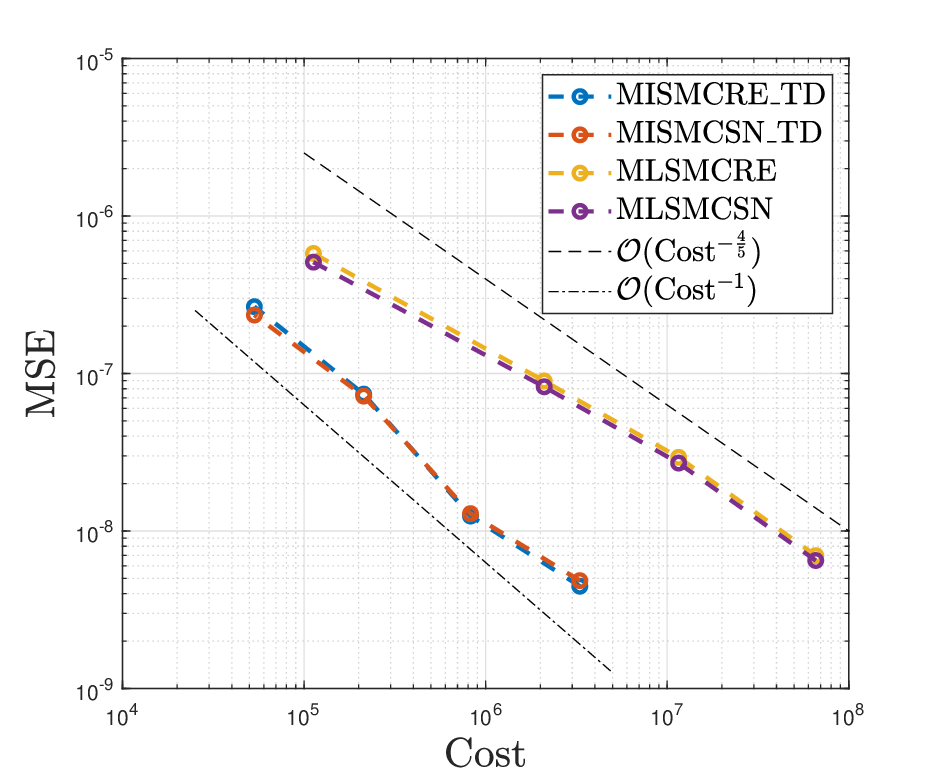}
    \caption{LGC model with MLSMC and MISMC with the self-normalised increments estimator and the ratio estimator, where MISMC corresponds to 
    the total degree index set. 
    Each MSE is computed with 100 realisations. 
    Rate of regression: (1) MISMCRE\_TD: -1.022 (2) MISMCSN\_TD: -0.973 (3) {MLSMCRE: -0.686} (4) {MLSMCSN: -0.677}.}
    \label{fig:LGCRate}
\end{figure}

{\subsection{Log-Gaussian Process Model}
In this subsection, we consider the LGP model introduced in subsection \ref{sec:lgc}. By changing the likelihood and parameters accordingly {as $\theta=(\theta_1,\theta_2,\theta_3) = (0, 1, 27.585)$}, the LGP model follows the same analysis as the LGC model in the previous subsection and gives the same numerical results for regularity and complexity. More precisely, the increment rates associated to MLSMC are $s = 0.8$ and $\beta = 1.6$, and the mixed rates associated to MISMC are $s_{i} = 0.8$ and $\beta_{i} = 1.6$ for $i = 1,2$, which are the same as LGC. The Figure \ref{fig:LGPMLSMC} shows the increment rates for $s$ and $\beta$ and the Figure \ref{fig:LGPMISMC2}, \ref{fig:LGPMISMC12} and \ref{fig:LGPContour} shows the mixed rates for $s_i$ and $\beta_i$ for $i = 1,2$. The rates corresponding to the computational costs of MLSMC and MISMC are $\gamma = 2+\omega$ and $\gamma_i = 1+\omega<\gamma$, for any $\omega>0$, respectively. Being the same as that of LGC, the complexity of LGP is MSE$^{-5/4 - \omega}$ with MLSMC and MSE$^{-1}$ with MISMC.} 

\blu{As above, SMC sampler is applied with the 
pre-conditioned Crank-Nicolson (pCN) MCMC \cite{pcn, nealpcn} 
as the mutation kernel and adaptive tempering described in Remark \ref{rem:temp}.}
{Considering the quantity of interest $\varphi(x)=\int_{[0,1]^2}\exp(x(z))dz$ and letting the starting indices $\alpha_{i} \leftarrow \alpha_i + 5$ for $i = 1,2$,
Figure \ref{fig:LGPRate}, and the rate of regression in the caption, 
show the same claims that MISMC TD is canonical with rate of convergence close to -1 and MLSMC is subcanonical.}

\begin{figure}[H]
    \centering
    \includegraphics[width=.6\linewidth]{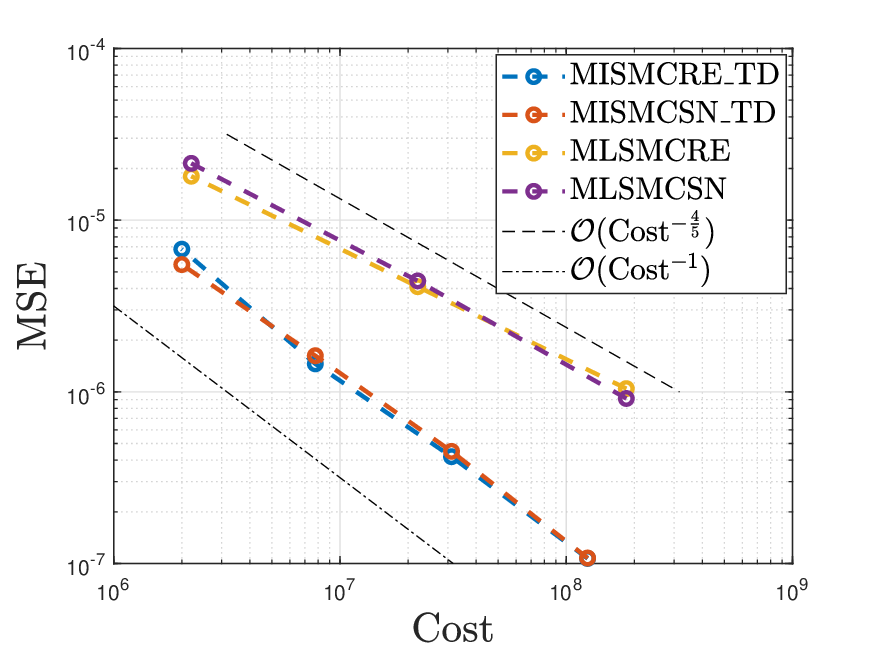}
    \caption{LGP model with MLSMC and MISMC with the self-normalised increments estimator and the ratio estimator, where MISMC corresponds to 
    the total degree index set. 
    Each MSE is computed with 100 realisations. 
    Rate of regression: (1) MISMCRE\_TD: -0.994 (2) MISMCSN\_TD: -0.950 (3) {MLSMCRE: -0.643} (4) {MLSMCSN: -712}.}
    \label{fig:LGPRate}
\end{figure}

\subsubsection*{Acknowledgements}

AJ was supported by KAUST baseline funding.

\appendix

\section{Proofs relating to Theorem \ref{thm:main2} and Remark 
\ref{thm:main1}}
\label{app:b}

Let $B_{\alpha} = F_\alpha(\psi_{\zeta_\alpha})$, 
and recall from Assumption \ref{ass:rate} (V) that 
$$\bbE\left[\left(F^{N_\alpha}_\alpha(\psi_{\zeta_\alpha}) - F_\alpha(\psi_{\zeta_\alpha})\right)^2\right] \leq V_\alpha/N_\alpha \, ,$$
and from Assumption \ref{ass:rate} (C) that the 
total computational cost is $\sum_{\alpha\in\cI}N_{\alpha}C_{\alpha}$. 
Following from Theorem \ref{thm:main}, 
we have
\begin{equation}\label{eq:MSEVB}
\bbE [ ( \widehat \varphi^{\rm MI}_\cI - \pi(\varphi) )^2 ] \leq C \max_{\zeta \in \{\varphi, 1\}} 
\left(\sum_{\alpha \in \cI}\frac{V_{\alpha}}{N_{\alpha}} + \left(\sum_{\alpha \notin \cI}B_{\alpha}\right)^2\right).
\end{equation}
Then, $\bbE [ ( \widehat \varphi^{\rm MI}_\cI - \pi(\varphi) )^2 ]$ is less than $C\varepsilon^2$ as long as both $\max_{\zeta \in \{\varphi, 1\}}\sum_{\alpha \in \cI}V_{\alpha}$ and $\max_{\zeta \in \{\varphi, 1\}}\left(\sum_{\alpha \notin \cI}B_{\alpha}\right)^2$ are of $\cO(\varepsilon^2)$. 
We can now prove the Remark \ref{thm:main1} as follows.
\begin{proof}
We start from inequality (\ref{eq:MSEVB}) and replace the general index set $\cI$ by the tensor product index set $\mathcal{I}_{L_1:L_D}:=\{\alpha \in \bbN^{d}:\alpha_1 \in \{0,...,L_1\},...,\alpha_D \in \{0,...,L_D\}\}$. Let $L_i = \lceil\log_{2}(D/\varepsilon)/s_i\rceil$, for $i = 1,...,D$, where $\lceil\cdot\rceil$ denotes ceiling a noninteger to an integer. The bias term is bounded as follows
\begin{equation*}
\sum_{\alpha \notin \mathcal{I}_{L_1:L_D}}B_{\alpha} = \sum_{\alpha \notin \mathcal{I}_{L_1:L_D}}F_\alpha(\psi_{\zeta_\alpha}) \leq C\sum_{\alpha \notin \mathcal{I}_{L_1:L_D}}\prod_{i=1}^D 2^{-\alpha_i s_i} \leq C\sum_{i=1}^D 2^{- L_i s_i},
\end{equation*}
where the inequality above follows from Assumption \ref{ass:rate}(B).
Substituting $L_i$ in the inequality, the bias term is of $\cO(\varepsilon)$. By Lemma \ref{lem:Lag}, $\sum_{\alpha \in \cI_{L_1:L_D}}V_{\alpha}/N_{\alpha}$ is minimised and equals $\varepsilon^2$ by choosing 
$$N_{\alpha} = \varepsilon^{-2}\left(\sum_{\alpha'\in\mathcal{I}_{L_1:L_D}}\sqrt{V_{\alpha'}C_{\alpha'}}\right)\l\sqrt{\frac{V_{\alpha}}{C_{\alpha}}}.$$
The sample size can only be treated as an integer and there should be at least one sample in each multi-index of resolution. So let the upper bound of $N_{\alpha}$ be $$N_{\alpha} \leq 1 + \varepsilon^{-2}\left(\sum_{\alpha'\in\mathcal{I}_{L_1:L_D}}\sqrt{V_{\alpha'}C_{\alpha'}}\right)\l\sqrt{\frac{V_{\alpha}}{C_{\alpha}}}.$$
Then, the total computational cost $C_{\mathcal{I}_{L_1:L_D}}$ is given by $$C_{\mathcal{I}_{L_1:L_D}} = \sum_{\alpha \in \mathcal{I}_{L_1:L_D}}N_{\alpha}C_{\alpha} = \cO\left(\varepsilon^{-2}\left(\sum_{\alpha\in\mathcal{I}_{L_1:L_D}}\sqrt{V_{\alpha}C_{\alpha}}\right)^{2} + \sum_{\alpha\in\mathcal{I}_{L_1:L_D}}C_{\alpha} \right).$$ 

By Assumption \ref{ass:rate}, we have  $$\sum_{\alpha\in\mathcal{I}_{L_1:L_D}}\sqrt{V_{\alpha}C_{\alpha}} \leq \sum_{\alpha \in \mathcal{I}_{L_1:L_D}}C\prod_{i=1}^{D}2^{\alpha_{i}(\gamma_{i}-\beta_{i})/2} = \left(\prod_{i=1}^{D}C\sum_{\alpha_{i}=0}^{L_{i}}2^{\alpha_{i}(\gamma_{i}-\beta_{i})/2} \right).$$ Since $\beta_i > \gamma_i$, $\sum_{\alpha\in\mathcal{I}_{L_1:L_D}}\sqrt{V_{\alpha}C_{\alpha}} = \cO(1)$. In addition, $\sum_{\alpha\in\mathcal{I}_{L_1:L_D}}C_{\alpha} = \cO(\varepsilon^{-\sum_{i=1}{D}\gamma_i/s_i})$ and this is bounded by $\cO(\varepsilon^{-2})$ due to the assumption that
$\sum_{i=1}^{D}\gamma_i/s_i \leq 2$. Thus, the total computational cost is dominated by $\cO(\varepsilon^{-2})$.
\end{proof}

The proof of Theorem \ref{thm:main2} is similar as that of Remark \ref{thm:main1}. 
The details are as follows.
\begin{proof}
We start from inequality (\ref{eq:MSEVB}) and replace the general index set with the total degree index set $\mathcal{I}_{L}:=\{\alpha \in \bbN^{d}:\sum_{i=1}^{D}\delta_{i}\alpha_{i} \leq L, \sum_{i=1}^{D}\delta_{i} = 1\}.$ 

Let $L = \log \left(
\varepsilon^{-1} (\log \varepsilon^{-1})^{2(n_1-1)}
\right)/A_1,$ where $A_1 = \min_{i=1,...,D}\log(2)\delta_{i}^{-1}s_i$ and $n_1 = \#\{i=1,...,D:\ \log(2)\delta_{i}^{-1}s_i=A_1\}.$ Using Lemma \ref{lem:ineq1}, the bias term can be bounded as follows
\begin{eqnarray*}
\sum_{\alpha \notin \mathcal{I}_{L}}B_{\alpha} &=& \sum_{\alpha \notin \mathcal{I}_{L}}F_\alpha(\psi_{\zeta_\alpha})\\
&\leq& C\sum_{\alpha \notin \mathcal{I}_{L}}\prod_{i=1}^D 2^{-\alpha_i s_i}\\
&\leq& C\int_{\{\mathbf{x}\in\bbR_{+}^{D}:\ \sum_{i=1}^{D}\delta_{i}x_{i}> L,\  \sum_{i=1}^{D}\delta_{i}=1\}}\prod_{i=1}^{D}e^{-\log(2)x_{i}s_i}\rm{d}\mathbf{x}\\
&=& C\int_{\{\mathbf{x}\in\bbR_{+}^{D}:\ \sum_{i=1}^{D}x_{i}> L\}}e^{-\sum_{i=1}^{D}\log(2)\delta_{i}^{-1}x_{i}s_i}\rm{d}\mathbf{x}\\
&\leq& Ce^{-A_1L}L^{n_{1}-1}
\end{eqnarray*}
where $A_1$ and $n_1$ are defined above. Substituting in the $L$ and applying Lemma \ref{lem:ineq3}, the bias term is of $\cO(\varepsilon)$.

Following the similar steps as the proof for Remark \ref{thm:main1} and replacing the tensor product index set with the total degree index set, the total computational cost $C_{\cI_{L}}$ can be formulated as $$C_{\mathcal{I}_{L}} = \sum_{\alpha \in \mathcal{I}_{L}}N_{\alpha}C_{\alpha} = \cO\left(\varepsilon^{-2}\left(\sum_{\alpha\in\mathcal{I}_{L}}\sqrt{V_{\alpha}C_{\alpha}}\right)^{2} + \sum_{\alpha\in\mathcal{I}_{L}}C_{\alpha} \right).$$
Starting from the first term, since $\beta_{i}>\gamma_{i}$, we have
\begin{eqnarray*}
\sum_{\alpha\in\mathcal{I}_{L}}\sqrt{V_{\alpha}C_{\alpha}} &\leq& \sum_{\alpha\in\mathcal{I}_{L}}2^{\sum_{i=1}^{D}\alpha_{i}(\gamma_i-\beta_i)/2}\\
&\leq&\frac{1}{\prod_{i=1}^{D}(1-2^{(\gamma_{i}-\beta_{i})/2})}.
\end{eqnarray*}
Considering the second term $\sum_{\alpha\in\cI_{L}}C_{\alpha}$ and using Lemma \ref{lem:ineq2}, we have
\begin{eqnarray*}
\sum_{\alpha\in\cI_{L}}C_{\alpha} &=& \sum_{\alpha\in\cI_{L}}2^{\sum_{i=1}^{D}\alpha_{i}\gamma_{i}}\\
&\leq& C\int_{\{\mathbf{x}\in\bbR_{+}^{D}:\ \sum_{i=1}^{D}x_{i}\leq L\}}e^{\sum_{i=1}^{D}log(2)\delta_{i}^{-1}\gamma_{i}x_{i}}\rm{d}\mathbf{x}\\
&\leq& Ce^{A_2L}L^{n_2-1},
\end{eqnarray*}
where $A_2 = \max_{i=1,...,D}\log(2)\delta_{i}^{-1}\gamma_{i}$ and $n_2 = \#\{i=1,...,D:\ \log(2)\delta_{i}^{-1}\gamma_{i}=A_2\}.$

Substituting $L$ into the upper bound, and since $2s_{i} \geq \beta_{i} > \gamma_{i}$, 
we have $\gamma_i/s_i \leq 2$ which gives $\sum_{\alpha\in\cI_{L}}C_{\alpha} \leq \cO(\varepsilon^{-2})$. 
Then, the summation of the two terms is of $\cO(\varepsilon^{-2})$.
\end{proof}

\blu{Lemma \ref{lem:ineq1} and \ref{lem:ineq2} below are from Lemma 6.3 and 6.2 of \cite{haji2016multi}.}

\begin{lem}\label{lem:ineq1}
\blu{For $L \geq 1$ and $\mathbf{a}\in\bbR_{+}^{D}$, there exists a $C(\mathbf{a})>0$
such that the following inequality holds} $$\int_{\{\mathbf{x}\in\bbR_{+}^{D}:\ \sum_{i=1}^{D}x_{i}> L\}}e^{-\sum_{i=1}^{D}a_{i}x_{i}}\rm{d}\mathbf{x} \leq 
C e^{-A_1L}L^{n_{1}-1},$$ where 
$$A_1 = \min_{i=1,...,D}a_{i},\ n_{1} =\#\{i=1,...,D:\ a_i = A_1\}.$$
\end{lem}

\begin{lem}\label{lem:ineq2}
\blu{For $L \geq 1$ and $\mathbf{a}\in\bbR_{+}^{D}$, there exists a $C(\mathbf{a})>0$
such that the following inequality holds} \blu{$$\int_{\{\mathbf{x}\in\bbR_{+}^{D}:\ \sum_{i=1}^{D}x_{i}\leq L\}}e^{\sum_{i=1}^{D}a_{i}x_{i}}\rm{d}\mathbf{x} \leq 
Ce^{A_2L}L^{n_{2}-1},$$ where 
$$A_2 = \max_{i=1,...,D}a_{i},\ n_{2} = \#\{i=1,...,D:\ a_i = A_2\}.$$}
\end{lem}

{
\begin{lem}\label{lem:ineq3}
	For 
$
L = \log \left(
\epsilon^{-1} (\log \epsilon^{-1})^{2(n-1)}
\right)/A
$
\begin{align*}
  e^{-AL} L^{n-1} \leq C \epsilon
\end{align*}
where $C = (2(n-1) /A)^{n-1} $.
\end{lem}
\begin{proof}
  The argument follows by the following sequence of equalities. The final inequality, below, follows since $[\log x -x]\leq 0$ for $x= \log \log \epsilon^{-1}$.
\begin{align*}
&  \log \left(
	e^{-AL} L^{n-1} 
\right)
\\
= 
&
- AL + (n-1) \log L
\\
=
&
- \log \epsilon^{-1}
- 2(n-1) \log \log \epsilon^{-1}
+
(n-1) \log \left(
	\frac{1}{A}
\log \left(
	\epsilon^{-1}
(\log \epsilon^{-1})^{2(n-1)}
\right)
\right)\\
=
&
 \log \epsilon
- 2(n-1) \log \log \epsilon^{-1}
+
(n-1)\log \frac{1}{A}
+
(n-1) \log \log \epsilon^{-1}
\\
&+(n-1) \log \left(
	2(n-1) \log \log \epsilon^{-1}
\right)\\
=
&
\log \epsilon
+(n-1)\log (2(n-1)/A)
+
(n-1)\left[
	\log \log \log \epsilon^{-1} - \log \log \epsilon^{-1}
\right]
\\
\leq 
&
\log  (\epsilon (2(n-1)/A)^{n-1} ),
\end{align*}
as required.
\end{proof}
}

{
\begin{lem}\label{lem:Lag}
For a fixed 
$\varepsilon^2=\sum_{\alpha \in \cI}\bbE\left[\left(F^{N_\alpha}_\alpha(\psi_{\zeta_\alpha}) - F_\alpha(\psi_{\zeta_\alpha})\right)^2\right] = 
\sum_{\alpha \in \cI}\frac{V_{\alpha}}{N_{\alpha}}$ (from Lemma \ref{prop:MSEbound}),
the cost is minimised 
by choosing $N_{\alpha}$ such that $$N_{\alpha} = \varepsilon^{-2}\left(\sum_{\alpha'\in\mathcal{I}}\sqrt{V_{\alpha'}C_{\alpha'}}\right)\l\sqrt{\frac{V_{\alpha}}{C_{\alpha}}}.$$
\end{lem}
}
{
\begin{proof}
Given fixed 
$\varepsilon^2=\sum_{\alpha \in \cI}\frac{V_{\alpha}}{N_{\alpha}}$,
the cost can be minimised as a function of $\{N_\alpha\}_{\alpha \in \cI}$ by applying the Lagrange multiplier method. For some Lagrange multiplier $\lambda$, we solve the minimisation problem 
$$\min_{N_{\alpha}}\sum_{\alpha \in \cI}N_{\alpha}C_{\alpha} 
+ \lambda^2\left(\sum_{\alpha \in \cI}\frac{V_{\alpha}}{N_{\alpha}} -\varepsilon^2\right) \, .$$ 
This gives the optimal value of $N_{\alpha} = \lambda\sqrt{V_{\alpha}/C_{\alpha}}$
for each $\alpha\in \cI$. 
Plugging the solution to $\{N_\alpha\}_{\alpha \in \cI}$ into the constraint equation 
$\varepsilon^2=\sum_{\alpha \in \cI}\frac{V_{\alpha}}{N_{\alpha}}$ gives 
$\lambda = \varepsilon^{-2}\sum_{\alpha \in \cI}\sqrt{V_{\alpha}C_{\alpha}}$ 
and therefore 
$$N_{\alpha} = \varepsilon^{-2}\left(\sum_{\alpha'\in\mathcal{I}}\sqrt{V_{\alpha'}C_{\alpha'}}\right)\l\sqrt{\frac{V_{\alpha}}{C_{\alpha}}}.$$
\end{proof}
}

\section{LGC results}
\label{app:lgc}

{
We restate and prove Proposition \ref{prop:prodspecconv}.

\begin{prop}\label{prop:prodspecconvapp}
 Let $x \sim \pi_0$, where $\pi_0$ is a Gaussian process 
 of the form \eqref{eq:GPKL} with spectral decay 
corresponding to \eqref{eq:prodspec}, 
and let $x_\alpha$ correspond to truncation on the index set 
$\cA_\alpha$ as in \eqref{eq:GPKLfin}. 
Then $x \in H_q^{\sf m}$ for all $q<\beta/2$, and for $r \in [0,q)$
there is a $C>0$ such that 
$$
\| x_\alpha - x \|_r^2 \leq C \|x\|_q^2 2^{-2(q-r)\min_i \alpha_i} \, .
$$
 \end{prop}
 
 \begin{proof}
Since $x\sim \pi_0$ is a Gaussian process, 
in order to prove $x \in H_q^{\sf m}$ for all $q<\beta/2$, 
it suffices to prove that 
$\bbE \| x \|_q^2 <\infty$.
Indeed there is a $C>0$ such that
$$
\bbE \|x\|_q^2 \leq C \sum_{k\in \bbZ^2} 
((1+k_1^2)(1+k_2^2))^{q-\frac{\beta+1}{2}} \, ,
$$
from which it is clear that $2q<\beta$ provides a sufficient condition for summability.
{To see this, define 
$x_k=\langle \phi_k, x \rangle \equiv 
\int \phi_k(z) x^*(z) dz$, and note that
$\bbE |x_k|^2 = \rho_k(\theta)$
and $\{\phi_k\}_{k \in \bbZ}$ are orthonormal}. 
In more detail,
\begin{align*}
    \mathbb{E}[\|x\|_q^2] &= \mathbb{E}\left[\|A^{q/2}x\|^2\right]\\
&{= \mathbb{E}\left \| \sum_{k\in \mathbb{Z}}a_k^{q/2} \phi_k 
\underbrace{\langle \phi_k, x \rangle}_{x_k}
\right \|^2} \\
 &={\sum_{k,k'\in\mathbb{Z}^2} a_k^{q/2} a_{k'}^{q/2}
 \underbrace{\langle \phi_k, \phi_{k'} \rangle}_{\delta_{k,k'}}
  \mathbb{E} x_kx_{k'}}  \\
&=\sum_{k\in\mathbb{Z}^2} 
 (1+k_1^2)^{q}(1+k_2^2)^{q}\mathbb{E} |x_k|^2 \\    
    &= \sum_{k\in\mathbb{Z}^2}(1+k_1^2)^{q}(1+k_2^2)^{q}\rho_k(\theta)^2\\
&\leq C \sum_{k\in\mathbb{Z}^2}(1+k_1^2)^{q}(1+k_2^2)^{q}
\frac{1}{((1+k_1^2)(1+k_2^2))^{\frac{\beta+1}{2}}}\\
    &= C \sum_{k\in\mathbb{Z}^2}((1+k_1^2)(1+k_2^2))^{q-\frac{\beta+1}{2}}
\end{align*}

Now let $P_{\cA_{\alpha}}$ denote the projection onto the index set $\cA_\alpha$. 
Observe that 
there is a $C>0$ such that 
\begin{eqnarray}\nonumber
\| A^{-q/2} - A^{-q/2}P_{\cA_{\alpha}} \|_{\cL(L^2,L^2)}^2 
&=& \sup_{\|x\|=1} \sum_{k \notin \cA_\alpha} a_k^{-q} x_k^2 \\
 \label{eq:opnorm}
 &\leq& C ( 2^{-2q\alpha_1} + 2^{-2q\alpha_2}) \, ,
\end{eqnarray}
where $\cL(L^2,L^2)$ denotes the space of linear operators from 
$L^2$ to $L^2$.

For $r\in[0,q)$, we have 
\begin{eqnarray*}
 \| x_\alpha - x \|_r^2 &=& 
\|A^{-q/2}A^{(q+r)/2}(x - P_{\cA_\alpha}x)  \|^2
 \, \\
 &\leq & \| (A^{-(q-r)/2} - A^{-(q-r)/2}P_{\cA_{\alpha}}) A^{q/2} x \|^2 \\
&\leq& C \|x\|_q^2 2^{-2(q-r)\min_i \alpha_i} \, ,
 \end{eqnarray*}
 where the first line follows from the definition, the second follows from commutativity 
 of $P_{\cA_{\alpha}}$ and $A$, 
 and the final line follows from the definition of operator norm 
 and \eqref{eq:opnorm}.
\end{proof}
}

We restate and prove Proposition \ref{prop:lgcpost}.

{\begin{prop}
	Given $x: \Omega \rightarrow \mathbb R$ is a Gaussian process, with probability measure denoted $\pi_0$, defined on compact finite dimensional space $\Omega$, that is almost surely continuous and has a finite mean and covariance. If we define $\pi$ by
\begin{eqnarray*}
{\sf(LGC)} \qquad \frac{d \pi}{d \pi_0}(x) 
&\propto& \exp \left [\sum_{j =1}^{n} x(z_j) - \int_{\Omega} \exp(x(z)) dz \right] \, , \\
{\sf(LGP)} \qquad \frac{d \pi}{d \pi_0}(x) 
&\propto& \exp \left [\sum_{j =1}^{n} x(z_j) - n\log \int_{\Omega} \exp(x(z)) dz \right]  \, .
\end{eqnarray*} 
for $n \in \mathbb N$ then $\pi(dx)$ is a well-defined probability measure,
and can be represented in terms of its density with respect to $\pi_0$:
$$
\pi(dx) = \frac{d\pi}{d\pi_0} \pi_0(dx) \, .
$$ 
\end{prop}
}

We first proof the proposition for LGC.
\begin{proof}
	For $\pi$ with LGC to be well-defined, the first exponential above must be integrable. Specifically, we require that 
\[
0 < Z:= \mathbb E_{\pi_0} \left[ \exp \left\{ \sum_{j =1}^{n} x(s_j) - \int_{\Omega} \exp(x(s)) ds \right\} \right] < \infty \, .
\]
To upper-bound $Z$ notice that $ \sum_{j=1}^n x(s_j)$ is a real-valued gaussian random variable with finite mean and variance, which we denote by $\mu$ and $\sigma^2$, respectively. Also $\exp(x(s)) $ is non-negative thus
\begin{align*}
Z :=&\, \mathbb E_{\pi_0} \left[ \exp \left\{ \sum_{j =1}^{n} x(s_j) - \int_{\Omega} \exp(x(s)) ds \right\} \right]
\notag 
\\
\leq &\, \mathbb E_{\pi_0} \left[ \exp \left\{ \sum_{j =1}^{n} x(s_j) \right\} \right]
= \mathbb E [ e^{\mu + \sigma^2/2} ] < \infty \, .
\end{align*}
This gives the required upper-bound.

For the lower-bound, we note that since $x(s)$ is almost surely continuous and $\Omega$ is compact, then $\sup_{s\in \Omega} x(s)$ is almost surely finite. 
Thus $\int_\Omega \exp ( x(s) ) ds$ is almost surely finite, because $\int_\Omega \exp ( x(s) ) ds < |\Omega |\exp (\sup_{s\in \Omega} x(s))$. Thus, by Monotone convergence, there exists a value of $K_1$ such that 
\[
\mathbb P \Big( \int_\Omega \exp ( x(s) ) ds >  K_1 \Big) \leq 1/4
\] 
Similarly since $\sum_{i=1}^n x_i(s)$ is Gaussian, there exists a value of $K_2$ such that 
\[
\mathbb P \left( \sum_{j=1}^n x_j(s) \leq - K_2 \right) \leq 1 / 4
\]
Taking, $K= 2(K_1 \vee K_2) $
\begin{align*}
&\mathbb P \left( \sum_{j =1}^{n} x(s_j) - \int_{\Omega} \exp(x(s)) ds  > -K\right) \\
\geq 
&
\mathbb P \left( \sum_{j =1}^{n} x(s_j) > - K/2 ,
- \int_{[0,1]^2} \exp(x(s)) ds  > -K/2\right) 
\notag 
\\
\geq 
&
1
-
\mathbb P\left( \sum_{j =1}^{n} x(s_j) \leq - K/2 \right)
-
\mathbb P \left( K/2 \leq \int_{\Omega} \exp(x(s)) ds \right) 
\notag 
\\
\geq
&
1 - \frac{1}{4}- \frac{1}{4} = \frac{1}{2} \, .
\end{align*}
Thus Markov's inequality for corresponding value of $K$ gives
\begin{align*}
Z :=&\, \mathbb E_{\pi_0} \left[ \exp \left\{ \sum_{j =1}^{n} x(s_j) - \int_{\Omega} \exp(x(s)) ds \right\} \right]
\notag 
\geq
\frac{1}{2} e^{-K} > 0 \, ,
\end{align*}
as required.
\end{proof}

\blu{We now prove the proposition for LGP.}

\begin{proof}
	For $\pi$ with LGP to be well-defined, the second exponential above must be integrable. Specifically, we require that 
\[
0 < Z:= \mathbb E_{\pi_0} \left[ \exp \left\{ \sum_{j =1}^{n} x(s_j) - n \log \left( \int_{\Omega} \exp(x(s)) ds \right) \right\} \right] < \infty \, .
\]
\blu{First, we show $Z$ is lower bounded by $0$.} 
Notice that since the process $x(s)$ is almost surely continuous in $s$ and the domain $\Omega$ is compact. Thus almost surely it holds that 
\[
\tilde Z := \exp \left\{ \sum_{j =1}^{n} x(s_j) - n \log \left( \int_{\Omega} \exp(x(s)) ds \right) \right\} >0 
\] 
As an immediate consequence $Z=\mathbb E_{\pi_0}[ \tilde Z ] > 0$, as required.

\blu{We now show $Z$ is upper bounded.} Notice that, if we let $ u = (u_i : i=1,..,n)$ be independent uniformly distributed on $\Omega$ (wlog we assume $|\Omega|=1$) then we can bound  as follows
\begin{align}
\tilde Z
=&
\exp \left\{ \sum_{j =1}^{n} x(s_j) \right\}
\left( 
\int_{\Omega} \exp(x(s)) ds
\right)^{-n}
=
\exp \left\{ \sum_{j =1}^{n} x(s_j) \right\}
\left(
	\mathbb E_{u} \left[
	e^{\sum_{i=1}^n x(u_i) }
\right]
\right)^{-1}\notag 
\\
\leq 
&
\exp \left\{ \sum_{j =1}^{n} x(s_j) \right\}
\mathbb E_{u}  \left[
	e^{- \sum_{i=1}^n x(u_i) }
\right]
=
\mathbb E_{u}  \left[ 
\exp \left\{ \sum_{j =1}^{n} x(s_j) - \sum_{i=1}^n x(u_i)  \right\}
\right] \label{Ztilde}
\end{align}
The inequality above applies Jensen's Inequality. 
Thus we see we can bound $Z$ by bounding the expectation of $\exp \{ \sum_{j =1}^{n} x(s_j) - \sum_{i=1}^n x(u_i)  \}$ with respect to $\pi_0$ and $u$. 
By Fubini's Theorem \cite{williams1991probability}, 
we can hold $u$ fixed and take the expectation with respect to $\pi_0$. 
Notice that under $\pi_0$, $x$ is a Gaussian process with bounded mean and covariance. We let $\mu$ and $\sigma$ bound the mean and covariance of $x$. So, conditional on $u$, it holds that
\begin{align*}
\mathbb E_{\pi_0} \left[ \exp \left\{ \sum_{j =1}^{n} x(s_j) - \sum_{i=1}^n x(u_i)  \right\}\right]
\leq e^{n\mu + n\sigma^2/2} \, .
\end{align*}
Since the above upper-bound is independent of $u$, we have that 
\begin{align}\label{Ubound}
\mathbb E_{\pi_0} \left[ \mathbb E_{u} \left[  \exp \left\{ \sum_{j =1}^{n} x(s_j) - \sum_{i=1}^n x(u_i)  \right\}\right]\right]
\leq e^{n\mu + n\sigma^2/2} \, .
\end{align}
{Note that Fubini for positive random variables allows interchanging integrals 
without a priori finite guarantees \cite{williams1991probability}.}
Combining inequalities \eqref{Ztilde} and \eqref{Ubound} we have
\begin{align*}
Z = \mathbb E_{\pi_0} [ \tilde Z ] \leq  \mathbb E_{\pi_0} \left[ \mathbb E_{u} \left[  \exp \left\{ \sum_{j =1}^{n} x(s_j) - \sum_{i=1}^n x(u_i)  \right\}\right]\right]
\leq e^{n\mu + n\sigma^2/2} < \infty \, ,
\end{align*}
which gives the required upper-bound on $Z$.
\end{proof}

{We now restate and prove Proposition \ref{prop:biplgc}.
\begin{prop}\label{prop:biplgcapp}
For both LGP and LGC, 
there is a $C>0$ such that for $x \sim \pi_0$
and $x_\alpha = P_{\cA_\alpha}x${, where  $P_{\cA_\alpha}$ denotes the projection onto the index set $\cA_\alpha$ defined in \eqref{eq:GPKLfin},}
the following rate estimate holds for all $q<(\beta-1)/2$ 
$$
\bbE | L_\alpha(x_\alpha) - L(x) |^2 \leq C 2^{-2\min\{q,1\} \min\{\alpha_1,\alpha_2\}} \, .
$$
\end{prop}

\begin{proof}
The result is proven for the more difficult case of LGC. {The LGP case follows similarly.}
Define 
\begin{eqnarray}\label{eq:phi1}
\Phi(x) &:=& \sum_{j =1}^{n} x(z_j) - \int_{\Omega} \exp(x(z)) dz\\ \label{eq:phialpha}
\Phi_\alpha(x_\alpha) &:=&
\sum_{j =1}^{n} \hat{x}_\alpha(z_j) - 
Q(\exp(x_{\alpha})) \, ,
\end{eqnarray}
where we recall the definition of $\hat{x}_\alpha$ above \eqref{eq:lgclikefin}.
Now $L(x) = \exp(\Phi(x))$ and $L_\alpha(x_\alpha) = \exp(\Phi_\alpha(x_\alpha))$. 
Note that $\pi_0 = N(0,\cC)$ and that $\cC$ has the kernel representation
$$
{\cC(z,z') = \sum_{k_1} \rho_{k_1}^2\phi_{k_1}(z_1)\phi_{k_1}(z_1')
\sum_{k_2} \rho_{k_2}^2\phi_{k_2}(z_2)\phi_{k_2}(z_2') =: 
\cC_1(z_1,z'_1)\cC_2(z_2,z'_2) \, .}
$$
This means that the dependence between $z_1$ and $z_2$
is only statistical, via $\xi_k$ in \eqref{eq:GPKL}, and 
therefore a given realization $x \sim \pi_0$ admits factorization 
{$$x(z)=\sum_{k_1} \rho_{k_1}\phi_{k_1}(z_1) x_{2,k_1}(z_2) \, ,$$ }
where, for each $k_1$, the i.i.d. random variables
{$$
x_{2,k_1}(z_2) = \sum_{k_2} \xi_{k_1,k_2} \rho_{k_2}\phi_{k_2}(z_2) \, 
$$}
have the properties of $x_2' \sim N(0,\cC_2)$, i.e. 
$x_{2,k_1} \in H_{\beta/2}$ and 
the 1-dimensional Sobolev Embedding Theorem
(see e.g. \cite{stuart2010inverse}) provides
$\|x_{2,k_1}\|_{L^\infty({\Omega})} \leq C \|x_{2,k_1}\|_r$, for 
$r \in [1/2,\beta/2]$.
Furthermore, letting 
{$\hat{x}_1(z_1):=\sum_{k_1} \rho_{k_1} \|x_{2,k_1}\|_r \phi_{k_1}(z_1)$,}
it is clear that $\|\hat{x}_1\|_r = \|x\|_r < \infty$ for $r<\beta/2$. 
Hence, applying Sobolev Embedding Theorem again on $\hat{x}_1(z_1)$,
we have $\|\hat{x}_1\|_{L^\infty({\Omega})} \leq C \|\hat{x}_1\|_r$, 
for $r \in [1/2,\beta/2]$.
Now, since $\|\phi_{k_1}\|_{L^\infty({\Omega})}=1$ for all $k_1$
and {$\rho_{k_1} \|x_{2,k_1}\|_r\geq 0$,
\begin{eqnarray}
\|\hat{x}_1\|_{L^\infty({\Omega})} &=& \sup_{z_1} \sum_{k_1} \rho_{k_1} \|x_{2,k_1}\|_r |\phi_{k_1}(z_1)| \\
&\geq& C \sup_{z_1} \sum_{k_1} \rho_{k_1} \|x_{2,k_1}\|_{L^\infty({\Omega})} |\phi_{k_1}(z_1)| \\
&\geq& C \sup_{z_1,z_2} \left | \sum_{k_1} \sum_{k_2} 
\rho_{k_1}\phi_{k_1}(z_1)\xi_{k_1,k_2} \rho_{k_2}\phi_{k_2}(z_2)\right| \\
&=& C \|x\|_{L^\infty({\Omega})} \, .
\end{eqnarray}}
Sobolev Embedding Theorem is used on 
$x_{2,k_1}$ in the second line, and definitions are used in the third and fourth lines.
In conclusion, 
$\|x\|_{L^\infty({\Omega})} \leq C \|x\|_r$, for $r \geq 1/2$.
We have that 

\begin{equation}\label{eq:sobemb}
\sum_{j =1}^{n} x(z_j) \leq 
n \sup_{z\in \Omega} x(z) \leq C \| x \|_r\, .    
\end{equation}
Therefore 
\begin{equation}\label{eq:phibnd1}
\Phi(x), \Phi_\alpha(x_\alpha) \leq \|x\|_r \, .
\end{equation}
Furthermore, observe that this implies that for all $\epsilon>0$
\begin{equation}\label{eq:phibnd2}
\Phi(x), \Phi_\alpha(x_\alpha) \leq \epsilon \|x\|_r^2 + \epsilon^{-1} \, .
\end{equation}
To see this consider the cases $\| x \|_r > \epsilon^{-1}$ and $\| x \|_r \leq \epsilon^{-1}$
separately.

By the mean value theorem, 
\begin{align}
\mathbb{E} [| L_\alpha(x_\alpha) - L(x) |^2] &= \int_{H_{\beta/2}^{\sf m}} \left(\exp(\Phi(x)) - \exp(\Phi_{\alpha}(x_\alpha)\right)^2d\pi_0\\
\label{eq:likbnd}
&\leq \int_{H_{\beta/2}^{\sf m}} \exp\left(\max\{\Phi(x),\Phi_\alpha(x_\alpha)\}\right)\left|\Phi(x)-\Phi_{\alpha}(x_\alpha)\right|^2d\pi_0 \, .
\end{align}
Note that for the exponential term in the integral
\begin{align}
    \exp(\max\{\Phi(x),\Phi_\alpha(x_\alpha)\}) &\leq \exp(\Phi(x)) + \exp(\Phi_\alpha(x_\alpha))\\
    \label{eq:expbnd}
    &\leq C(\epsilon)\exp(\epsilon\|x\|_r^2)
\end{align}
and for the squared term
\begin{align}
    |\Phi(x)-\Phi_{\alpha}(x_\alpha)|^2 &= \left|\sum_{j =1}^{n} x(z_j) - \int_{\Omega} \exp(x(z)) dz - \sum_{j =1}^{n} \hat{x}_\alpha(z_j) + Q(\exp(x_{\alpha}))\right|^2\\
    \label{eq:sumbnd}
    &\leq 2\Bigg|\sum_{j =1}^{n} x(z_j) - \sum_{j =1}^{n} \hat{x}_\alpha(z_j)\Bigg|^2 + 2\Bigg|\int_{\Omega} \exp(x(z)) dz - Q(\exp(x_{\alpha})) dz\Bigg|^2
\end{align}

The bound \eqref{eq:expbnd} is clear following \eqref{eq:phibnd2}. Now consider the bound of $|\Phi(x)-\Phi_{\alpha}(x_\alpha)|^2$. For the first term of \eqref{eq:sumbnd}, let $\hat{x}_\alpha$ correspond to piecewise linear interpolation for simplicity. As in \eqref{eq:sobemb}
we have 
$$\Bigg|\sum_{j =1}^{n} ( x(z_j) - \hat{x}_\alpha(z_j) )\Bigg| 
\leq C \|x - \hat{x}_\alpha\|_r.$$
Given standard piecewise linear approximation
estimates which lead to Proposition \ref{prop:standardpde},
and weaker versions such as \cite{ern2004theory}
$$
\|x_\alpha - \hat{x}_\alpha\| \leq 2^{-\min\{\alpha_1,\alpha_2\}} \|\nabla x_\alpha \| 
\leq C 2^{-\min\{\alpha_1,\alpha_2\}} \| x_\alpha \|_1 \, ,
$$
it is natural to assume the following generalization, for $p>0$ and $r+p \leq 2$,
\begin{equation}
\|x_\alpha - \hat{x}_\alpha\|_r \leq 2^{-p\min\{\alpha_1,\alpha_2\}} \| x_\alpha \|_{r+p} \, .
\end{equation}
For $p=(\beta-1)/2$,
and $r=1/2$, 
the bound is
\begin{equation}\label{eq:interpbnd}
\|x - \hat{x}_\alpha\|_r \leq \| x - x_\alpha \|_r + \|x_\alpha - \hat{x}_\alpha \|_r 
\leq \| x - x_\alpha \|_r + 
2^{-\frac{(\beta-1)}{2}\min\{\alpha_1,\alpha_2\}} \| x_\alpha \|_{\beta/2}\, .
\end{equation}
Recall that, by \ref{prop:prodspecconv}, 
$x\sim \pi_0$
implies that $x \in H^{\sf m}_{\beta/2}$ a.s. and hence $x_\alpha \in H^{\sf m}_{\beta/2}$ a.s.

Now consider the second term of \eqref{eq:sumbnd}. For the sake of concreteness, 
we use the trapezoidal quadrature rule so that 
$Q(\exp(x_\alpha))=2^{-(\alpha_1+\alpha_2)} 
\sum_{h \in \prod_{i=1}^2\{0,2^{-\alpha_i}, \dots, 
1\}} w_h \exp(x_{\alpha}(h))$ 
(where $w_h=(1/2)^{I}$ and $I=\#\{i; h_i\in \{0,1\}\}$, 
i.e. the boundary terms are down-weighted, by $1/2$ on edges and $1/4$ at corners). Now
\begin{eqnarray}\nonumber
\int_{\Omega} \exp(x(z)) dz - Q(\exp(x_\alpha)) &=&
\int_{\Omega} (\exp(x(z)) - \exp(x_\alpha(z))) dz \\ \label{eq:quadsplit}
&+& 
\int_{\Omega} \exp(x_\alpha(z)) dz- Q(\exp(x_\alpha)) \, .
\end{eqnarray}
For the first term, we have
\begin{eqnarray}\nonumber
\int_{\Omega} (\exp(x(z)) - \exp(x_\alpha(z))) dz &\leq & 
\| \exp(\max\{x(z),{x}_\alpha(z)\}) \|  \|x - {x}_\alpha \|  \\
\label{eq:quadfirst}
&\leq& C \exp(\|x\|_r) \| x - {x}_\alpha \| \, ,
\end{eqnarray}
where the first line follows from the mean value theorem
and Cauchy Schwartz inequality,
while in the second line, the first factor follows from the inequality 
$\|x\|_{L^\infty({\Omega})} \leq \|x\|_r$, 
and the fact that $|\Omega|<\infty$.
Note that we are restricted to the $\|\cdot \|_r$
estimate as a result of the pointwise observations, 
as in \eqref{eq:phibnd1}, but $\|x\| \leq \|x\|_r$ so 
\eqref{eq:quadfirst} is suitable.
For the second term of \eqref{eq:quadsplit}, 
since the trapezoidal rule for $D=2$ follows from iterating 
the $D=1$ rules, Theorem 1.8 of \cite{cruz2002sharp},
along with similar manipulations as above, implies
\begin{equation}\label{eq:quad2nd}
\int_{\Omega} \exp(x_\alpha(z)) dz- Q(\exp(x_\alpha)) 
\leq C 2^{-\min\{\alpha_1,\alpha_2\}} \exp(\|x\|_r) \| x_\alpha \|_1 \, .
\end{equation}
Following from Proposition \ref{prop:prodspecconv},
{we require $\beta \geq 2$ so that $x \in H^{\sf m}_1$ a.s.}
and the constant in the second term \eqref{eq:quad2nd} is controlled. 
Combining \eqref{eq:quadfirst} and \eqref{eq:quad2nd} in \eqref{eq:quadsplit},
we have 
\begin{equation}\label{eq:integralbnd}
\int_{\Omega} \exp(x(z)) dz - Q(\exp(x_\alpha)) \leq
C \exp(\|x\|_r) \| x - {x}_\alpha \| + C 2^{-\min\{\alpha_1,\alpha_2\}} \exp(\|x\|_r) \| x_\alpha \|_1 \, .
\end{equation}
Now let $r=1/2+\delta$ for $\delta>0$ arbitrarily small.
Then 
for $q \in (0,(\beta-1)/2)$, 
plugging \eqref{eq:interpbnd} and \eqref{eq:integralbnd} into \eqref{eq:sumbnd} 
and using 
the same argument leading to \eqref{eq:phibnd2}
and then Proposition \ref{prop:prodspecconv}, 
we have 
\begin{equation}\label{eq:phidiff}
|\Phi(x)-\Phi_{\alpha}(x_\alpha)|^2\leq C(\epsilon) \exp(2\epsilon \|x\|^2_{\beta/2})2^{-2\min\{q,1\} \min\{\alpha_1,\alpha_2\}} \, .
\end{equation}

We note that our interest here is in rough priors with $q\leq 1$.
In case $q>1$, one would employ higher order interpolation and 
quadrature such that these errors do not limit the rate of convergence. 

Finally,
\begin{align*}
    \mathbb{E} [| L_\alpha(x_\alpha) - L(x) |^2]
    &\leq C(\epsilon)\int_{H_{\beta/2}^{\sf m}} \exp(3\epsilon \|x\|^2_{\beta/2})d\pi_02^{-2\min\{q,1\} \min\{\alpha_1,\alpha_2\}}\\
    &\leq C2^{-2\min\{q,1\} \min\{\alpha_1,\alpha_2\}}
\end{align*}
The first inequality is by substituting \eqref{eq:expbnd} and \eqref{eq:phidiff} in \eqref{eq:likbnd} and applying 
$\|x\|^2_{r} \leq \|x\|^2_{\beta/2}$. 
We note that the first inequality holds for all $\epsilon>0$. 
Fernique Theorem (e.g. Theorem 6.9 of \cite{stuart2010inverse})
guarantees that $\pi_0(\exp(3\epsilon \|x\|_{\beta/2}^2))<\infty$
for some $\epsilon>0$, and allows us to conclude with the second line.

\end{proof}

}

\section{Additional Numerical Results}
\label{app:fig}

\subsection{1D Toy Example}
\label{app:1dtoy}

We consider a 1D Toy Example first, whose likelihood is analytically tractable. This example is taken from \cite{rmlsmc}. Note that the multi-index methods are the same as multilevel methods in 1D. \blu{Considering the PDE \eqref{eq:elliptic}-\eqref{eq:boundary} with $D = 1$,} let $\Omega = [0,1]$, $a=1$, and the forcing term be $\mathsf{f} = x$, where $x$ is a random input with a uniform prior such that $x \sim U[-1,1]$. 
This differential equation can be solved analytically as $u(x) = -0.5x(z^2-z).$ 
Assume the observation operator as \eqref{eq:obop} and the observation taking the form as \eqref{eq:ob}. 
The pointwise observations are well-defined in 1D with $x \in L^2(\Omega)$. We take the observations at ten points in the interval (0,1) with a step size 1/10. Let $\Xi = 0.2$. Observations are generated by {$y = -0.5x^{*}(z^2-z) + \nu$, where $y = [y_1,...,y_{10}]$, $z = [z_1,...,z_{10}]$, $x^* = 0.2581$ drawn from $U[-1,1]$ and $\nu \sim N(0,0.2^2)$.}

\begin{figure}[H]
    \centering
    \includegraphics[width=.6\linewidth]{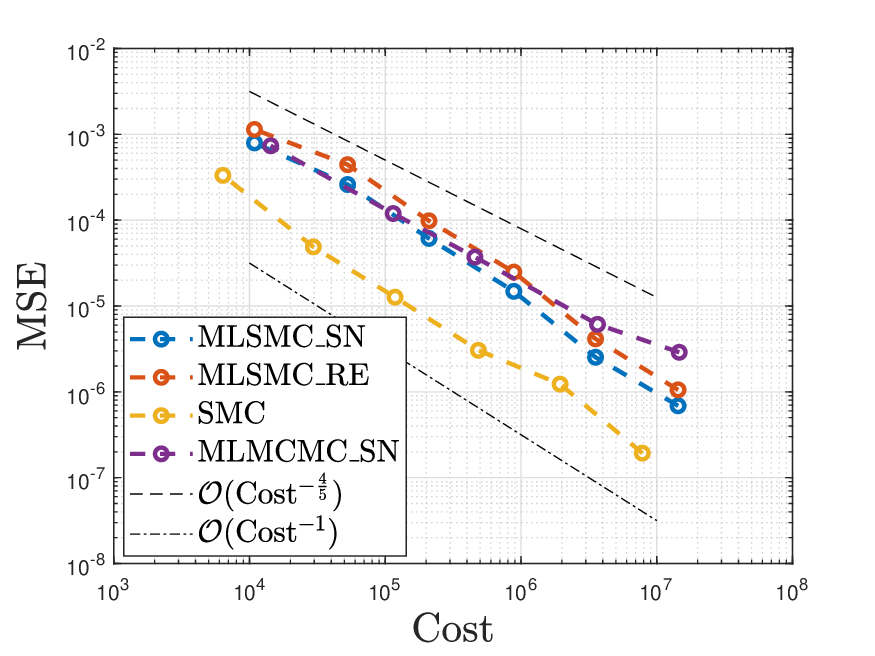}
    \caption{1D toy example, SMC, MLMCMC and MISMC rate of convergence by MSE vs Cost. Rates of regression: (1) MLSMC\_SN: $-1.011$  (2) MLSMC\_RE: $-1.005$  (3) SMC: $-0.753$ (4) MLMCMC\_SN: $-1.005$.}
    \label{fig:1DToyRate}
\end{figure}

For this example, the quantity of interest used is $x^2$. By applying the FEM and discretising the differential equation with the step size {$h_l = 2^{-l-1}$}, we have $s = 2$, $\beta = 4$. This is shown in Figure \ref{fig:1DToyMrMean} and \ref{fig:1DToyMrVar}. The value of $\gamma$ is 1 because we use a linear nodal basis function for FEM and tridiagonal solver. \blu{The algorithm is applied with Metropolis-Hastings method and a fixed tempering schedule for all $\alpha$, where $J = 3$.}
The MSE shown in Figure \ref{fig:1DToyRate} is calculated with 100 realisations, where the reference solution can be worked out as in \cite{rmlsmc}. 
The total computational cost is of $\cO(\sum_{l=0}^{L}N_lC_l)$. 

For comparison, single-level SMC, MLMCMC and MLSMC with the self-normalised increment estimator are applied in this example. 
It is difficult to observe the approximate rates from the plot directly, so we fit the rates and demonstrate those in the caption. 
The rate of convergence of single-level SMC is close to -4/5. 
The rate of convergence of MLMCMC with the self-normalised increment estimator, 
MLSMC with the self-normalised increment estimator and 
our MLSMC with ratio estimator are all approximately -1, 
which is the canonical complexity and better in terms of rate of convergence than the single-level methods as expected. 
The difference of performance between MLMCMC and MLSMC 
with either of the two estimators is only up to a constant. 
{MLMCMC has a smaller constant here, presumably as a consequence of 
the simplicity of the problem and the tuning of MLSMC.}
Our MLSMC with the ratio estimator appears to have a slightly larger constant,
while the theoretical results remain its advantage.

\begin{figure}[H]
    \centering
    \subfloat[Convergence of $B_l$]
    {\includegraphics[width=.45\linewidth]{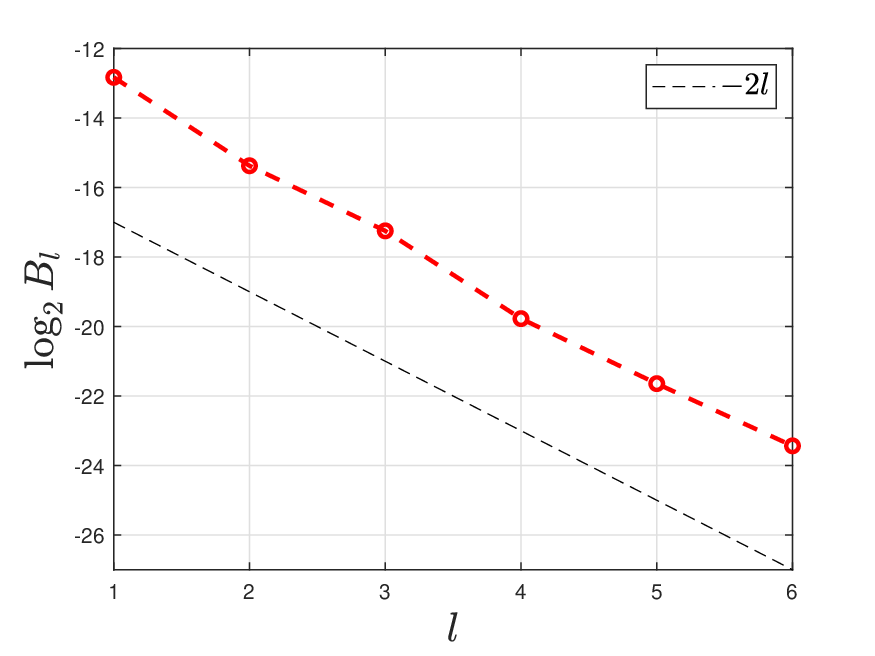}
    \label{fig:1DToyMrMean}}
    \subfloat[Convergence of $V_l$]
    {\includegraphics[width=.45\linewidth]{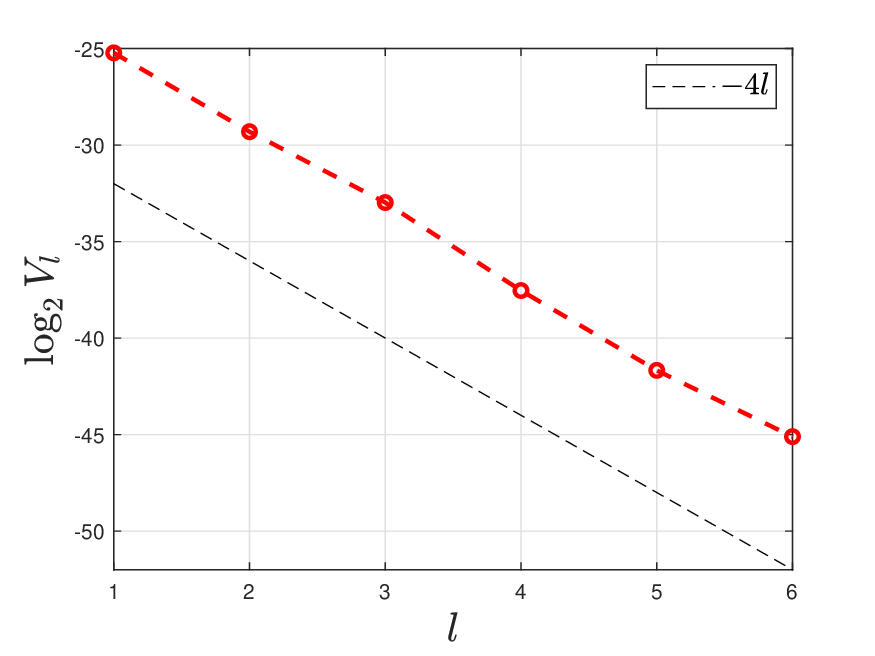}
    \label{fig:1DToyMrVar}}
    \caption{1D toy example convergence rates. 
    $B_l$ is computed with 100 realisations and 
    shown in panel \ref{fig:1DToyMrMean} 
     along with a line corresponding to $s = 2$. 
    $V_l$ is computed with 100 realisations and 
    shown in panel \ref{fig:1DToyMrMean} 
    along with a line corresponding to $\beta = 4$. }
    \label{fig:1DToyMr}
\end{figure}

\subsection{2D Elliptic PDE with random diffusion coefficient}
\begin{figure}[H]
    \centering
    \includegraphics[width=.95\linewidth]{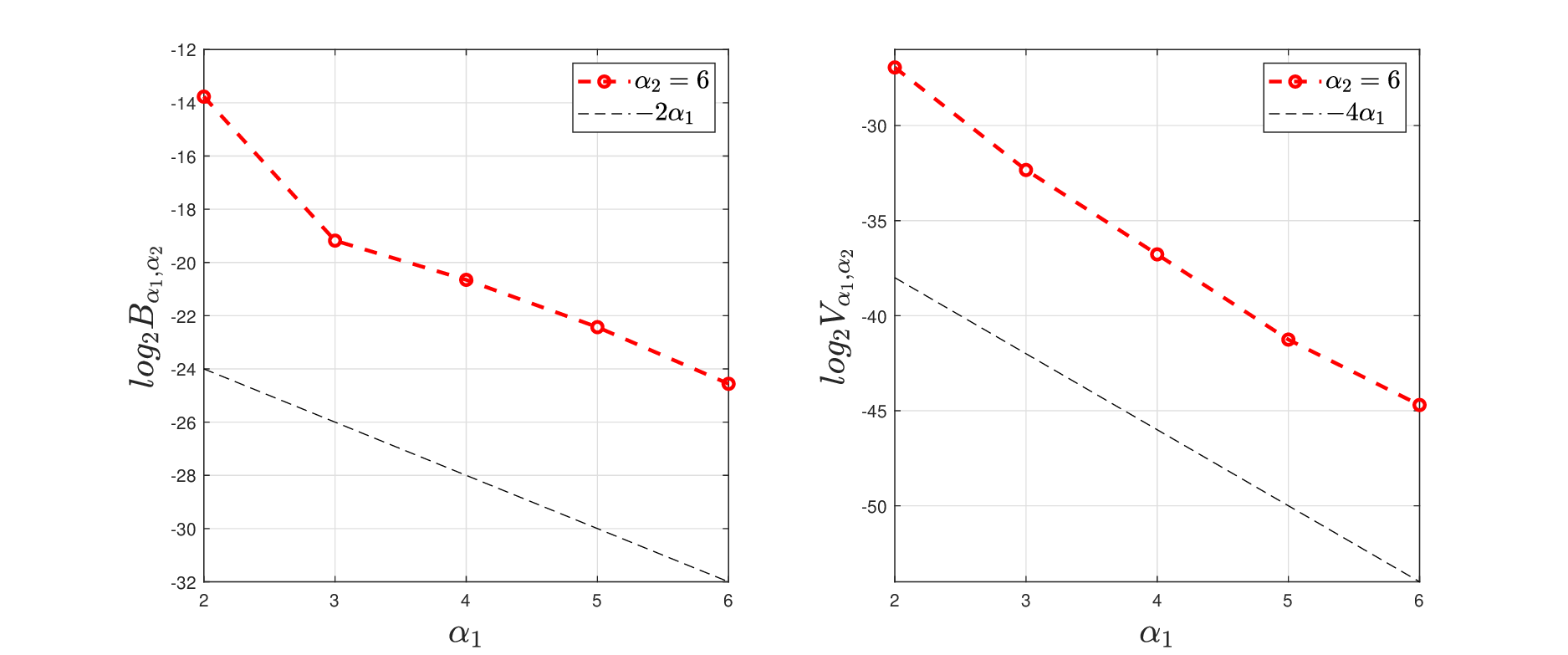}
    \caption{2D Elliptic PDE with random diffusion coefficient. 
    Verification of MISMC rates as in 
    Assumption \ref{ass:rate} for $\alpha_1$ given $\alpha_2 = 7$,
    computed with 20 realisations and 1000 samples for each realisation. 
    Left: $s_1$. Right: $\beta_1$. The same result holds for an $\alpha_1=7$ (not shown).}
    \label{fig:2DnlinMISMC1}
\end{figure}


\begin{figure}[H]
    \centering
    \includegraphics[width=.95\linewidth]{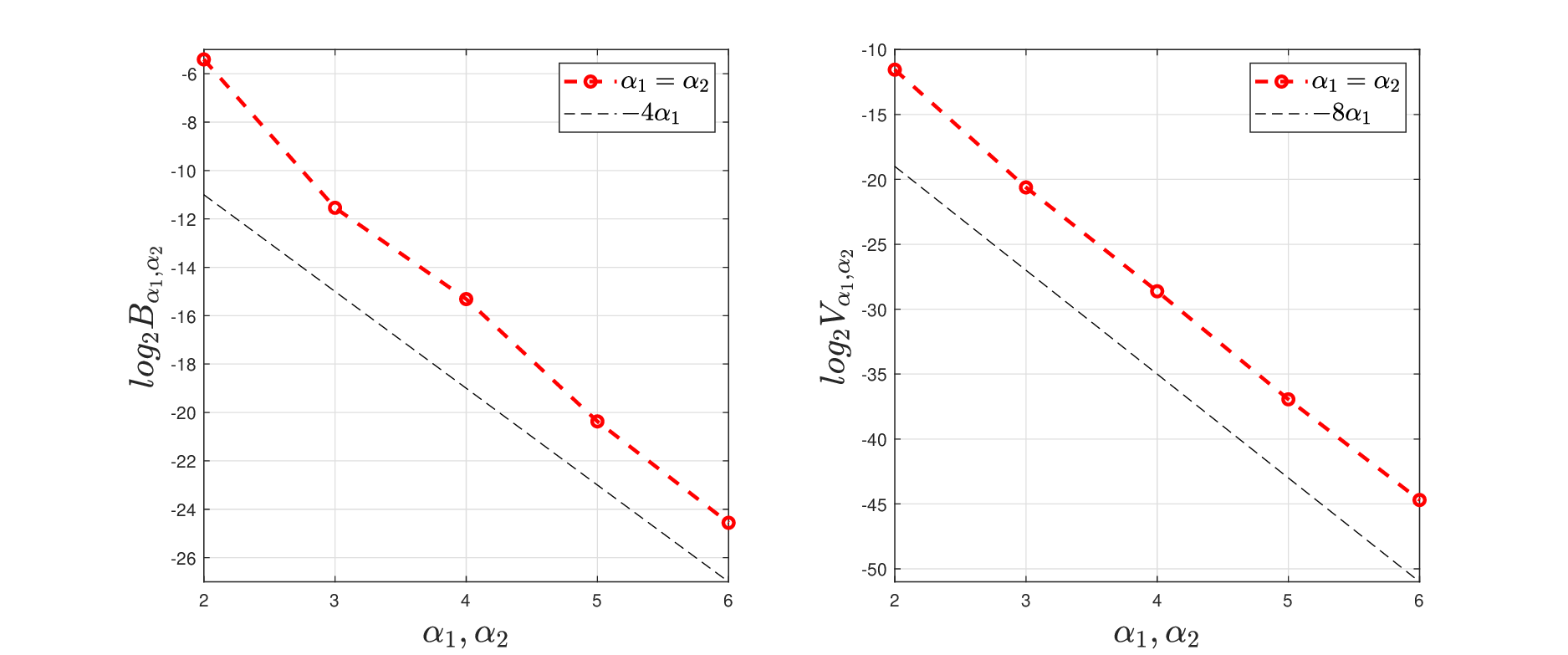}
    \caption{2D Elliptic PDE with random diffusion coefficient. 
    Verification of MISMC rates as in Assumption \ref{ass:rate} for $\alpha_1=\alpha_2$,
    computed with 20 realisations and 1000 samples for each realisation. 
    Left: $s_1+s_2$. Right: $\beta_1+\beta_2$.}
    \label{fig:2DnlinMISMC12}
\end{figure}

\begin{figure}[H]
    \centering
    \includegraphics[width=.95\linewidth]{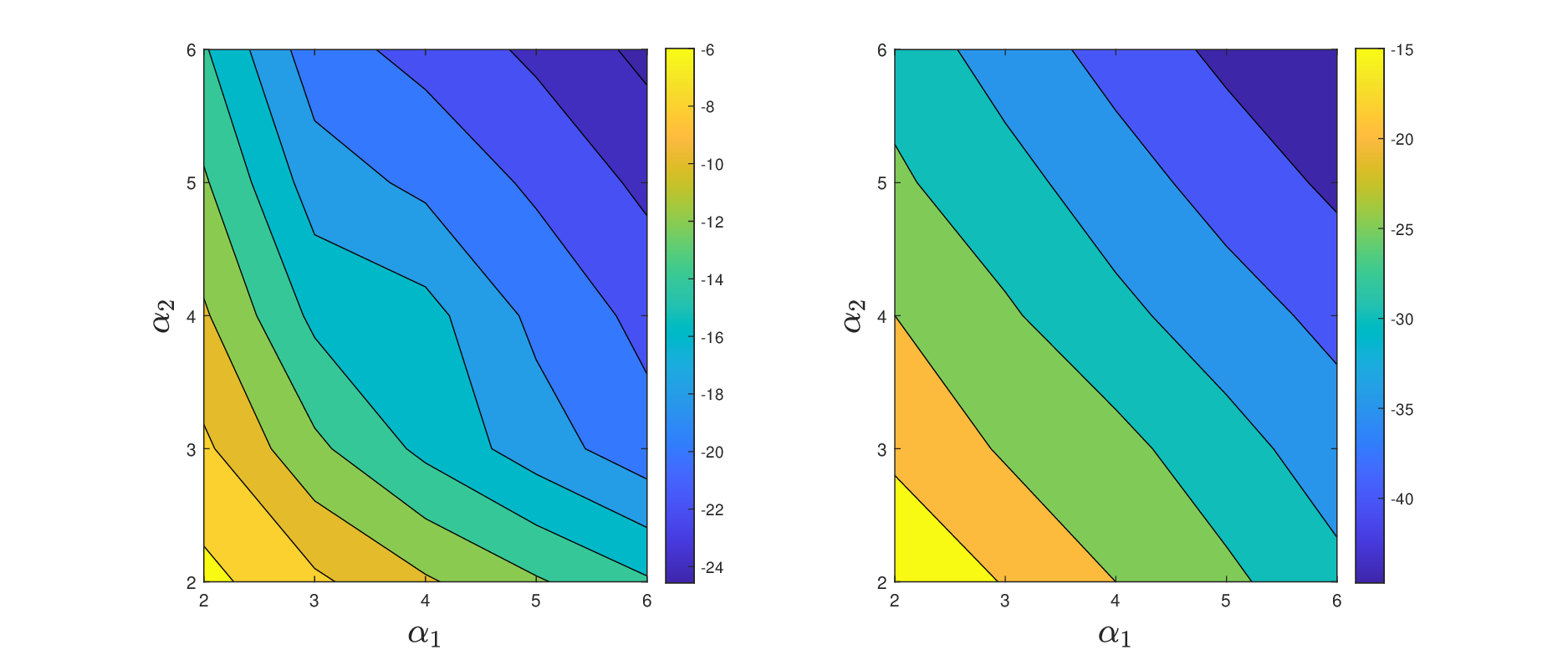}
    \caption{2D Elliptic PDE with random diffusion coefficient. 
    Verification of MISMC rates as in Assumption \ref{ass:rate},
    computed with 20 realisations and 1000 samples for each realisation. 
	Left: $s$. Right: $\beta$.}
    \label{fig:2DnlinContour}
\end{figure}

\begin{figure}[H]
    \centering
    \includegraphics[width=.95\linewidth]{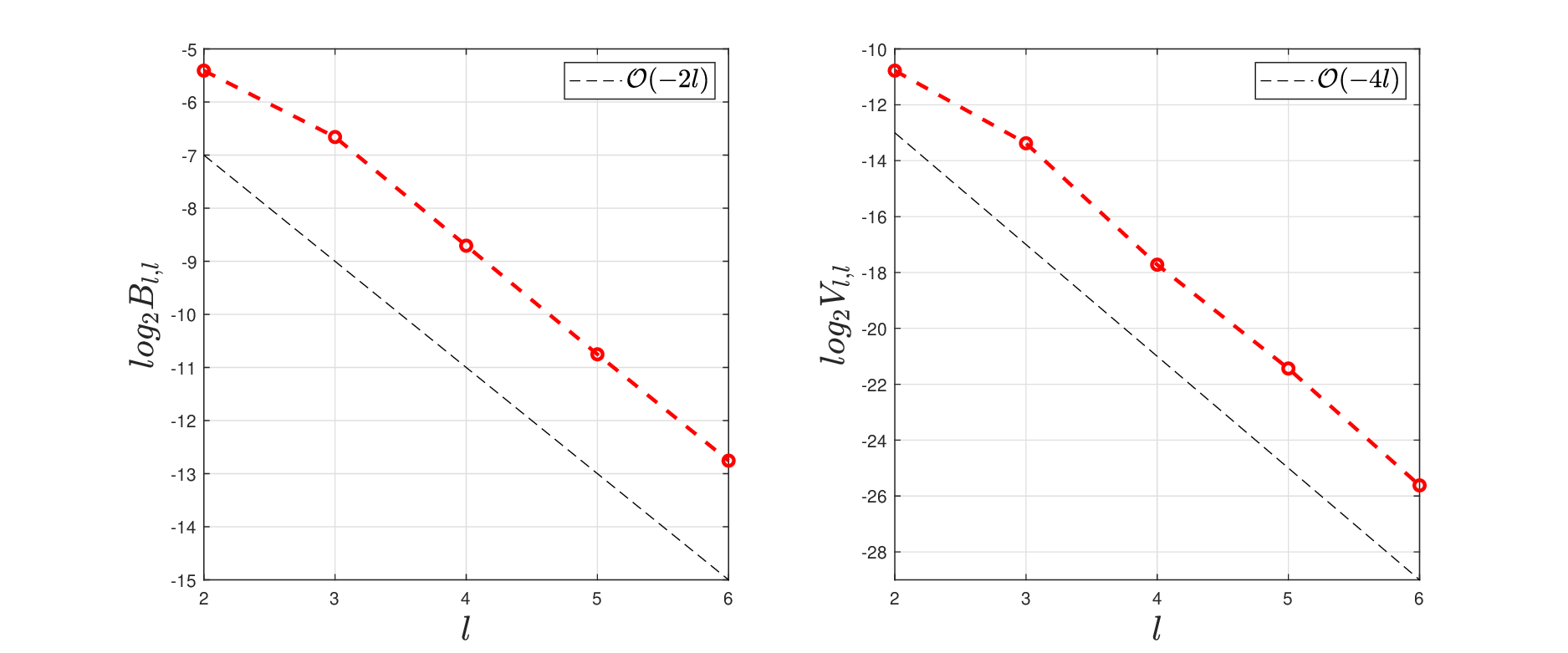}
    \caption{2D Elliptic PDE with random diffusion coefficient. 
    Verification of increment rates associated to MLSMC. 
    Computed with 20 realisations and 2000 samples for each realisation. Left: $s$. Right: $\beta$.}
    \label{fig:2DnlinMLSMC}
\end{figure}


\subsection{LGC}
\begin{figure}[H]
    \centering
    \includegraphics[width=.95\linewidth]{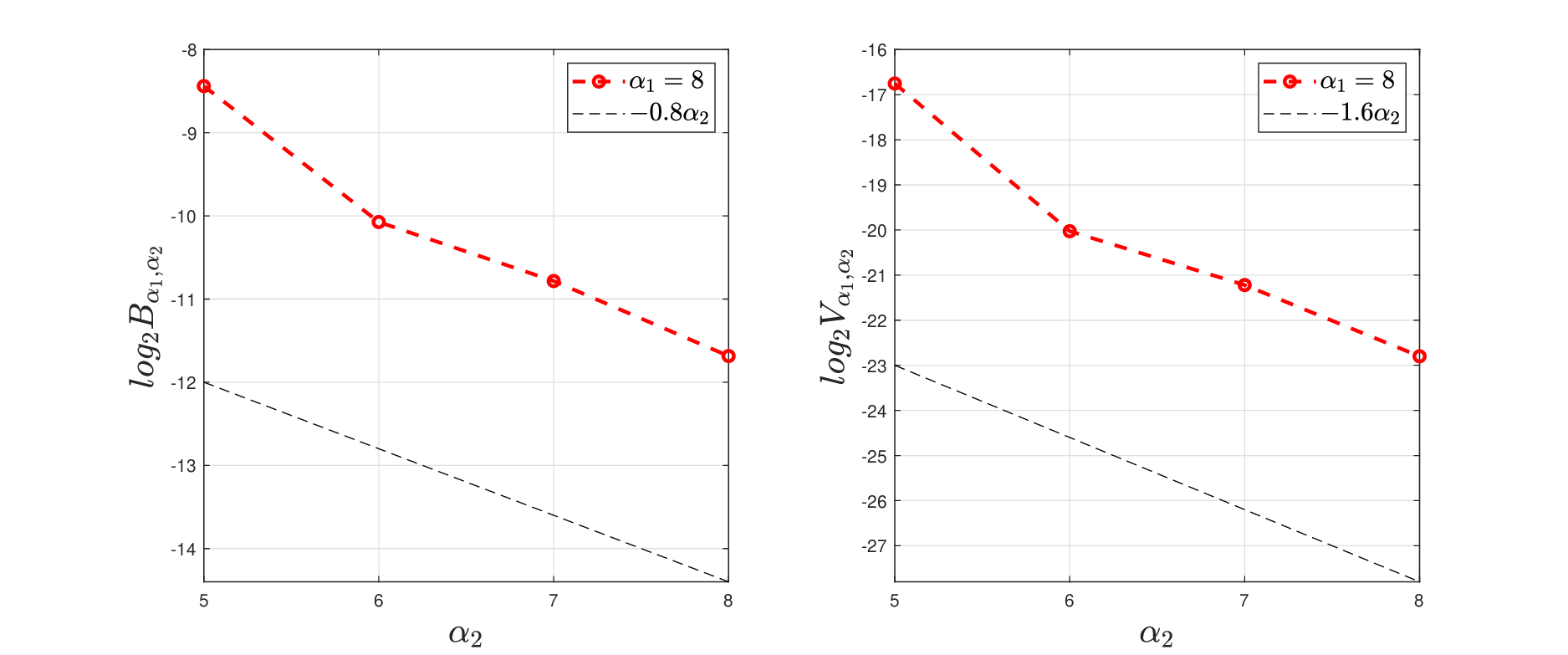}
    \caption{LGC model. 
    Verification of mixed rates associated to Assumption \ref{ass:rate} for MISMC, over
    $\alpha_2$ given $\alpha_1 = 8$, 
    computed with 20 realisations and 1000 samples for each realisation. 
    Left: $s_2$. Right: $\beta_2$. 
    The same result holds over $\alpha_1$ for $\alpha_2=8$ (not shown).}
    \label{fig:LGCMISMC2}
\end{figure}

\begin{figure}[H]
    \centering
    \includegraphics[width=.95\linewidth]{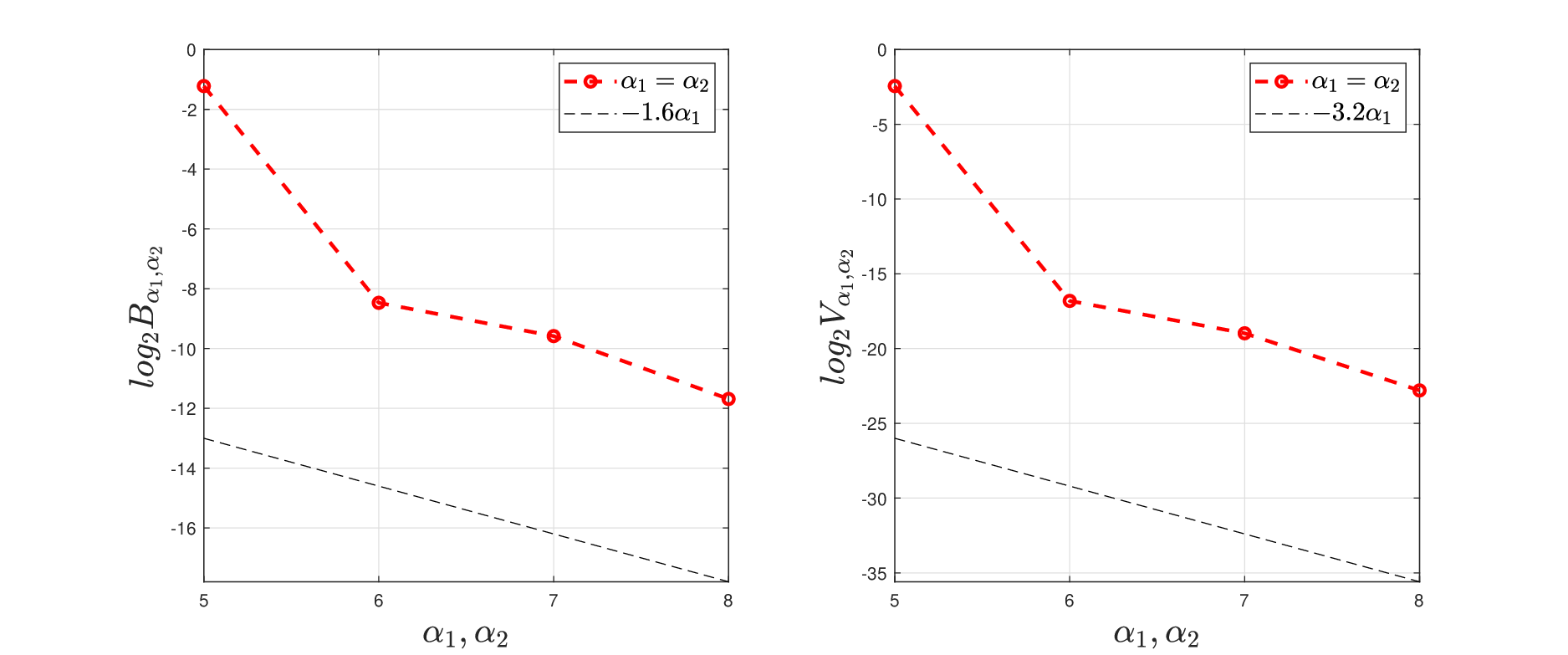}
    \caption{LGC model. 
    Verification of mixed rates associated to Assumption \ref{ass:rate} for MISMC, over
    $\alpha_2 = \alpha_1$, 
    computed with 20 realisations and 1000 samples for each realisation. 
    Left: $2s$. Right: $2\beta$.}
    \label{fig:LGCMISMC12}
\end{figure}

\begin{figure}[H]
    \centering
    \includegraphics[width=.95\linewidth]{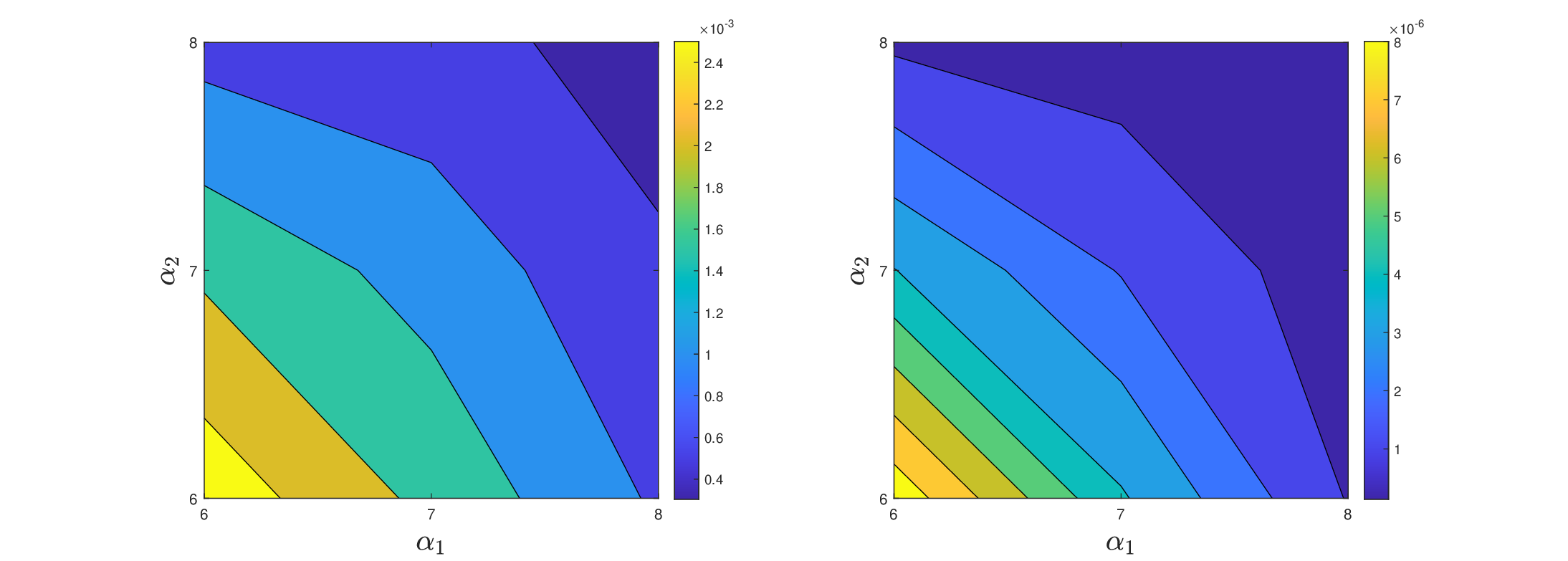}
    \caption{LGC model. 
    Verification of mixed rates associated to Assumption \ref{ass:rate} for MISMC, 
    over $\alpha_2$ and $\alpha_1$, 
    computed with 20 realisations and 1000 samples for each realisation. 
    Left: $2s$. Right: $2\beta$.}
    \label{fig:LGCContour}
\end{figure}

\begin{figure}[H]
    \centering
    \includegraphics[width=.95\linewidth]{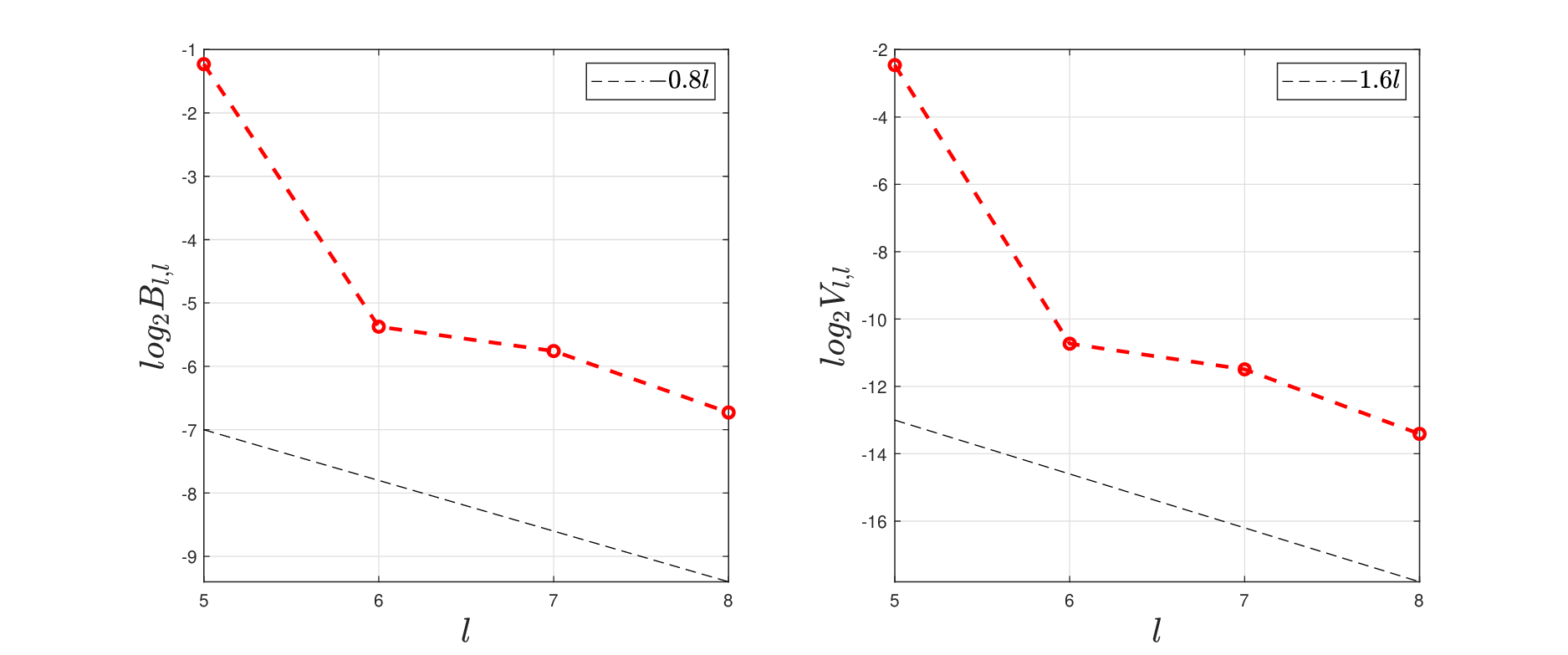}
    \caption{LGC model. Verification of increment rates for MLSMC, 
    computed with 20 realisations and 1000 samples for each realisation. 
    Left: $s$. Right: $\beta$.}
    \label{fig:LGCMLSMC}
\end{figure}

\subsection{LGP}
\begin{figure}[H]
    \centering
    \includegraphics[width=.95\linewidth]{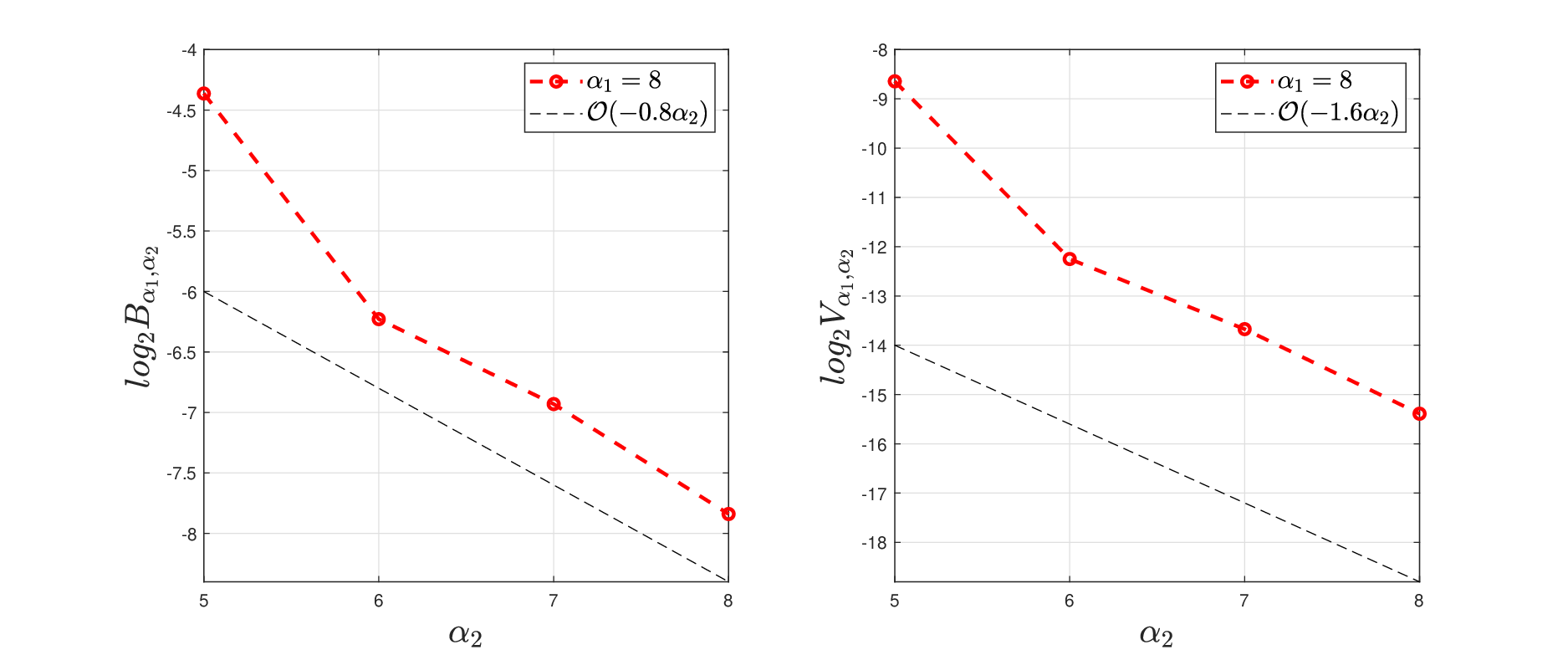}
    \caption{LGP model. 
    Verification of mixed rates associated to Assumption \ref{ass:rate} for MISMC, over
    $\alpha_2$ given $\alpha_1 = 8$, 
    computed with 20 realisations and 1000 samples for each realisation. 
    Left: $s_2$. Right: $\beta_2$. 
    The same result holds over $\alpha_1$ for $\alpha_2=8$ (not shown).}
    \label{fig:LGPMISMC2}
\end{figure}

\begin{figure}[H]
    \centering
    \includegraphics[width=.95\linewidth]{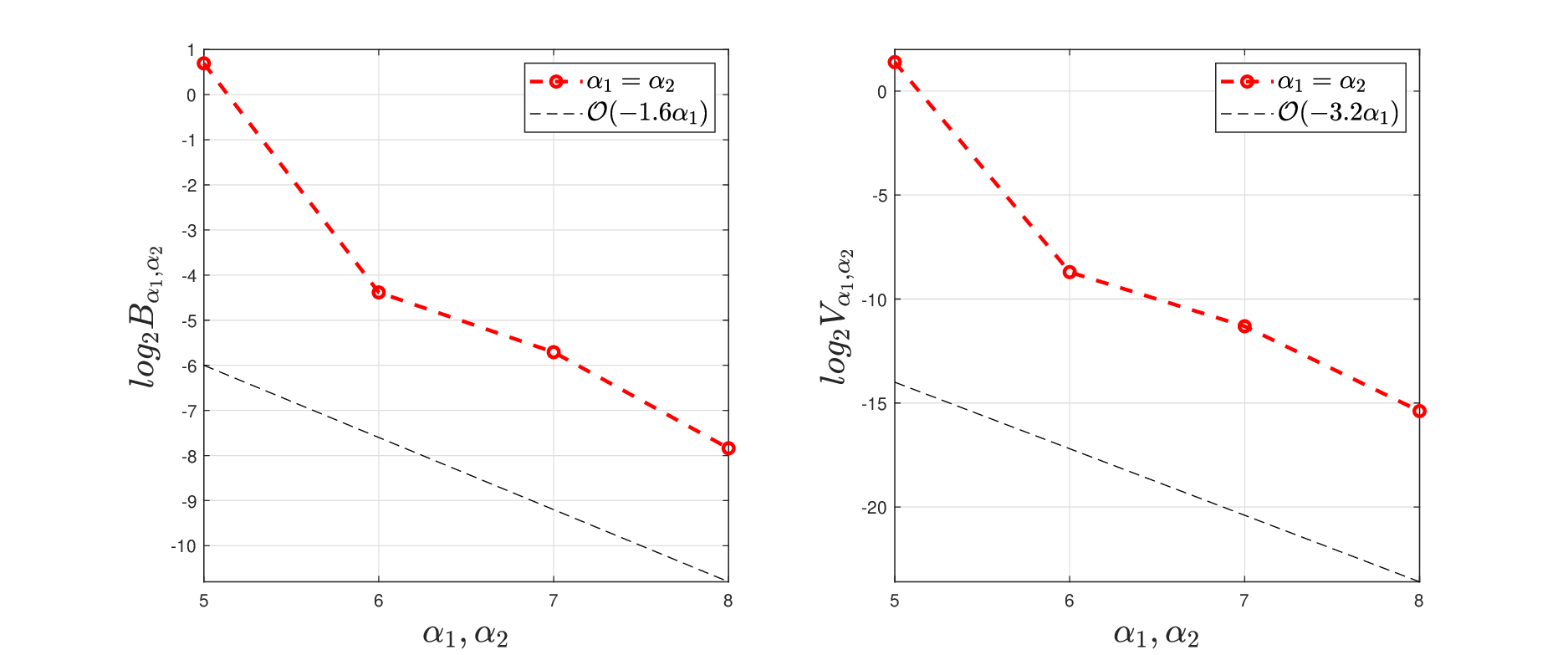}
    \caption{LGP model. 
    Verification of mixed rates associated to Assumption \ref{ass:rate} for MISMC, over
    $\alpha_2 = \alpha_1$,
    computed with 20 realisations and 1000 samples for each realisation. 
    Left: $2s$. Right: $2\beta$.}
    \label{fig:LGPMISMC12}
\end{figure}

\begin{figure}[H]
    \centering
    \includegraphics[width=.95\linewidth]{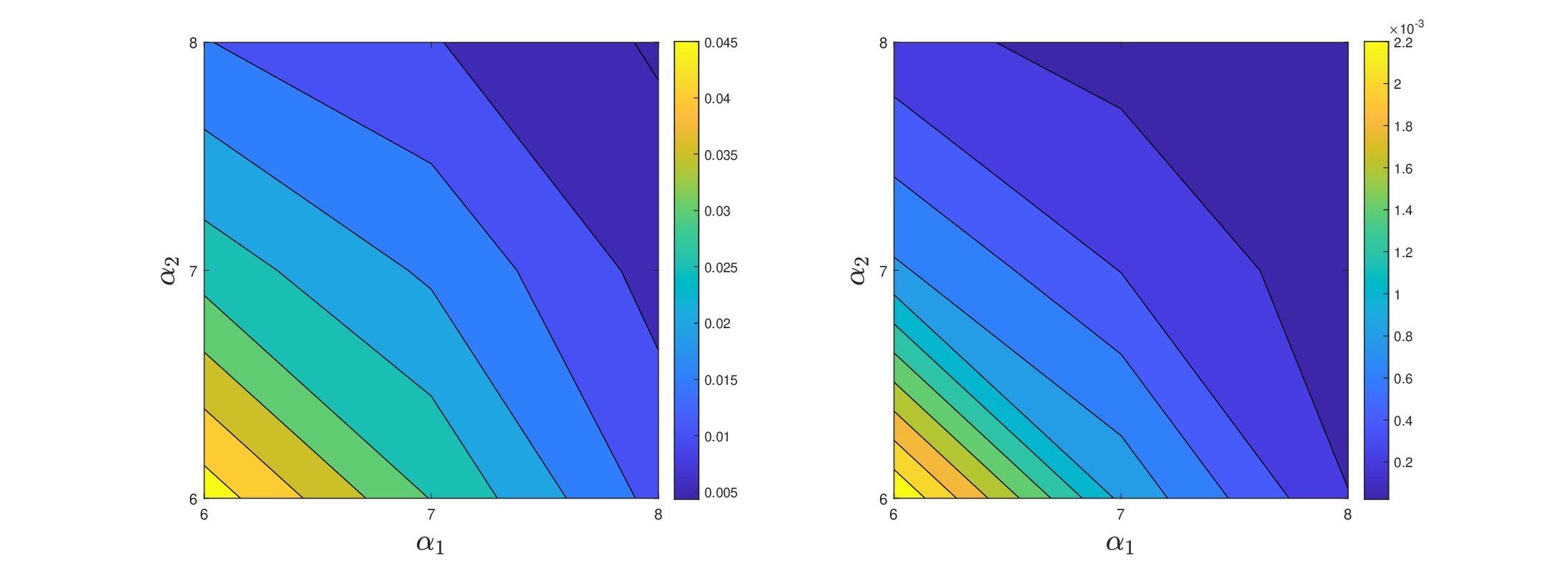}
    \caption{LGP model. 
    Verification of mixed rates associated to Assumption \ref{ass:rate} for MISMC, 
    over $\alpha_2$ and $\alpha_1$, 
    computed with 20 realisations and 1000 samples for each realisation. 
    Left: $2s$. Right: $2\beta$.}
    \label{fig:LGPContour}
\end{figure}

\begin{figure}[H]
    \centering
    \includegraphics[width=.95\linewidth]{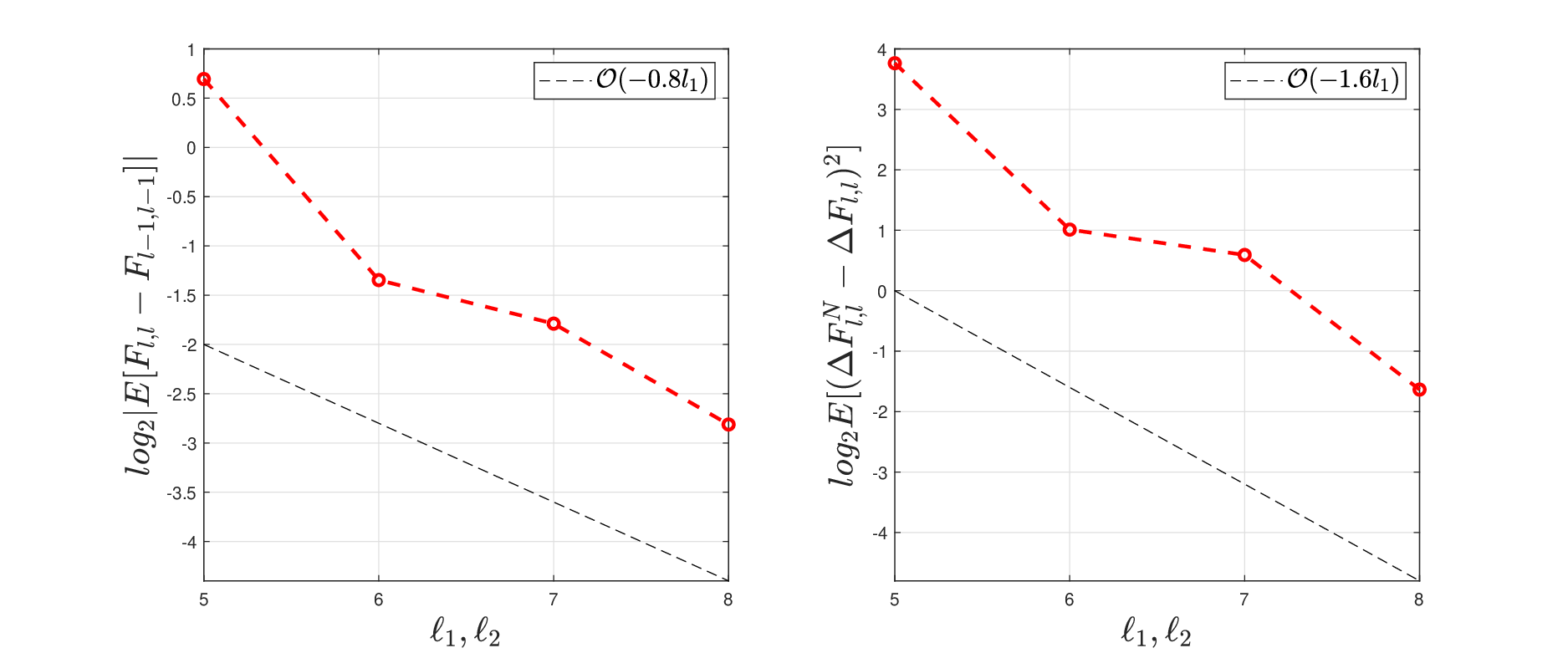}
    \caption{LGP model. Verification of increment rates for MLSMC, 
    computed with 20 realisations and 1000 samples for each realisation. 
    Left: $s$. Right: $\beta$.}
    \label{fig:LGPMLSMC}
\end{figure}

\let\itshape\upshape
\bibliography{refs}
\bibliographystyle{abbrv}

\end{document}